\documentclass[a4paper, leqno, 10pt]{amsart}

\usepackage[utf8]{inputenc}
\usepackage[a4paper, left=90pt, right=90pt, bottom=100pt, top=100pt]{geometry}

\usepackage[american]{babel}
\usepackage{amsthm}
\usepackage{amssymb}
\usepackage{amsmath}
\usepackage[shortlabels]{enumitem}
\usepackage{mathrsfs}
\usepackage{url}

\usepackage{tikz} 
\usetikzlibrary{matrix,arrows,calc,decorations.pathreplacing,fit}
\usepackage{tikz-cd}

\numberwithin{equation}{section}

\newtheorem{theorem}{Theorem}[section]
\newtheorem{lemma}[theorem]{Lemma}
\newtheorem{corollary}[theorem]{Corollary}
\newtheorem{proposition}[theorem]{Proposition}

\theoremstyle{definition}
\newtheorem{definition}[theorem]{Definition}
\newtheorem{notation}[theorem]{Notation}
\theoremstyle{remark}
\newtheorem{remark}[theorem]{Remark}
\newtheorem{example}[theorem]{Example}
\newtheorem{claim}{Claim}

\DeclareMathOperator{\Ad}{Ad}

\DeclareMathOperator{\Aut}{Aut}

\DeclareMathOperator{\codim}{codim}
\DeclareMathOperator{\cha}{char}
\DeclareMathOperator{\Def}{Def}
\DeclareMathOperator{\defect}{def}
\DeclareMathOperator{\diag}{diag}
\DeclareMathOperator{\End}{End}

\DeclareMathOperator{\id}{id}
\DeclareMathOperator{\image}{im}

\DeclareMathOperator{\inv}{inv}
\DeclareMathOperator{\Gal}{Gal}
\DeclareMathOperator{\GL}{GL}
\DeclareMathOperator{\GSp}{GSp}
\DeclareMathOperator{\GU}{GU}
\DeclareMathOperator{\height}{ht}
\DeclareMathOperator{\Hom}{Hom}
\DeclareMathOperator{\Lie}{Lie}
\DeclareMathOperator{\Mor}{Mor}
\DeclareMathOperator{\Norm}{Norm}
\DeclareMathOperator{\M}{M}

\DeclareMathOperator{\PGL}{PGL}

\DeclareMathOperator{\spa}{span}
\DeclareMathOperator{\Spec}{Spec}
\DeclareMathOperator{\Spf}{Spf}

\DeclareMathOperator{\Res}{Res}

\DeclareMathOperator{\rank}{rk}
\DeclareMathOperator{\tr}{tr}
\DeclareMathOperator{\val}{val}
\DeclareMathOperator{\vol}{vol}

\def \Homfunc {\underline{\Hom}}
\def \Specfunc {\underline{\smash{\Spec}}}

\def \ad {\mathrm{ad}}

\def \dom {\mathrm{dom}}
\def \nr {\mathrm{nr}}

\def \mmax {\mathrm{max}}

\newcommand{\pot}[1]{ [\hspace{-0,17em}[ {#1} ]\hspace{-0,17em}] }
\newcommand{\rpot}[1]{ (\hspace{-0,23em}( {#1} )\hspace{-0,23em}) }

\def \AA {\mathbb{A}}

\def \CC {\mathbb{C}}

\def \FF {\mathbb{F}}
\def \GG {\mathbb{G}}

\def \NN {\mathbb{N}}

\def \QQ {\mathbb{Q}}
\def \RR {\mathbb{R}}

\def \XX {\mathbb{X}}

\def \ZZ {\mathbb{Z}}

\def \afr {\mathfrak{a}}

\def \gfr {\mathfrak{g}}

\def \nfr {\mathfrak{n}}

\def \Ascr {\mathscr{A}}
\def \Bscr {\mathscr{B}}

\def \Dscr {\mathscr{D}}
\def \Escr {\mathscr{E}}
\def \Fscr {\mathscr{F}}

\def \Iscr {\mathscr{I}}

\def \Mscr {\mathscr{M}}
\def \Nscr {\mathscr{N}}
\def \Oscr {\mathscr{O}}

\def \Rscr {\mathscr{R}}
\def \Sscr {\mathscr{S}}
\def \Tscr {\mathscr{T}}

\def \Vscr {\mathscr{V}}

\def \Yscr {\mathscr{Y}}

\def \Nbar {\bar{N}}

\def \Xbar {\bar{X}}
\def \Ybar {\bar{Y}}

\def \abar {\bar{a}}

\def \cbar {\bar{c}}

\def \hbar {\bar{h}}

\def \kbar {\bar{k}}

\def \sbar {\bar{s}}

\def \zbar {\bar{z}}

\def \Wtilde {\tilde{W}}

\def \wtilde {\tilde{w}}

\def \Gsf {\mathbf{G}}

\def \nbf {\mathbf{n}}

\def \FFbar {\overline{\mathbb{F}}}

\def \QQbar {\overline{\mathbb{Q}}}

\def \XXbar {\overline{\mathbb{X}}}

\def \BT    {{Barsotti-Tate group}}
\def \BTD   {{Barsotti-Tate group with $\Dscr$-structure}}
\def \BTDs  {{Barsotti-Tate groups with $\Dscr$-structure}}

\def \Lra   {\Leftrightarrow}
\def \lto   {\longrightarrow}
\def \mono  {\hookrightarrow}

\def \epi   {\twoheadrightarrow}

\def \isom  {\stackrel{\sim}{\rightarrow}}

\def \pair {\langle\,\, , \, \rangle}

\newcommand{\bigslant}[2]{{\raisebox{.2em}{$#1$}\hspace{-.3em}\left/ \hspace{-.2em}\raisebox{-.2em}{$#2$}\right.}}

\def \Bsf  {\mathsf{B}}
\def \Osf  {\mathsf{O}}
\def \Vsf  {\mathsf{V}}
\def \Esf  {\mathsf{E}}
\def \Gsf  {\mathsf{G}}

\def \Aund {\underline{A}}
\def \Mund {\underline{M}}
\def \Xund {\underline{X}}
\def \Yund {\underline{Y}}
\def \XXund {{\underline{\mathbb{X}}}}
\def \omegaund {{\underline{\omega}}}
\def \mubar {\bar{\mu}}

\def \univ {{\rm univ}}
\def \cont {{\rm cont}}

\DeclareMathOperator{\nd}{ndim}
\DeclareMathOperator{\relht}{relht}
\DeclareMathOperator{\NP}{NP}
\DeclareMathOperator{\tc}{tc}
\DeclareMathOperator{\length}{length}

\newcommand{\lengthG}[2]{\length ([ {#1}, {#2}])}

\newcommand{\textsubscript}[1]{$_{\text{#1}}$}

\newenvironment{subenv} {\begin{enumerate}[(1), labelsep = 0.5em, leftmargin = 1.5em]} {\end{enumerate}}
\newenvironment{assertionlist} {\begin{enumerate}[(a), labelsep = 5pt, labelwidth = 13pt, leftmargin = 25pt, topsep = -3pt, parsep = 3pt ]} {\end{enumerate}}
\newenvironment{itemlist} {\begin{itemize}[itemsep=1ex, labelsep = 0.5em, leftmargin = 1.5 em]} {\end{itemize}}

\newenvironment{simplelist}{%
  \begin{list}{}
     {\setlength{\leftmargin}{10pt}
      \setlength{\rightmargin}{0pt}
      \setlength{\itemindent}{0pt}
      \setlength{\labelsep}{5pt}
      \setlength{\listparindent}{\parindent}
      \setlength{\parsep}{0pt}
      \setlength{\itemsep}{0pt}
      \setlength{\topsep}{0pt}}}
  {\end{list}}

\def \BTEL {{Barsotti-Tate group with EL structure}}
\def \BTPEL {{Barsotti-Tate group with PEL structure}}
\def \BTPELs {{Barsotti-Tate groups with PEL structure}}
\def \BTpEL {{Barsotti-Tate group with (P)EL structure}}
\def \BTpELs {{Barsotti-Tate groups with (P)EL structure}}

\title[The Newton stratification on Shimura varieties of PEL type]{The geometry of Newton strata in the reduction modulo $p$ of Shimura varieties of PEL type}
\author[P. Hamacher]{by Paul Hamacher}

\begin{document}

 \maketitle
 
 \begin{abstract}
  In this paper we study the Newton stratification on the reduction of Shimura varieties of PEL type with hyperspecial level structure. Our main result is a formula for the dimension of Newton strata and the description of their closure, where the dimension formula was conjectured by Chai. As a key ingredient of its proof we calculate the dimension of some Rapoport-Zink spaces. Our result yields a dimension formula, which was conjectured by Rapoport (up to a minor correction).

  As an interesting application to deformation theory, we determine the dimension and closure  of Newton strata on the algebraisation of the deformation space of a \BTpEL . Our result on the closure of a Newton stratum generalises conjectures of Grothendieck and Koblitz.
 \end{abstract}

 \section{Introduction} \label{sect introduction}
 
 We fix a prime $p$ and denote by $\sigma$ the Frobenius automorphism over $\FF_p$ or $\QQ_p$ (where the latter is considered in an unramified field extension). We denote by $\breve\QQ_p := \widehat\QQ_p^{\nr}$ the completion of the maximal unramified extension of $\QQ_p$. Let $\Dscr$ be a PEL-Shimura datum unramified at $p$ as in \cite{kottwitz92}~ch.~5 such that the associated linear algebraic group $\Gsf$ is connected. We denote by $\Ascr_0$ the reduction modulo $p$ of the associated moduli space defined by Kottwitz in \cite{kottwitz92}. The points of $\Ascr_0$ correspond to abelian varieties equipped with polarisation, endomorphisms and level structure. The Newton stratification is the stratification corresponding to the isogeny class of \BT s (with endomorphisms and polarisation) $\Aund[p^\infty] = (A[p^\infty],\lambda_{|A[p^\infty]},\iota_{|A[p^\infty]})$ of points $\Aund = (A,\lambda,\iota,\eta)$ of $\Ascr_0$.

 We call \BT s with additional structure induced by $\Dscr$ ``{\BTDs}'' (for a more precise definition see section~\ref{ss preliminaries}). By Dieudonn\'e theory their isogeny classes correspond to a certain finite subset $B(\Gsf_{\QQ_p},\mu)$ of the set $B(\Gsf_{\QQ_p})$ of $\sigma$-conjugacy classes in $G(\breve\QQ_p)$. For $b \in B(\Gsf_{\QQ_p},\mu)$ denote by $\Ascr_0^{b}$ the associated Newton stratum of  $\Ascr_0$. Viehmann and Wedhorn have shown in \cite{VW13}, Thm.~11.1 that $\Ascr_0^{b}$ is always non-empty.

 The set $B(\Gsf_{\QQ_p})$ is equipped with a partial order, which is given in group theoretic terms. In the ``classical'' case of \BT s without additional structure, i.e.\ $\Gsf_{\QQ_p} = \GL_n$, we have the following description of this order. By a result of Dieudonn\'e, the set $B(\Gsf_{\QQ_p})$ equals the set of (concave) Newton polygons over $[0,n]$. Then $b' \leq b$ iff the polygons have the same endpoint and $b$ lies above $b'$ (for more details see section~\ref{ss sigma conjugacy}). It is known that the closure of $\Ascr_0^{b}$ in $\Ascr_0$ is contained in $\Ascr_0^{\leq b} := \bigcup_{b' \leq b} \Ascr_0^{b'}$ by a theorem of Rapoport and Richartz (\cite{RR96}~Thm.~3.6). Their result generalises of Grothendieck's specialisation theorem which states that (concave) Newton polygons only ``go down'' under specialisation.

 \subsection{The main results}

 The primary goal of this paper is the following theorem.

 \begin{theorem} \label{thm dimension shimura}
  \begin{subenv}
   \item $\Ascr_0^{\leq b}$ is equidimensional of dimension
   \begin{equation} \label{term dimension shimura}
    \langle \rho, \mu+\nu(b) \rangle - \frac{1}{2} \defect (b)
   \end{equation}
   where $\rho$ denotes the half-sum of (absolute) positive roots of $\Gsf$, $\mu$ is the cocharacter induced by $\Dscr$ and $\nu(b)$ and $\defect (b)$ denote the Newton point resp.\ the defect of $b$ (cf.~section~\ref{ss sigma conjugacy}).
   \item $\Ascr_0^{\leq b}$ is the closure of $\Ascr_0^{b}$ in $\Ascr_0$.
  \end{subenv}
 \end{theorem}
 In the case of the Siegel moduli variety this theorem was proven by Oort (cf.~\cite{oort00}~Cor.~3.5).

 In the general PEL-case the dimension formula (\ref{term dimension shimura}) proves a conjecture of Chai. In \cite{chai00}~Question~7.6 he conjectured a formula for the codimension of $\Ascr_0^b$ using his notion of chains of Newton points. We prove the equivalence of his conjecture and our dimension formula in section~\ref{ss chains}.

 The most important ingredient of the proof of the above theorem is the following result, which is also interesting in its own right.
 \begin{theorem} \label{thm dimension RZ-space}
  Let $\Mscr_G(b,\mu)$ be the underlying reduced subscheme of the Rapoport-Zink space associated to an unramified Rapoport-Zink datum (cf.\ Def.~\ref{def RZ datum}).
  \begin{subenv}
   \item The dimension of $\Mscr_G(b,\mu)$ equals
   \begin{equation} \label{term dimension RZ-space}
    \langle \rho, \mu - \nu_G(b) \rangle - \frac{1}{2} \defect_G (b)
   \end{equation}
   \item If $b$ is superbasic then the connected components of $\Mscr_G (b,\mu)$ are projective.
  \end{subenv}
 \end{theorem}

 The dimension formula for $\Mscr_G(b,\mu)$ coincides with the formula conjectured by Rapoport in \cite{rapoport05}, p.296 up to a minor correction. See also Remark~\ref{rem dimension formula}.

 Theorem~\ref{thm dimension RZ-space} is already known in the case of moduli spaces of \BT s without endomorphism structure by results of Viehmann (\cite{viehmann08}, \cite{viehmann08b}). Moreover, the dimension formula (\ref{term dimension RZ-space}) is known to hold for some affine Deligne-Lusztig varieties, which are a function field analogue of Rapoport-Zink spaces. The dimension formula was proved for affine Deligne-Lusztig varieties in the affine Grassmannian of split groups in \cite{GHKR06} and \cite{viehmann06}, this proof was generalized to the case of unramified groups in \cite{hamacher}. 

 \subsection{Application to deformation theory} 
 Let $\Xund$ be a {\BTD} over a perfect field $k_0$ of characteristic $p$. We denote by $\Def(\Xund)$ the deformation functor of $\Xund$. It is known that $\Def(\Xund)$ is representable, we denote by $\Sscr_{\Xund}$ its algebraisation. By a result of Drinfeld there exists a (unique) algebraisation of the universal deformation of $\Xund$ to a {\BTD} over $\Sscr_{\Xund}$ (for more details see section~\ref{ss deformation}). This induces a Newton stratification on $\Sscr_{\Xund}$ for which we use the analogous notation as above. We derive the following theorem from Theorem~\ref{thm dimension shimura} by using a Serre-Tate argument (see section~\ref{ss comparision}).

 \begin{theorem} \label{thm dimension deformation}
   Denote by $b_0$ the isogeny class of $\Xund$ and let $b \in B(\Gsf_{\QQ_p},\mu)$ with $b \geq b_0$.
  \begin{subenv}
   \item $\Sscr_{\Xund}^{\leq b}$ is equidimensional of dimension
   \[
    \langle \rho_{\Gsf}, \mu+\nu(b) \rangle - \frac{1}{2} \defect (b).
   \]
   \item $\Sscr_{\Xund}^{\leq  b }$ is the closure of $\Sscr_{\Xund}^{b}$ in $\Sscr_{\Xund}$.
  \end{subenv}
 \end{theorem}
 In the case of \BT s without additional structure and for polarised \BT s (without endomorphism structure) this was also proven by Oort (\cite{oort00}~Thm.~3.2,~3.3).

 We give an intrinsic definition of the \BT s with additional structure which coincides with $\Dscr'$-structure for a suitable PEL-Shimura datum $\Dscr'$ (and thus Theorem~\ref{thm dimension deformation} applies). We show that one basically has to exclude case D in Kottwitz's notation and hence call these groups of type (AC) (for more details, see section~\ref{ss PEL vs D}).

 \subsection{Conjectures of Grothendieck and Koblitz}
 Let $\Xund_0$ und $\Xund_\eta$ be two {\BTEL} or PEL structure of type (AC). We say that  $\Xund_0$ is the specialisation of $\Xund_\eta$ if there exists an integral local scheme $S$ of characteristic $p$ and a {\BTpEL} $\Xund$ over $S$ which has generic fibre $\Xund_\eta$ and special fibre $\Xund_0$.

 Now assume that $\Xund_0$ is a specialisation of $\Xund_\eta$ and denote by $b_0$ and $b$ their respective isogeny classes. Then \cite{RR96}~Thm.~3.6 states that $b_0 \leq b$. In the case of \BT s without additional structure this is a result of Grothendieck known as Grothendieck's specialisation theorem. Grothendieck conjectured in a letter to Barsotti (see e.g.\ the appendix of \cite{grothendieck74}) that the converse of his specialisation theorem also holds true. He writes ``The necessary conditions (1) (2) that $G'$ is a specialisation of $G$ are also sufficient. In other words, taking the formal modular deformation in char.~$p$ (over a modular formal variety $S$ [...]) and the BT group $G$ over $S$ thus obtained, we want to know if for every sequence of rational numbers $(\lambda_i)_i$ which satisfies (1) and (2), these numbers occur as the sequence of slopes of a fibre of $G$ at some point $S$.'' Here he considers the isogeny class $b$ of a {\BT} via the family of the slopes of their Newton polygon and the conditions (1) and (2) reformulate to $b_0 \leq b$ where $b_0$ denotes the isogeny class of $G'$.

 The following generalisation follows from Theorem~\ref{thm dimension deformation}, as it is a reformulation of the non-emptiness of Newton strata. In particular, it was already shown in the case of \BT s without additional structure and for polarised \BT s by Oort (\cite{oort00}, Thm.~6.2, Thm.~6.3).

 \begin{proposition}
  Let $\Xund$ be a {\BTEL} or PEL structure of type (AC) and let $b_0$ denote its isogeny class. For any isogeny class $b$ with $b \geq b_0$ there exists a deformation of $\Xund$ which has generically isogeny class $b$.
 \end{proposition}

 Motivated by Grothendieck's conjecture, Koblitz conjectured in \cite{koblitz75}~p.~211 that ``all totally ordered sequences of Newton polygons can be realized by successive specialisations of principally polarised abelian varieties''. In other words, if $\Ascr_0$ is the Siegel moduli space and $b_1 > \ldots > b_h$ is a chain of isogeny classes (or equivalently a chain of symmetric Newton polygons with slopes between $0$ and $1$) then
 \[
  \overline{\Ascr_0^{b_1} \cap \overline{\Ascr_0^{b_2} \cap \ldots}} \cap \Ascr_0^{b_h} \not= \emptyset
 \]
 The second assertion of Theorem~\ref{thm dimension shimura} implies that the analogue holds for arbitrary Shimura varieties of PEL-type $\Ascr_0$, as the left hand side equals $\Ascr_0^{b_h}$.

 \subsection{Overview}
 The proof of Theorem~\ref{thm dimension shimura} follows an idea of Viehmann. By an analogous argument as in \cite{viehmann13} we prove in section~\ref{sect newton stratification} that the dimension formula as well as the closure relations follow if we show that $\dim \Ascr_0^{\leq b}$ is less or equal than the term (\ref{term dimension shimura}). By the work of Mantovan (\cite{mantovan05}) each Newton stratum is in a finite to finite correspondence with the product of a (truncated) Rapoport-Zink space and a so-called central leaf inside the Newton stratum. In particular the dimension of a Newton stratum is the sum of the dimension of a central leaf and a Rapoport-Zink space. Here a central leaf is defined as the locally closed subset of $\Ascr_0(\FFbar_p)$ where $\Aund[p^\infty] \cong \Xund$ for a fixed {\BTD} $\Xund$. In the sections \ref{sect central leaves} and \ref{sect EO strata} we calculate the dimension of the central leaves, thus reducing Theorem~\ref{thm dimension shimura} to the claim that $\Mscr(b,\mu)$ has dimension less or equal than (\ref{term dimension RZ-space}).

 We construct a correspondence between $\Mscr_G(b,\mu)$ and a disjoint union of Rapoport-Zink spaces associated to data with superbasic $\sigma$-conjugacy classes in section~\ref{sect BT again} using similar moduli spaces as in \cite{mantovan08}. In section~\ref{sect numerical dimension} we translate the dimension of the fibres of the correspondence into group theoretical terms and calculate it in section~\ref{sect fibre dimension}. This allows us to reduce to the case of a superbasic Rapoport-Zink datum of EL type in section~\ref{sect reduction to superbasic}. In section~\ref{sect superbasic} we prove Theorem~\ref{thm dimension RZ-space} in this special case thus finishing the proof of Theorem~\ref{thm dimension shimura}.

 We remark that the reduction step which reduces the task of estimating the dimension of Rapoport-Zink spaces to the case of a superbasic datum of EL type is an analogy of the reduction step in \cite{GHKR06}~sect.~5 for affine Deligne-Lusztig varieties. Especially the arguments given in section~\ref{sect fibre dimension} and section~\ref{sect reduction to superbasic} are very similar to those in \cite{GHKR06}. The proof of Theorem~\ref{thm dimension RZ-space} in the superbasic EL case in section~\ref{sect superbasic} follows the proof in \cite{viehmann08} replacing her combinatorial invariants by their generalisation introduced in \cite{hamacher}.

 The article is subdivided as follows. Sections~\ref{sect shimura} to \ref{sect BT} are mostly recapitulations of already known facts, except for section~\ref{ss defect} and \ref{ss PEL vs D}. In sections~\ref{sect newton stratification} to \ref{sect EO strata} we consider the Newton stratification on the special fibre of Shimura varieties. We give an overview of the relationship between theorems~1.1 -1.3 
 in section~\ref{sect epilogue}. In the subsequent sections we exclusively deal with the geometry of Rapoport-Zink spaces.

 \begin{notation}
  Throughout the article we keep the following notation. For any ring $R$ we denote by $W(R)$ its Witt-vectors. If $R$ has characteristic $p$, we denote by $\sigma$ the Frobenius endomorphism of $R$, as well as the Frobenius of $W(R)$ and $W(R)_\QQ$. For any $p$-adic field $F$ we denote by $O_F$ its ring of integers, by $k_F$ its residue field and set $q_F := \# k_F$. We denote $\Gamma_F := \Gal (\kbar_F/k_F) = \Aut_{F,\cont} (\hat{F}^\nr)$ and $\Gamma := \Gamma_{\QQ_p}$.  We denote by $k$ an arbitrary algebraically closed field of characteristic $p$ and $L = W(k)_\QQ$, $O_L = W(k)$. Starting in section~\ref{sect BT again} we will assume $k = \FFbar_p$.
  
  In most cases we will denote objects defined over number fields by letters in sans serif while we use the usual italic letters for objects defined over $p$-adic fields.
 \end{notation}

 \emph{Acknowledgements:} I would like to express my sincere gratitude to my advisor E.~Viehmann for entrusting me with this topic and for her steady encouragement and advice. I thank M.~Chen, M.~Kisin, E.~Viehmann and M.~Rapoport for giving me a preliminary version of their respective articles. The author was partially supported by ERC starting grant 277889 ``Moduli spaces of local $G$-shtukas''.

 \section{Shimura varieties of PEL type with good reduction} \label{sect shimura}

 \subsection{Moduli spaces of abelian varieties}
 We recall  Kottwitz's definition of an integral PEL-Shimura datum with hyperspecial level structure at $p$ (which we will abbreviate to ``unramified PEL-Shimura datum'') and the associated moduli spaces. They were defined in \cite{kottwitz92}, which is also the main reference for this section.

 The moduli problem is given by a datum $\Dscr= (\Bsf,\, ^{\ast},\Vsf,\pair,\Osf_\Bsf,\Lambda,h)$. To this datum one associates an linear algebraic group $\Gsf$ over $\QQ$, a conjugacy class $[\mu_h]$ of cocharacters of $\Gsf_\CC$ and a number field $\Esf$. This data have the following meaning.

 \begin{itemlist}
  \item $\Bsf$ is a finite dimensional semi-simple $\QQ$-algebra such that $\Bsf_{\QQ_p}$ is a product of matrix algebras over unramified extensions of $\QQ_p$.
  \item $^{\ast}$ is a positive involution of $\Bsf$ over $\QQ$.
  \item $\Vsf$ is a non-zero finitely generated left-$\Bsf$-module.
  \item $\pair$ is a symplectic form of the underlying $\QQ$-vector space of $\Vsf$ which is $\Bsf$-adapted, i.e.\ for all $v,w \in \Vsf$ and $b\in \Bsf$
  \[
   \langle bv,w\rangle = \langle v,b^\ast w\rangle.
  \]
  \item $\Osf_\Bsf$ is a $\ZZ_{(p)}$-order of $\Bsf$, whose $p$-adic completion $\Osf_{\Bsf,p}$ is a maximal order of $\Bsf_{\QQ_p}$.
  \item $\Lambda$ is a lattice in $\Vsf_{\QQ_p}$ which is self dual for $\pair$ and preserved under the action of $\Osf_\Bsf$.
  \item The group $\Gsf$ represents the functor
  \[
   \Gsf (R) = \{ g \in \GL_{\Bsf}(\Vsf \otimes R) \mid \exists c(g) \in R^\times: \langle g(v), g(w) \rangle = c(g) \cdot \langle v,w \rangle\}.
  \]
  We assume henceforth that $\Gsf$ is connected.
  \item $h: \CC \rightarrow \End_B(\Vsf)_\RR$ is a homomorphism of algebras with involution (where on the left hand side the involution is the complex conjugation and on the right hand side the involution maps a homomorphism to its adjoint with respect to $\pair$) such that the form $\langle v, h(i)\cdot w\rangle$ on $\End_\Bsf(\Vsf)_\RR$ is positive definite.
  \item Let $\Vsf_\CC = \Vsf^0 \oplus \Vsf^1$, where $\Vsf^0$ resp.\ $\Vsf^1$ is the subspace where $h(z)$ acts by $\zbar$ resp.\ $z$. We define $\mu_h$ to be the cocharacter of $\Gsf_\CC$ which acts with weight $0$ on $\Vsf^0$ and with weight $1$ on $\Vsf^1$. Then $[\mu_h]$ is defined as the $\Gsf(\CC)$-conjugacy class of $\mu_h$.
  \item $\Esf$ is the field of definition of $[\mu_h]$.
 \end{itemlist}

 \begin{definition}
  We call a datum $\Dscr$ as above an unramified PEL-Shimura datum. The field $\Esf$ is called reflex field of $\Dscr$.
 \end{definition}

 Let $K^p \subset \Gsf(\AA^p)$ be an open compact subgroup. We consider the functor $\Ascr_{\Dscr,K^p}$ which associates to an $\Osf_\Esf \otimes \ZZ_{(p)}$-algebra $R$ the set of isomorphism classes of tuples $(A,\lambda,\iota,\eta)$ where
 \begin{itemlist}
  \item $A$ is a projective abelian scheme over $\Spec R$.
  \item $\lambda$ is a polarization of $A$ of degree prime to $p$.
  \item $\iota: \Osf_\Bsf \rightarrow \Aut(A)$ is a homomorphism satisfying the following two conditions. For every $a\in \Osf_\Bsf$ we have the compatibility of $\lambda$ and $\iota$
   \begin{equation} \label{term kottwitz compatibility}
    \iota(a) = \lambda^{-1} \circ \iota(a^\ast)^\vee \circ \lambda
   \end{equation}
   and $\iota$ satisfies the Kottwitz determinant condition. That is, we have an equality of characteristic polynomials
   \begin{equation} \label{term kottwitz determinant}
    \cha (\iota(a)|\Lie A) = \cha (a|\Vsf^1).
   \end{equation}
   The polynomial on the right hand side has actually coefficients in $\Osf_\Esf \otimes \ZZ_{(p)}$, but we consider it as element of $R[X]$ via the structural morphism.
  \item $\eta$ is a level structure of type $K^p$ in the sense of \cite{kottwitz92} \S 5.
 \end{itemlist}
 Two such tuples $(A,\lambda,\iota,\eta), (A',\iota',\lambda',\eta')$ are isomorphic if there exists an isogeny $A \rightarrow A'$ of degree prime to $p$ which commutes with $\iota$, carrying $\lambda$ into a $\ZZ_{(p)}^\times$-scalar multiple of $\lambda'$ and carrying $\eta$ to $\eta'$.

 If $K^p$ is small enough, this functor is representable by a smooth quasi-projective $\Osf_\Esf \otimes \ZZ_{(p)}$-scheme. Henceforth we will always assume that this is the case.

 We fix $\Dscr,K^p$ as above and choose an embedding $\QQbar \mono \QQbar_p$ of the algebraic closure of $\QQ$ in $\CC$ into an algebraic closure of $\QQ_p$. We denote by $v$ the place of $\Esf$ over $p$ which is given by this embedding and by $\kappa$ its residue field. We write
 \[
  \Ascr_0 := \Ascr_{\Dscr,K^p,0} = \Ascr_{\Dscr,K^p} \times \bar\kappa
 \]
 for the geometric special fibre of $\Ascr_{\Dscr,K^p}$ at $v$. We fix an isomorphism $\bar\kappa \cong \bar\FF_p$ and denote by $\Aund^{\univ}$ the universal object over $\Ascr_0$.

 \subsection{\BTDs} \label{ss preliminaries}

 Let $R$ be a $\FFbar_p$-algebra and $(A,\iota,\lambda,\eta) \in \Ascr_0(R)$. Then by functoriality we obtain additional structure on the Barsotti-Tate group $X = A[p^\infty]$ of $A$. That is, we get an action of $\Osf_{\Bsf,p}$ on $A[p^\infty]$ and a polarisation up to $\ZZ_p^\times$-scalar multiple, satisfying the same compatibility conditions as $\iota$ and $\lambda$. 

 \begin{notation}
  We call a polarisation up to $\ZZ_p^\times$-scalar multiple as above a $\ZZ_p^\times$-homogeneous polarisation. We will use the analogous notion for bilinear forms which are defined up $\ZZ_p^\times$- or $\QQ_p^\times$-scalar multiple. 
 \end{notation}

 \begin{definition}
   Let $X$ be a \BT\ over a $\kappa$-algebra $R$, together with a homomorphism $\iota:\Osf_{\Bsf,p} \to \End X$ and a $\ZZ_p^\times$-homogeneous polarisation $\lambda:X \to X^\vee$. The tuple $\Xund = (X,\iota,\lambda)$ is called a \BTD\ if the following conditions are satisfied.
   \begin{assertionlist}
    \item Let $\Bsf_{\QQ_p} = \prod B_i$ be a decomposition into simple factors and $\Vsf_{\QQ_p} = \prod V_i$, $\Osf_{\Bsf,p} = \prod O_{B_i}$,the induced decompositions. Denote by $\epsilon_i$ the multiplicative unit in $B_i$ and let $X_i := \image \epsilon_i$. Then $X_i$ is a \BT\ with height $\dim_{\QQ_p} V_i$.
    \item $\iota(a) = \lambda^{-1} \circ \iota(a^\ast)^\vee \circ \lambda$.
    \item $\cha (\iota(a)|\Lie X) = \cha (a| \Vsf_v^1)$ where $\Vsf_v^1$ denotes the $v$-adic completion of $\Vsf^1$.
   \end{assertionlist}
 \end{definition}
 
 \begin{remark}
  The condition that $X_i$ is a \BT\ is automatic, cf.~section~\ref{ss normal forms of BTpELs}.
 \end{remark}

 Some of the data above are superfluous for the general study of \BTDs. Therefore we introduce the following (simpler) objects.

 \begin{definition}
  \begin{subenv} 
   \item Let $B$ be a finite product of matrix algebras over finite unramified field extensions of $\QQ_p$ and $O_B \subset B$ be a maximal order. We call a \BT\ with $O_B$-action a \BT\ with EL structure.
   \item Let $O_B,B$ be as above and $^\ast$ be a $\QQ_p$-linear involution of $B$ which stabilizes $O_B$. We call a \BT\ with polarisation $\lambda$ and $O_B$-action $\iota$ satisfying
   \[
    \iota(a) = \lambda^{-1} \circ \iota(a^\ast)^\vee \circ \lambda
   \]
   a \BT\ with PEL structure.
   \item We call a tuple $(B,O_B)$ resp.\ $(B,O_B,\ast)$ as above an EL-datum resp.\ a PEL-datum.
  \end{subenv}
 \end{definition}

 \begin{notation}
  When we want to consider the EL case and the PEL case simultaneously, we will write the additional data in brackets, e.g.\ ``Let $\Xund$ be a \BTpEL .''
 \end{notation}

 \begin{definition}
  Let $\Xund = (X,\iota,(\lambda))$ and $\Xund' = (X',\iota',(\lambda'))$ be two \BTpELs.
  \begin{subenv}
   \item A morphism $\Xund \to \Xund'$ is a homomorphism of \BT s $X \to X'$ which commutes with the $O_B$-action and the polarisation in the PEL case.
   \item An isogeny $\Xund \to \Xund'$ is an isogeny $X \to X'$ which commutes with the $O_B$-action and in the PEL case also commutes with the polarisation up to $\QQ_p^\times$-scalar. The scalar given in the PEL case is called the similitude factor of the isogeny.
  \end{subenv}
 \end{definition}
 We note that an isogeny of {\BTPELs} is not necessarily a homomorphism.

 \subsection{Deformation theory} \label{ss deformation}
 Let $\Xund = (X,\iota,(\lambda))$ be a \BTpEL\ over a perfect field $k_0$ of characteristic $p$. We briefly recall the construction of its universal deformation $\Xund^\univ$, where we only consider deformations in equal characteristic. By \cite{illusie85},~Cor.~4.8 the deformation space $\Def (X)$ is pro-representable by $\Spf k_0 \pot{X_1, \ldots , X_{d\cdot (n-d)}}$ where $n$ denotes the height of $X$ and $d$ its dimension. In order to describe $\Def(\Xund)$ we use the following result of Drinfeld. For any Artinian local $k_0$-algebra $A$ and \BT s $X',X''$ over $A$ the canonical map
 \[
  \Hom(X',X'') \otimes_{\ZZ_p} \QQ_p \to \Hom(X'_{k_0},X''_{k_0}) \otimes_{\ZZ_p} \QQ_p
 \]
 is an isomorphism. Now the condition that an element of $\Hom(X',X'') \otimes \QQ_p$ is a homomorphism is closed on $\Spec A$ by \cite{RZ96},~Prop.~2.9. Thus $\Def (\Xund)$ is a closed subfunctor of $\Def(X)$ and in particular pro-representable by $\Spf \Rscr_{\Xund}$ for some adic ring $\Rscr_{\Xund}$. Denote by $\Xund^{\rm def}$ the universal object over $\Spf \Rscr_X$. By a result of Messing (\cite{messing72}~Lemma~II.4.16) for any $I$-adic ring $R$ the functor
 \begin{eqnarray*}
  \{\textnormal{\BT s over } \Spec R\} &\to& \{\textnormal{\BT s over} \Spf R\} \\
  X' &\mapsto& (X' \mod I^n)_{n\in\NN}
 \end{eqnarray*}
 is an equivalence of categories. Thus $\Xund^{\rm def}$ induces a \BTpEL\ $\Xund^\univ$ over $\Spec\Rscr_{\Xund}$. We note that the same construction also works for \BTPELs\ with $\ZZ_p^\times$-homogeneous polarisation and in particular yields a canonically isomorphic deformation functor.
 
 \begin{definition} \label{defi deformation space}
  We call $\Xund^\univ$ the universal deformation of $\Xund$. We denote $\Sscr_{\Xund} = \Spec \Rscr_{\Xund}$
 \end{definition}
 
 The Serre-Tate theorem states that the canonical homomorphism $\Def (A) \to \Def (A[p^\infty])$ is an isomorphism for every abelian variety $A$ over an algebraically closed field $k$ of characteristic $p$. We obtain the following corollary.
 
 \begin{proposition} \label{prop serre-tate}
  Let $x = \Aund \in \Ascr_0 (\FFbar_p)$ and $\Xund := \Aund[p^\infty]$. Then the morphism $\Rscr_{\Xund} \to \widehat\Oscr_{\Ascr_0,x}$ induced by the deformation $(A^{\univ}[p^\infty])_{\widehat\Oscr_{\Ascr_0,x}}$ is an isomorphism and the pull-back of $(A^{\univ}[p^\infty])_{\widehat\Oscr_{\Ascr_0,x}}$ equals $\Xund^\univ$.
 \end{proposition}
 
 \section{Group theoretic preliminaries} \label{sect group theory}

 \subsection{Reductive group schemes over $\ZZ_p$} \label{ss group theory}

 We recall the following definitions from \cite{SGA3-3}.

 \begin{definition}
  Let $S$ be an arbitrary scheme. A group scheme $G \to S$ is called reductive if the structure morphism is smooth and affine with connected fibres and for every geometric point $\sbar$ of $S$ the linear algebraic group $G_{\sbar}$ is reductive.  
 \end{definition}

 \begin{definition}
  Let $G \to S$ be a reductive group scheme.
  \begin{subenv}
   \item A maximal torus of $G$ is a subtorus $T \subset G$ such that for every $T_{\sbar}$ is a maximal torus of $G_{\sbar}$ for all geometric points $\sbar$ of $S$.
   \item A Borel subgroup of $G$ is a subgroup $B \subset G$ such that $B_{\sbar}$ is a Borel subgroup of $G_{\sbar}$ for all geometric points $\sbar$ of $S$.
  \end{subenv}
 \end{definition}

 In the case where $S = \Spec R$ is the spectrum of a local ring we make the following definitions. A reductive group scheme $G$ over $S$ is called \emph{split} if it contains a maximal split torus and it is called \emph{quasi-split} if it contains a Borel subgroup. The notions of split resp.\ quasi-split reductive group schemes also exist over arbitrary bases, but are a bit more complicated.

 Let $G$ be a reductive group scheme over $\ZZ_p$. Then it is automatically quasi-split and splits over a finite unramified extension $O_F$ of $\ZZ_p$ (cf. \cite{VW13}~A.4). We fix $T \subset B \subset G$, where $T$ is a maximal torus and $B$ a Borel subgroup of $G$. Denote
 \begin{eqnarray*}
  X^*(T) &=& \Homfunc (T,\GG_m) \\
  X_*(T) &=& \Homfunc (\GG_m,T)
 \end{eqnarray*}
 These sheaves become constant after base change to $O_F$, thus we regard them as abstract groups with $\Aut_{\ZZ_p}(O_F)$ action. We obtain canonical isomorphisms of Galois-modules
 \[
  X_*(T) \cong X_*(T_{\QQ_p}) \cong X_*(T_{\FF_p})
 \]
 where we also identify $\Gal(F/\QQ_p) = \Aut_{\ZZ_p}(O_F) = \Gal(k_F/\FF_p)$ (and analogously for $X^*(T)$). We denote by $R$ (resp. $R^\vee$) the set of absolute roots of $G$ with respect to $T$, that is the lifts of the absolute roots (resp. coroots) of $G_{\FF_p}$ to $O_F$. This definition coincides with the definition of roots (resp.\ coroots) of $G_{O_F}$ given in \cite{SGA3-3}, Exp.~XXII, ch.~1. In particular, the absolute roots (resp.\ coroots) of $G$ also coincide with the absolute roots (resp.\ coroots) of $G_{\QQ_p}$ w.r.t. the identifications above. We denote by $R^+,\Delta^+$ and $R^{\vee,+},\Delta^{\vee,+}$ the system of positive/simple roots resp.\ positive/simple coroots determined by $B$. Let $\pi_1(G)$ denote the fundamental group of $G$, i.e. the quotient of $X_*(T)$ by the coroot lattice.

 The Weyl group of $G$ is defined as the quotient $W := (\Norm_G T)/T$. It is represented by a finite \'etale scheme which becomes constant after base change to $O_F$. Thus we may identify $W = (\Norm_G T)(O_F)/T(O_F)$ with the canonical Galois action. In particular $W$ is canonically isomorphic to the absolute Weyl groups of $G_{\FF_p}$ and $G_{\QQ_p}$ equipped with Galois action.

 We denote by $\Wtilde := \Norm_G(T)(L)/T(O_L) \cong \Norm_G T(F)/T(O_F)$ the (absolute) extended affine Weyl group of $G$, equipped with the canonical Galois action. We will often consider an element $x \in \Wtilde$ as an element of $G(L)$ by which we mean an arbitrary lift of $x$. We have $\Wtilde \cong W \rtimes X_*(T)$ canonically; denote by $p^\mu$ the image of a cocharacter $\mu$ in $\Wtilde$. The canonical inclusion of the affine Weyl group $W_a$ into $\Wtilde$ yields a short exact sequence
 \begin{center}
  \begin{tikzcd}
   0 \arrow{r} & W_a \arrow{r} & \Wtilde \arrow{r} & \pi_1(G) \arrow{r} & 0.
  \end{tikzcd}
 \end{center}
 The isomorphism $\Wtilde \cong W \rtimes X_*(T)$ defines an action of $\Wtilde$ on the apartment $\afr := X_*(T)_\RR$ by affine linear maps. As $W_a$ acts simply transitively on the set of alcoves in $\afr$, the stabilizer $\Omega \subset \Wtilde$ of a fixed ``base'' alcove defines a right-splitting of the exact sequence above. We choose as the base alcove the alcove in the anti-dominant chamber whose closure contains the origin. This alcove corresponds to the Iwahori subgroup $\Iscr$ of $G(L)$ which is defined as the preimage of $B(\FFbar_p)$ w.r.t.\ the canonical projection $G(O_L) \epi G(\FFbar_p)$. We define the length function on $\Wtilde$ by
 \[
  \ell(w\tau) = \ell(w).
 \]
 for $w \in W_a,\tau\in  \Omega$. In particular, the elements of length $0$ are precisely those which are contained in $\Omega$.

 \subsection{$\sigma$-conjugacy classes} \label{ss sigma conjugacy}
 Using Dieudonn\'e theory one gets a bijection between the isogeny classes of \BTDs\ over $k$ and certain $\sigma$-conjugacy classes in the $L$-valued points of a reductive group scheme over $\ZZ_p$ (cf.\ next section). For this reason we briefly recall Kottwitz's classification of $\sigma$-conjugacy classes in the case of unramified groups.  The main reference for this subsection is the article of Rapoport and Richartz \cite{RR96}.

 We keep the notation of the previous subsection. Recall that two elements $b,b' \in G(L)$ are called $\sigma$-conjugated if there exists an element $g\in G(L)$ such that $b' = gb\sigma(g)^{-1}$. The equivalence classes with respect to this relation are called $\sigma$-conjugacy classes; we denote the $\sigma$-conjugacy class of an element $b \in G(L)$ by $[b]$. Let $B(G)$ denote  the set of all $\sigma$-conjugacy classes in $G(L)$. By \cite{RR96} Lemma 1.3 the sets of $\sigma$-conjugacy classes does not depend on the choice of $k$ (up to canonical bijection), so this notation is without ambiguity.

 Kottwitz assigns in \cite{kottwitz85} to each $\sigma$-conjugacy class $[b]$ two functorial invariants
 \[
  \nu_G(b) \in X_*(T)_{\QQ,\dom}^\Gamma
 \]
 \[
  \kappa_G (b) \in \pi_1(G)_\Gamma, 
 \]
 which are called the Newton point resp.\ the Kottwitz point of $[b]$. Those two invariants determine $[b]$ uniquely. 

 \begin{example} \label{ex isocrystal}
   \emph{(1)} Assume $G = \GL_n$. We have a one-to-one correspondence
   \begin{eqnarray*}
    B(G) &\leftrightarrow& \{ \textnormal{isocrystals over } k \textnormal{ of height } n \}/\cong \\
    \left[ b \right] &\mapsto& (L^n,b\sigma).
   \end{eqnarray*}
   The above bijection is easy to see, as a base change of $(L^n,b\sigma)$ by the matrix $g$ replaces $b$ with $gb\sigma(g)^{-1}$. 
   
   Now we choose $T$ to be the diagonal torus and $B$ to be the Borel subgroup of upper triangular matrices. Then we have canonical isomorphisms $X_*(T)_{\QQ}^\Gamma = X_*(T)_{\QQ} = \QQ^n$ and $\pi_1(G)_\Gamma = \pi_1(G) = \ZZ$. The first isomorphism identifies
   \[
    X_*(T)_{\QQ,\dom} = \{\nu = (\nu_1,\ldots,\nu_n\} \in \QQ^n\mid \nu_1 \geq \ldots \geq \nu_n\}.
   \]
   Then $\nu_G(b) = (\nu_1, \ldots, \nu_n)$ is the vector of Newton slopes of the isocrystal $(L^n,b\sigma)$ given in descending order. Of course, this already determines  $[b]$ uniquely. The Kottwitz point is given by
   \[
    \kappa_G (b) = \val \det b = \nu_1 + \ldots + \nu_n.
   \]
 
   \emph{(2)} Let $F/\QQ_p$ be a finite unramified field extension of degree $d$ and let $G = \Res_{F/\QQ_p} \GL_n$. Similar as above one sees that $[b] \mapsto ((F\otimes L)^n,b(1\otimes\sigma))$ defines a bijection between $B(G)$ and the isomorphism classes of isocrystals $(N,\Phi)$ over $k$ of height $n$ together with an $F$-action $\iota: F \mono \Aut(N,\Phi)$. We have a canonical isomorphism $F\otimes L \cong \prod_{\tau: F \mono L} L$ and likewise
   \[
    (L\otimes F)^n \cong \prod_{\tau:F\mono L} L^n =: \prod_{\tau:F\mono L} N_\tau.
   \]
   Then $\sigma$ defines a bijection of $N_\tau$ onto $N_{\sigma\tau}$ and any element $b\in G(L)$ stabilizes the $N_\tau$. Fixing an embedding $\tau: F \mono L$, we thus obtain an equivalence of categories
   \begin{eqnarray*}
    \{ \textnormal{isocrystals over } k \textnormal{ of height } n \textnormal{ with } F-\textnormal{action} \} &\rightarrow& \{ \sigma^d-\textnormal{isocrystals over } k \textnormal{ of height } n \} \\
    (N,\Phi,\iota) &\mapsto& (N_\tau, \Phi^d).
   \end{eqnarray*}
   Using that in $\GL_n(L)$ any element $g$ is $\sigma^d$-conjugate to $\sigma(g)$ (this holds as every $\sigma^d$-conjugacy class contains a $\sigma$-stable element,\ e.g. a suitable lift of an element in $\Wtilde$) one sees that the isomorphism class of the object on the right hand side does not depend on the choice of $\tau$. Hence if we denote the Newton slopes of $(N_\tau, \Phi^d)$  by $(\lambda_1,\ldots,\lambda_n)$ the slopes of the isocrystal $(N,\Phi)$ (forgetting the $F$-action) equal
   \[
    (\underbrace{\frac{\lambda_1}{d},\ldots,\frac{\lambda_1}{d}}_{d\textnormal{ times}},\ldots,\underbrace{\frac{\lambda_n}{d},\ldots,\frac{\lambda_n}{d}}_{d \textnormal{ times}}).
   \]
   We choose $T$ to be the diagonal torus and $B$ to be the Borel subgroup of upper triangular matrices. Then $X_*(T) \cong \prod_{\tau:F\mono L} \ZZ^n$ canonically identifying
   \[
    X_*(T)_{\QQ,\dom}^\Gamma = \{\nu = ((\nu_1,\ldots,\nu_n))_{\tau} \in \prod \QQ^n\mid \nu_1 \geq\ldots\geq \nu_n \}
   \]
   Then by functoriality $\nu_G(b) = ((\frac{\lambda_1}{d},\ldots , \frac{\lambda_n}{d}))_\tau$.  
 \end{example}

 Recall that we have the Cartan decomposition
 \[
  G(L) = \bigsqcup_{\mu \in X_*(T)_\dom} G(O_L)p^{\mu}G(O_L).
 \]
 An estimate for $\nu$ and $\kappa$ on a $G(O_L)$-double coset is given by the generalised Mazur inequality. Before we can state it, we need to introduce some more notation.
 We equip $X_*(T)_\QQ$ with a partial order $\leq$ where we say that $\mu' \leq \mu$ if $\mu-\mu'$ is a linear combination of positive coroots with positive (rational) coefficients. For any cocharacter $\mu \in X_*(T)$ we denote by $\mubar$ the average of its $\Gamma$-orbit.

 \begin{proposition}[\cite{RR96},~Thm.~4.2]
  Let $b \in G(O_L)\mu(p)G(O_L)$ for $\mu \in X_*(T)_\dom$. Then the following assertions hold.
  \begin{subenv}
   \item We have $\nu_G(b) \leq \mu$.
   \item The Kottwitz point $\kappa_G(b)$ equals the image of $\mu$ in $\pi_1(G)_\Gamma$.
  \end{subenv}
 \end{proposition}
 
 \begin{definition}
  \begin{subenv}
   \item We define the partial order $\leq$ on $B(G)$ by
   \[
    [b'] \leq [b] :\Lra \nu_G(b') \leq \nu_G(b) \textnormal{ and } \kappa_G(b') = \kappa_G(b). 
   \]
   \item We denote
   \begin{eqnarray*}
    B(G,\mu) &=& \{[b] \in B(G)\mid \nu_G(b) \leq \mu \textnormal{ and } \kappa_G(B) \textnormal{ is the image of } \mu \textnormal{ in } \pi_1(G)_\Gamma\} \\
             &=& \{[b] \in B(G)\mid [b] \leq [p^\mu]\}.
   \end{eqnarray*}
  \end{subenv}
 \end{definition}

 By the generalized Mazur inequality $B(G,\mu)$ contains all $\sigma$-conjugacy classes which intersect $G(O_L)\mu(p)G(O_L)$ non-emptily. It is known that the converse is also true. Many authors have worked on this conjecture, the result in the generality we need was proven by Kottwitz and Gashi (\cite{kottwitz03}~ch.~4.3, \cite{gashi10}~Thm.~5.2).

 To every $\sigma$-conjugacy class $[b]$ one associates linear algebraic groups $M_b$ and $J_b$ which are defined over $\QQ_p$. The group $M_b$ is defined as the centraliser of $\nu(b)$ in $G$. So in particular $M_b$ is a standard Levi subgroup of $G$. Kottwitz showed that the intersection of $M_b$ and $[b]$ is non-empty (\cite{kottwitz85}~ch.~6). The group $J_b$ represents the functor
 \[
  J_b(R) = \{g \in G(R \otimes_{\QQ_p} L)\mid gb = b\sigma(g) \}.
 \]
 This group is an inner form of $M_b$ which (up to canonical isomorphism) does not depend on the choice of the representative of $[b]$ (\cite{kottwitz85}~\S~5.2).

 \begin{definition} \label{def basic}
  Let $[b] \in B(G)$.
  \begin{subenv}
   \item $[b]$ is called basic if $\nu_G(b)$ is central.
   \item $[b]$ is called superbasic if every intersection with a proper Levi subgroup of $G$ is empty.
  \end{subenv}
 \end{definition}
 We note that $M_b$ is a proper subgroup of $G$ if and only if $[b]$ is not basic. As $M_b$ always intersects $[b]$ non-trivially, this observation shows that every superbasic $\sigma$-conjugacy class is also basic. We have a bijection between the basic $\sigma$-conjugacy classes of $G$ and $\pi(G)_\Gamma$ induced by the Kottwitz point (\cite{kottwitz85}~Prop.~5.6).
 
 Finally we can define the last group theoretic invariant which appears in the dimension formula.
 \begin{definition}
  Let $[b] \in B(G)$. We define the defect of $[b]$ by
  \[
   \defect_G (b) := \rank_{\QQ_p} G - \rank_{\QQ_p} J_b
  \]
 \end{definition}

 \subsection{A formula for the defect} \label{ss defect}
 We keep the notation above. Furthermore, let $\omegaund_1,\ldots,\omegaund_l$ be the sums over all elements in a Galois orbit of absolute fundamental weights of $G$. Recall that we have an embedding $\pi_1(G) \cong \Omega \mono \Wtilde$. For $\varpi \in \pi_1(G)$ let $\dot\varpi$ be its image in $\Wtilde$. Then by construction $\dot\varpi$ is basic and $\kappa_G(\dot\varpi)$ is the image of $\varpi$ in $\pi_1(G)_\Gamma$.

 \begin{proposition} \label{prop calculation of defect}
  Let $b \in G(L)$. Then
  \[
   \defect_G (b) = 2\cdot\sum_{i=1}^l \{ \langle \nu_G(b), \omegaund_i \rangle \}
  \]
  where $\{\cdot\}$ denotes the fractional part of a rational number.
 \end{proposition}
 
 The proposition is a generalisation of \cite{kottwitz06}~Cor.~1.10.2. The proof of the proposition given here is a generalisation of Kottwitz's proof. The calculations will use the following combinatorial result.

 \begin{lemma} \label{lem root theoretic}
  Let $\Psi = (X,R,R^+,X^\vee,R^\vee,R^{\vee,+})$ be a based reduced root datum. We fix the following notations, which will be used until the end of the proof of this lemma.
  \begin{simplelist}
   \item $P^\vee := $ coweight lattice of $\Psi$.
   \item $Q^\vee := $ coroot lattice of $\Psi$.
   \item $\pi_1 := P^\vee/Q^\vee$ denotes the fundamental group.
   \item $I := \Aut\Psi$.
   \item $\omega_1,\ldots,\omega_l := $ fundamental weights of $\Psi$.
   \item $\varpi_1,\ldots,\varpi_l := $ images of $\omega_1^\vee,\ldots,\omega_l^\vee$ in $\pi_1$. 
   \item $\chi_1,\ldots, \chi_l$ are the characters of $\pi_1$ defined by $\chi_j(\varpi) := \exp (-2\pi i \cdot \langle \varpi,\omega_j\rangle)$.
   \item $\Xi := \mathbf{1} \oplus \chi_1 \oplus \cdots \oplus \chi_l$ seen as $I \ltimes \pi_1$-representation. Here the action of $I$ is given by the permutation of the $\chi_j$ according to its action on the fundamental weights.
   \item $V' :=$ vector space of affine linear functions on $X^\vee_\RR$ seen as $I \ltimes \pi_1$-representation.
  \end{simplelist}
  Then $V' \cong \Xi$.
 \end{lemma}
 
 \begin{remark}
  If $\Psi$ is the root datum of a reductive group $G$ then in general $\pi_1(G)$ is only a subgroup of $\pi_1$. We have equality if and only if $P^\vee = X^\vee$, i.e. $G$ is adjoint.
 \end{remark}

 \begin{proof}
  The assertion that $\Xi$ and $V'$ are isomorphic as representations of $\pi_1$ was proven in \cite{kottwitz06}~Lemma~4.1.1. In particular, it proves our assertion in the case where $I$ is trivial.
  
  We assume without loss of generality that the Dynkin diagram of $\Psi$ is connected. We have to check the assertion only in the cases where $I$ is non-trivial, i.e.\ when the Dynkin diagram is of type $A_l$, $D_l$ or $E_6$. We show that $V'$ and $\Xi$ are isomorphic by calculating their characters. It is obvious how to calculate the character of $\Xi$. For the character of $V'$ we use that the action of $I \ltimes \pi_1$ permutes the simple affine roots. We obtain for $\sigma\cdot\varpi \in I \ltimes \pi_1$
  \begin{eqnarray*}
   \tr (\sigma\cdot\varpi \mid V') &=& \# \textnormal{ simple affine roots fixed by } \sigma\cdot\varpi, \\
   \tr (\sigma\cdot\varpi \mid \Xi) &=& 1 + \sum_{i; \sigma(\omega_i)  = \omega_i} \omega_i(\varpi).
  \end{eqnarray*}
  Now all data we need to calculate the right hand sides are given in \cite{bourbaki68} and thus the claim is reduced to some straightforward calculations.

  We give these calculations for type $A_l$ as an example and skip the calculations in the other two cases as they are analogous. We use the notation of \cite{bourbaki68}~ch.~VI,~planche~I. That is $\alpha_0,\alpha_1,\ldots,\alpha_l$ denote the simple affine roots, where $\alpha_0$ is the unique root which is not finite and $\omega_i$ (in Bourbaki's notation $\bar\omega_i$) denotes the fundamental weight associated to $\alpha_i$. We will consider the indices as elements of $\ZZ/(l+1)\ZZ$.

  We have $\pi_1 = \{1, \varpi_1,\ldots,\varpi_l \} \cong \ZZ/(l+1)\ZZ$ with distinguished generator $\varpi_1$. The group $\pi_1$ acts  on the set of simple affine roots by cyclic permutation. In particular an element of $\pi_1$ different from the identity acts fixed point free on the set of simple affine roots.
  
  Now $I = \{1,\sigma\} \cong \ZZ/2\ZZ$. The non-trivial automorphism $\sigma$ acts on the set of simple affine roots by $\sigma(\alpha_j) = \alpha_{-j}$. Altogether, we obtain
  \[
   \sigma\cdot\varpi_1^k(\alpha_j) = \alpha_{-j-k}
  \]
  Thus the number of simple affine roots fixed by $\sigma\cdot\varpi_1^k$ equals the number of solutions of the equation \[2j \equiv -k \mod l+1.\] This yields the following values for the character of $V'$. If $l$ is even, then
  \begin{eqnarray*}
   \tr (1\cdot \varpi_1^k \mid V') &=& \left\{ \begin{array}{ll}
                                       0 \quad & \textnormal{ if } k = 1,\ldots,l \\
                                       l+1 \quad & \textnormal{ if } k = 0
                                      \end{array} \right. \\
   \tr (\sigma\cdot\varpi_1^k \mid V') &=& 1
  \end{eqnarray*}
  and if $l$ is odd we have
  \begin{eqnarray*}
   \tr (1\cdot \varpi_1^k \mid V') &=& \left\{ \begin{array}{ll}
                                       0 \quad & \textnormal{ if } k = 1,\ldots,l \\
                                       l+1 \quad & \textnormal{ if } k = 0
                                      \end{array} \right. \\
   \tr (\sigma\cdot\varpi_1^k \mid V') &=& \left\{ \begin{array}{ll}
                                          0 \quad & \textnormal{ if } k \textnormal{ is odd,} \\
                                          2 \quad & \textnormal{ if } k \textnormal{ is even.}
                                         \end{array} \right.
  \end{eqnarray*}
  Now we have to compare the above formulas to the character of $\Xi$. We note that
  \[
   \omega_1^\vee = \frac{1}{l+1} ( l \alpha_1^\vee + (l-1) \alpha_2^\vee + \ldots + 2 \alpha_{l-1}^\vee + \alpha_l^\vee).
  \]
  Hence
  \[
   \chi_j(\varpi_1^k) = \chi_j(\varpi_1)^k = \exp(2\pi i \cdot \frac{j}{l+1})^k.
  \]
  We obtain
  \[
   \tr (1\cdot \varpi_1^k \mid \Xi) = \sum_{k=0}^{l-1} \exp(2\pi i \cdot \frac{j}{l+1})^k = \left\{ \begin{array}{ll}
                                       0 \quad & \textnormal{ if } k = 1,\ldots,l \\
                                       l+1 \quad & \textnormal{ if } k = 0
                                      \end{array} \right. , 
  \]
  coinciding with the formula above. Now we have $\sigma(\alpha_j) = \alpha_j$ if and only if $l$ is odd and $j = \frac{l+1}{2}$. We obtain for even $l$
  \[
   \tr (\sigma\cdot\varpi_1^k \mid \Xi) = 1 + 0 = 1,
  \]
  and for odd $l$
  \[
   \tr (\sigma\cdot\varpi_1^k \mid \Xi) = 1 + \chi_{\frac{l+1}{2}} (\varpi_1^k) = 1 + (-1)^k.
  \]
  Thus we have indeed $ \tr ( \cdot \mid V') = \tr( \cdot \mid \Xi)$.
 \end{proof}

 \begin{proof}[Proof of Proposition~\ref{prop calculation of defect}]
  First we note that the equation does not change if we replace $G$ by $M_b^{\ad}$ (cf.~\cite{kottwitz06}~Lemma~1.9.1). Thus we may assume that $b$ is basic and $G$ of adjoint type. Hence $[b]$ is uniquely determined by its Kottwitz point, so we may even assume that $b$ is a representative of an element $\wtilde\in\Omega$ in the normaliser $(\Norm_G T)(L)$.

  Now conjugation with $b$ fixes $T$ and the standard Iwahori subgroup. By Bruhat-Tits theory the twist of $T$ by $b$ is a maximal torus of $J_b$ which contains a maximal split torus (see for example \cite{hamacher}~Lemma~4.5). Now the automorphism of $\afr$ induced by conjugation with $\wtilde$ equals the finite Weyl group part of $\wtilde$, which we denote by $w$. Thus $\rank_{\QQ_p} J_b = \dim \afr^{w\sigma}$ and we obtain
  \[
   \defect_G (b) = \dim \afr^\sigma - \dim \afr^{w\sigma}.
  \]
  As the action of $\Gamma$ factorises through the automorphism group of the root datum of $G$, we have an isomorphism $\afr \cong \chi_1 \oplus \cdots \oplus \chi_l$ of $\pi_1(G) \rtimes \Gamma$-representations by Lemma~\ref{lem root theoretic}. We have to calculate the dimension of $\afr^{w\sigma}$. Assume that $\chi_1,\ldots,\chi_{m}$ are one $\Gamma$-orbit with $\sigma$ mapping $\chi_j$ to $\chi_{j+1}$. Let $ v = (v_1,\ldots,v_m) \in \chi_1 \oplus \cdots\oplus\chi_m$, then 
  \[
   (w\sigma) (v) = (\chi_1(\wtilde)\cdot v_m, \chi_2(\wtilde)\cdot v_1, \ldots , \chi_m(\wtilde)\cdot v_{m-1}).
  \]
  We see that $v$ is fixed by $w\sigma$ if and only if $v_1 = \chi_1(\wtilde) \cdot \cdots \cdot \chi_m(\wtilde) \cdot v_1$ and $v_j = \chi_j (v_{j-1})$ for $j>1$. Thus the subspace of $\chi_1 \oplus \cdots\oplus\chi_m$ of vectors fixed by $w\sigma$ has dimension $1$ if $\chi_1(\wtilde) \cdot \ldots \cdot \chi_m(\wtilde) =1$ and dimension $0$ otherwise. We obtain
  \[
   \rank J_b = \dim \afr^{w\sigma} = \#\{i \mid \langle \nu, \omegaund_i \rangle \in \ZZ\}.
  \]
  As the $\Omega \rtimes \Gamma$-representation $\afr$ (and thus $\chi_1 \oplus\cdots\oplus \chi_l$) is self-contragredient, we have
  \begin{eqnarray*}
   2 \sum_{i=1}^l \{\langle \nu, \omegaund_i\rangle\} &=& \sum_{i=1}^l \{\langle \nu, \omegaund_i\rangle\} + \sum_{i=1}^l \{\langle \nu, -\omegaund_i\rangle\} \\
   &=& l - \#\{ i \mid \langle \nu,\omegaund_i \rangle \in \ZZ \} \\
   &=& \rank_{\QQ_p} G - \rank_{\QQ_p} J_b \\
   &=& \defect_G (b).
  \end{eqnarray*}
 \end{proof}

 \subsection{Chains of Newton points} \label{ss chains}
 Our calculations in the previous subsection allows us to reformulate the dimension formula (\ref{term dimension shimura}) in terms of Chai's chains of Newton points. We first recall some of his notions and results from \cite{chai00}.

 Denote by $\Nscr(G)$ the image of $\nu_G$. For $\nu \in \Nscr(G)$ and $[b] \in B(G)$ with $\nu_G(b) = \nu$ the image $\Nscr(G)_{\leq\nu}$ of the set
 $
  \{ [b'] \in B(G)\mid \nu(b') \leq \nu, \kappa(b') = \kappa(b) \}
 $
 in $\Nscr(G)$ only depends on $\nu$ (\cite{chai00}~Prop.~4.4). For elements $\nu'' \leq \nu'$ in $\Nscr(G)_{\leq \nu}$ define
 \[
  [\nu'',\nu'] = \{ \xi \in \Nscr(G)_{\leq\nu}\mid \nu'' \leq \xi \leq \nu'\}.
 \]
 We note that $[\nu'',\nu']$ does not change if we replace $\nu$ by $\nu'$, thus it is independent of the choice of $\nu$. We denote by  $\lengthG{\nu'}{\nu}$ the maximum of all integers $n$ such that there exists a chain $\nu_0 \leq \ldots \leq \nu_n$ in $[\nu',\nu]$.

 Chai gave a formula for $\lengthG{\nu'}{\nu}$ but made a small mistake in his calculations. In the formula at the bottom of page 982 one has to replace the relative fundamental weights $\omega_{F,j}$ by the sum of elements of a Galois orbit of absolute fundamental weights $\omegaund_j$. As the $\omega_{F.j}$ and $\omegaund_j$ are scalar multiples of each other, the other assertions in \cite{chai00}~section~7 and the proofs remain valid.

 \begin{proposition}[cf.~\cite{chai00}~Thm.~7.4~(iv)] \label{prop chai}
  Let $\nu \in \Nscr(G)$ and $\nu' \in \Nscr(G)_{\leq\nu}$. Then
  \[
   \lengthG{\nu'}{\nu} = \sum_{j=1}^{l} \lceil\langle\nu,\omegaund_j\rangle - \langle\nu',\omegaund_j\rangle\rceil.
  \]
 \end{proposition}

 Now we can reformulate the dimension formula of Theorem~\ref{thm dimension shimura} in a more elegant terms.

 \begin{corollary} \label{cor dimension formulas}
  Let $[b] \in B(G,\mu)$ and $\nu$ its Newton point. Then
  \[
   \langle \mu+\nu, \rho \rangle - \frac{1}{2} \defect_G(b) = 2 \langle \rho, \mu \rangle - \lengthG{\nu}{\mu}
  \]
 \end{corollary}

 \begin{proof}
  \begin{eqnarray*}
   \langle \mu+\nu, \rho \rangle - \frac{1}{2} \defect_G(b) &=& 2 \langle \rho, \mu \rangle - \langle \rho, \mu-\nu \rangle+ \frac{1}{2} \defect_G(b)\\
   &\stackrel{\rm Prop.~\ref{prop calculation of defect}}{=}& 2 \langle \rho, \mu \rangle - \sum_{j=1}^l \lceil \langle\omegaund_j, \mu-\nu \rangle\rceil \\
   &\stackrel{\rm Prop.~\ref{prop chai}}{=}& 2 \langle \rho, \mu \rangle - \lengthG{\nu}{\mu}.
  \end{eqnarray*}
 \end{proof}

 \begin{remark} \label{rem dimension formula}
  In a similar fashion one can rewrite the dimension formula of Theorem~\ref{thm dimension RZ-space}
  \[
   \langle \rho, \mu-\nu \rangle - \frac{1}{2} \defect_G(b) = \sum_{j=1}^l \lfloor \langle \omegaund_j, \mu-\nu\rangle \rfloor.
  \]
  This is almost the same as the formula conjectured by Rapoport in \cite{rapoport05},~p.~296
 \[
  \langle 2\rho, \overline{\mu}-\nu \rangle + \sum_{j=1}^l \lfloor -\langle \omega_{\QQ_p,j}, \bar\mu-\nu\rangle \rfloor,
 \]
 where $\overline\mu$ denotes the average over the $\Gamma$-orbit of $\mu$ and $\omega_{\QQ_p,1},\ldots,\omega_{\QQ_p,l}$ are the relative fundamental weights. One just has to replace the $\omega_{\QQ_p,j}$ in Rapoport's formula by $\omegaund_j$.
 \end{remark}

 \section{\BT s with additional structure} \label{sect BT}

 \subsection{Decomposition of \BTpELs} \label{ss normal forms of BTpELs}
 In this subsection we recall a mechanism that will often allow us to reduce to the case where $B=F$ is a finite unramified field extension of $\QQ_p$. Even though this mechanism is well known (see for example \cite{fargues04}), I could not find a reference for its proof.

 \begin{lemma} \label{lem BT decomposition}
  Let $X$ be a Barsotti-Tate group over a scheme $S$.
  \begin{subenv}
   \item Assume we have a subalgebra $O_1 \times \ldots \times O_r \subset \End X$. Denote by $\epsilon_i$ the multiplicative unit in $O_i$ and let $X_i := \image \epsilon_i$. Then $X= X_1 \times \ldots \times X_r$ and the $X_i$ are \BT s over $S$.
   \item Assume there is a subalgebra $\M_d(O) \subset \End X$. Let
   \[
    \epsilon = \left( \begin{array}{ccc}
                       1 & 0 & \ldots \\
                       0 & 0 & \\
                       \vdots & & \ddots
                      \end{array} \right) \in \M_d(O)
   \]
   and $X' = \image \epsilon$. Then $X'$ is a \BT\ with $O$-action and $X \cong (X')^d$ compatible with $\M_d(O)$-action.
  \end{subenv}
 \end{lemma}
 \begin{proof}
  The first assertion of (1) is obvious. The sheaves $X_i$ are $p$-divisible and $p$-torsion because $X$ is. It remains to show that $X_i[p]$ is representable by  a finite locally free group scheme over $S$. Let $\epsilon_i' := \sum_{j \not= i} \epsilon_j$. Then $X_i = \ker \epsilon_i'$, in particular $X_i[p]$ is represented by a closed subgroup scheme of $X[p]$. Let $\Escr, \Escr_i$ be coherent $\Oscr_S$-algebras such that $X[p] = \Specfunc\Escr$ and $X_i[p] = \Specfunc\Escr_i$. Then the closed embedding $X_i[p] \mono X[p]$ induces a surjection $\Escr \epi \Escr_i$ and $\epsilon_{i |X[p]}:X[p] \epi  X_i[p]$ induces a splitting of this surjection. Thus $\Escr_i$ is a direct summand of $\Escr$ as $\Oscr_S$-module, in particular it is again locally free.

  Now the isomorphism $X \cong (X')^d$ is the standard Morita argument and the fact that $X'$ is a \BT\ can be shown by the same argument as above.
 \end{proof}

 \begin{corollary} \label{cor BTEL decomposition}
  \begin{subenv}
   \item Let $(B,O_B)$ be an unramified EL-datum and let $B= \prod B_i$ be a decomposition into simple factors and $O_{B_i} = O_B \cap B_i$. Let $\Xund$ be a \BT\ with $(B,O_B)$-EL structure. Then $\Xund$ decomposes as $\Xund = \prod \Xund_i$ where $\Xund_i$ is a \BT\ with $(B_i,O_{B_i})$-EL structure. This defines an equivalence of categories
   \[
    \left\{ \begin{array}{l}\textnormal{\BT s with} \\ (B,O_B) \textnormal{-EL structure} \end{array}\right\} \cong \prod_i \left\{\begin{array}{l} \textnormal{\BT s with} \\ (B_i,O_{B_i}) \textnormal{-EL structure} \end{array}\right\}
   \]
   that preserves isogenies.
   \item Let $(B,O_B)$ be an unramified EL-datum with $B \cong \M_d(F)$ simple. Let $\Xund$ be a \BT\ with $(B,O_B)$-EL structure. Then $O_B \cong \M_d(O_F)$ and $\Xund \cong (\Xund')^d$ where $\Xund'$ is a \BT\ with $(F,O_F)$-EL structure. This defines an equivalence of categories
   \[
    \left\{\begin{array}{l} \textnormal{\BT s with} \\ (B,O_B) \textnormal{-EL structure} \end{array}\right\} \cong \left\{ \begin{array}{l}\textnormal{\BT s with} \\ (F,O_F) \textnormal{-EL structure} \end{array}\right\}
   \]
   which preserves isogenies.
  \end{subenv}
 \end{corollary}

 \begin{definition}
  Let $X$ be a \BT\ with $(B,O_B)$-EL structure. We define the relative height of $X$ as the tuple
  \[
   \relht X = \left(\frac{\height X_1}{[B_1:\QQ_p]^{\frac{1}{2}}}, \ldots , \frac{\height X_r}{[B_r:\QQ_p]^{\frac{1}{2}}}\right)
  \]
  with $X_i, B_i$ as above.
 \end{definition}

 Let $(B,O_B, ^\ast)$ be a PEL-datum. Following \cite{RZ96}~ch.~A, we decompose $(B,^\ast) = \prod B_i$ where $B_i$ is isomorphic to one of the following.
 \begin{enumerate}[topsep = -3pt, parsep = 3pt]
  \item[(I)] $\M_d (F) \times \M_d(F)^{\rm opp}$ where $F$ is an unramified $p$-adic field and $(a,b)^\ast = (b,a)$.
  \item[(II\textsubscript{C})] $\M_d (F)$ with $F$ as above, $a^\ast = a^t$.
  \item[(II\textsubscript{D})] $\M_d (F)$ with $F$ as above, $a^\ast = J^{-1}a^\ast J$ with $J^tJ =-1$.
  \item[(III)] $\M_d (F)$ with $\QQ_p \subset F' \subset F$ finite unramified field extensions, $[F:F'] = 2$ and $a = \abar^t$ where $\overline{\cdot}$ denotes the non-trivial $F'$-automorphism of $F$, acting on $M_d(F)$ componentwise.
 \end{enumerate}
 We call the algebras with involution which are isomorphic to one of the above indecomposable. Note that we also have the analogous decomposition for $O_B$.
 
 \begin{definition}
  A PEL-datum $(B,O_B,\,^\ast)$ is of type (AC) if no factors of type (II\textsubscript{D}) appear in the decomposition of $(B,^\ast)$. 
 \end{definition}

 Now Lemma~\ref{lem BT decomposition} implies the following result.
 \begin{corollary} \label{cor polarised BT decomposition}
  \begin{subenv}
   \item Let $(B,O_B, ^\ast)$ be an unramified PEL-datum of type (AC) and let $(B, ^\ast) = \prod B_i$ be a decomposition into indecomposable factors. Let $\Xund$ be a \BT\ with $(B,O_B, ^\ast)$-PEL structure. Then $\Xund$ decomposes as $\Xund = \prod \Xund_i$ where $\Xund_i$ is a \BT\ with $(B_i,O_{B_i}, ^\ast)$-PEL structure. This defines an equivalence of categories
   \[
    \left\{\begin{array}{l} \textnormal{\BT s with} \\ (B,O_B,\,^\ast) \textnormal{-PEL structure}\end{array}\right\} \cong \prod_i \left\{\begin{array}{l}\textnormal{\BT s with} \\ (B_i,O_{B_i},\,^\ast) \textnormal{-PEL structure} \end{array}\right\}
   \]
   and also a bijection of isogenies on the left hand side with tuples of isogenies on the right hand side which have the same similitude factor.
   \item Let $(B,O_B, ^\ast)$ be an unramified PEL-datum of type (AC) with $(B, ^\ast)$ indecomposable and let $\Xund$ be a \BT\ with $(B,O_B, ^\ast)$-PEL structure. Using the notation above, we may describe $\Xund$ as follows. \smallskip \\
    (I) $\Xund \cong (\Xund')^d \times (\Xund'^\vee)^d$ where $X'$ is a \BT\ with $(F,O_F)$-EL structure and the polarisation is given by $\lambda (a,b) = (b,-a)$. \smallskip \\
    (II\textsubscript{C}) $\Xund \cong (\Xund')^d$ where $X'$ is a \BT\ with $(F,O_F,\id)$-PEL structure. \smallskip  \\
    (III) $\Xund \cong (\Xund')^d$ where $X'$ is a \BT\ with $(F,O_F,\bar\cdot)$-PEL structure. \smallskip \\
   We obtain equivalences of categories
    \begin{flalign*} 
     (I) & \quad \left\{\begin{array}{l} \textnormal{\BT s with}\\ (B,O_B,\,^\ast) \textnormal{-PEL structure}\end{array}\right\} \cong \left\{\begin{array}{l} \textnormal{\BT s with} \\ (F,O_F) \textnormal{-EL structure} \end{array} \right\} & \\
     (II_C) & \quad \left\{ \begin{array}{l} \textnormal{\BT s with} \\ (B,O_B,\,^\ast) \textnormal{-PEL structure}\end{array}\right\} \cong \left\{\begin{array}{l} \textnormal{\BT s with} \\ (F,O_F,\id) \textnormal{-PEL structure}\end{array} \right\} & \\
     (III) & \quad \left\{ \begin{array}{l} \textnormal{\BT s with} \\ (B,O_B,\,^\ast) \textnormal{-PEL structure}\end{array} \right\} \cong \left\{\begin{array}{l} \textnormal{\BT s with} \\ (F,O_F,\bar\cdot) \textnormal{-PEL structure}\end{array} \right\}
    \end{flalign*}
   and in the cases (II\textsubscript{C}) and (III) also a bijection of the sets of isogenies. This is also true in case (I) if one fixes the similitude factor.
  \end{subenv}
 \end{corollary}
 \begin{proof}
  Write $X_i = \image \epsilon_i$ as in Lemma \ref{lem BT decomposition}. Let $\Xund = (X,\iota,\lambda)$. Then
  \[
   \image \lambda_{|X_i} = \image (\lambda \circ \epsilon_i) =  \image (\epsilon_i^\ast \circ \lambda) = \image (\epsilon_i \circ \lambda) = X_i,
  \]
  which proves (1).

  For the second assertion assume first that $(B,\,^\ast)$ is of type (II\textsubscript{C}) or (III). Let $\epsilon_{i,j}$ be the matrix with $1$ as $(i,j)$-th entry and $0$ otherwise and denote $X_i := \image (\epsilon_{i,i})$. Then we have by Lemma~\ref{lem BT decomposition} that $X = \prod X_i$ with $O_F$-equivariant isomorphisms $\epsilon_{i,j}: X_j \isom X_i$. Now the same argument as in (1) shows $\image \lambda_{|X_i} = X_i$ (consider $\epsilon_{i,i}$) and $X_i \cong X_j$ as {\BTPELs} (consider $\epsilon_{i,j}$).  If $(B,\,^\ast)$ is of type (I) let $\epsilon_0$ and $\epsilon_1$ be the units of the $\M_d(O_F)$-factors. Decompose $X = X_0 \oplus X_1$ as in  Corollary~\ref{cor BTEL decomposition}. Then
  \[
   \image \lambda_{|X_i} = \image (\lambda \circ \epsilon_i) = \image (\epsilon_i^\ast \circ \lambda) = \image(\epsilon_{1-i} \circ \lambda) = X_{1-i}
  \]
  Now the claim follows as in cases (II\textsubscript{C}) and (III).
 \end{proof}

 \subsection{Dieudonn\'e theory for {\BT}s with additional structure} \label{ss isocrystals}
 Given an {\BTpEL} $X$ over $k$, we denote by $M(X)$ resp.\ $N(X)$ the Dieudonn\'e module, resp.\ rational Dieudonn\'e module. The (P)EL structure on $X$ induces additional structure on $M(X)$ an d$N(X)$, which induces the following definition.

 \begin{definition}
  \begin{subenv}
   \item Let $\Bscr = (B,O_B,(^\ast))$ be a (P)EL-datum. A $\Bscr$-isocrystal is a pair $(N,\Phi)$ such that
   \begin{itemlist}
    \item $N$ is a finite-dimensional $L_0$-vector space with $B$-action.
    \item In the PEL case $N$ is equipped with  a perfect anti-symmetric $B$-compatible $\QQ_p$-linear pairing $\pair$.
    \item $\Phi:N \to N$ is a $\sigma$-semilinear bijection which commutes with $B$-action and in the PEL-case satisfies
    \begin{equation} \label{term polarisation isocrystal}
     \langle \Phi(v),\Phi(w) \rangle = p^a \langle v,w \rangle
    \end{equation}
    for some fixed integer $a$.
   \end{itemlist}
   \item Let $N = \prod N_i$ be the decomposition of $N$ induced by the decomposition $B = \prod B_i$ into simple algebras. Then the tuple $\relht (N,\Phi) := (\frac{\dim_L N_i}{[B_i:\QQ_p]^{\frac{1}{2}}})_i$ is called relative height of $(N,\Phi)$.
   \item A $\Bscr$-isocrystal $(N,\Phi)$ is called relevant, if there exists an  $O_B$-stable lattice $\Lambda \subset N$ which in the PEL-case is self-dual w.r.t.\ some representative of $\pair$.
  \end{subenv}
 \end{definition}
 
 Now given $\Bscr$-{\BTpEL} $X$ of relative height $\nbf$, the associated isocrystal $N(X)$ is a relevant $\Bscr$-isocrystal and $M(X) \subset N(X)$ is a (self-dual) $O_B$-lattice. In order to associate $\sigma$-conjugacy class to $N(X)$, we need the following lemma.

 \begin{lemma} \label{lem isocrystal}
  \begin{subenv}
   \item Assume $\Bscr$ is an EL-datum and let $(N,\Phi)$ be a $\Bscr$-isocrystal. Then there exists a unique $B$-module $V$ such that $N \cong V \otimes L$. In particular $(N,\Phi)$ is relevant, $\Lambda$ may be chosen to be defined over $\ZZ_p$ and is unique up to $O_B$-linear isomorphism.
   \item Let $\Bscr$ be a PEL-datum of type (AC) and let $(N,\Phi)$ be a relevant $\Bscr$-isocrystal. Then there exists a unique $B$-module $V$ with $\QQ_p^\times$-homogeneous polarisation, depending only on the relative height of $(N,\Phi)$, such that $N \cong V \otimes L$. Furthermore, $\Lambda$ can be chosen to be defined over $\ZZ_p$ and is unique up to $O_B$-linear isometry.
  \end{subenv}
 \end{lemma}
 \begin{proof}
  For the proof of part (1) we may assume that $B=F$ is a finite unramified field extension of $\QQ_p$ by Morita equivalence. Then $F\otimes L \cong \prod_{\tau:F\mono L} L$ inducing a decomposition $N \cong \prod_{\tau:F\mono L} N_\tau$. As $\Phi$ induces a $\sigma$-linear bijection $N_\tau \to N_{\sigma \circ \tau}$, we have that all $N_\tau$ have dimension $\relht (N,\Phi)$ over $L$, thus $N \cong L \otimes F^{\relht (N,\Phi)}$. Obviously $O_L \otimes O_F^{\relht (N,\Phi)}$ is an $O_F$-stable lattice in $N$ and any $O_F$-stable lattice of $N$ is isomorphic to it. Now part (2) follows from part (1) and \cite{kottwitz92}~Lemma~7.2 and Remark~7.5.
 \end{proof}

 In particular the underlying (polarised) $B$-module of a relevant $\Bscr$-isocrystal is uniquely determied by its relative height. Thus fixing a (polarised) left-$B$-module $V$ and a (self-dual) $O_B$-stable lattice $\Lambda \subset V$, we obtain a bijection
 \begin{eqnarray*}
  B(G) &\to& \{ \textnormal{relevant } \Bscr-\textnormal{isocrystals of relative height } \nbf\}/\cong \\
  \left[ b \right] &\mapsto& (V_L,b\sigma),
 \end{eqnarray*}
 where $G$ denotes the group scheme of $O_B$-linear automorphisms (similitudes) of $\Lambda$ and $\nbf = (\frac{\dim_{\QQ_p} V_i}{[B_i:\QQ_p]^{\frac{1}{2}}})_i$.
 We warn the reader that in the PEL-case the above bijection requires twisting $\pair$ by a scalar, which is unique up to $\QQ_p^\times$, so that (\ref{term polarisation isocrystal}) is satisfied (see \cite{RZ96}~1.38). This scalar would not necessarily exist if $k$ was only perfect and not algebraically closed. Combining this bijection with the Dieudonn\'e functor, we obtain an injective map
 \[
  \left\{ \begin{array}{c} \textnormal{\BT s over } k \textnormal{ with } \Bscr-\textnormal{(P)EL structure} \\ \textnormal{of relative height } \nbf  \textnormal{ up to isogeny} \end{array}\right\} \mono B(G).
 \]
 
 In the case $k = \FFbar_p$ denote by $C(G)$ the $G(O_L)$-$\sigma$-conjugacy classes in $G(L)$. By a similar argument one gets an injection
 \[
  \{\textnormal{\BT s over } \FFbar_p \textnormal{ with } \Bscr-\textnormal{(P)EL structure of relative height } \nbf\} \mono C(G)
 \]
 where in the PEL case the polarisation is meant to be defined up to $\ZZ_p^\times$-scalar. For $b \in G(L)$ we denote by $\pot{b}$ its $G(O_L)$-$\sigma$-conjugacy class.

 \subsection{Rapoport-Zink spaces} \label{ss RZ spaces}

 We restrict the definition of a Rapoport-Zink space in \cite{RZ96} to our case, that is with unramified (P)EL-datum and hyperspecial level structure.

 \begin{definition} \label{def RZ datum}
  \begin{subenv}
   \item An unramified Rapoport-Zink datum of type EL is a tuple $(B,O_B,V,\mu_{\QQbar_p},\Lambda,[b])$ where
   \begin{itemlist}
    \item $\Bscr = (B,O_B)$ is an EL-datum.
    \item $V$ is a finite left-$B$-module,
    \item $\Lambda$ is an $O_B$-stable lattice in $V$,
    \item $\mu_{\QQbar_p}$ is a conjugacy class of cocharacters $\mu_{\QQbar_p}: \GG_{m,\QQbar_p} \to G_{\QQbar_p}$, where $G = \underline{\Aut}_{O_B} (\Lambda)$.
    \item $[b] \in B(G,\mu_{\QQbar_p})$
   \end{itemlist}
   which satisfy the conditions given below.
   \item An unramified Rapoport-Zink datum of type PEL is a tuple $\hat\Dscr = (B,O_B,\ast,V,\pair, \mu_{\QQbar_p},\Lambda,[b])$ where
   \begin{itemlist}
    \item $\Bscr = (B,O_B,^\ast)$ is a PEL-datum,
    \item $V$ is a finite left-$B$-module with a $B$-adapted symplectic form $\pair$ on the underlying $\QQ_p$-vector space,
    \item $\Lambda$ is an $O_B$-stable self-dual lattice in $V$,
    \item $\mu_{\QQbar_p}$ is a conjugacy class of cocharacters $\mu_{\QQbar_p}: \GG_{m,\QQbar_p} \to G_{\QQbar_p}$, where $G$ is the linear algebraic group with $R$-valued points (for every $\ZZ_p$-algebra $R$)
    \begin{eqnarray*}
     G(R) &=& \{g\in \underline{\Aut}_{O_B}(\Lambda)(R) \mid \langle g(v_1),g(v_2) \rangle = c(g) \langle v_1,v_2 \rangle \textnormal{ for some } c(g) \in R^\times\}
    \end{eqnarray*}
    \item $[b] \in B(G,\mu_{\QQbar_p})$.
   \end{itemlist}
   which satisfy the conditions below.
   \item The field of definition $E$ of $\mu_{\QQbar_p}$ is called the local Shimura field.
   \item A Rapoport-Zink datum as above is called superbasic if $[b]$ is superbasic.
  \end{subenv}
 \end{definition}

 Let $(V_L,\Phi) := (L \otimes_{\QQ_p} V, b(\sigma \otimes \id))$ denote the $\Bscr$-isocrystal associated to $[b]$.  The additional conditions we impose on the data are the following.
 \begin{assertionlist}
  \item $G$ is connected.
  \item The Newton slopes of $(V_L,\Phi)$ are in $[0,1]$.
  \item The weight decomposition of $V_{L}$ w.r.t.\ $\mu_{\QQbar_p}$ contains only the weights $0$ and $1$,
  \item In the PEL-case, we have $c(\mu(p)) = p$.
 \end{assertionlist}
 The first condition implies that $G$ is a reductive group scheme, in particular $E$ is a finite unramified field extension of $\QQ_p$. We fix a Killing pair $T \subset B$ of $G$. Denote by $\mu$ the dominant element in $\mu_{\QQbar_p}$ and let $V_E = V^0 \oplus V^1$ be the decomposition into weight spaces w.r.t. $\mu$.

 Now the $\Bscr$-isocrystal $(V,\Phi)$ gives rise to an isogeny class of \BTpELs. We fix an element $\underline{\XX}$ of this isogeny class.

 \begin{definition} \label{def RZ space}
  Let $Nilp_{O_L}$ denote the full subcategory of schemes over $O_L$ on which $p$ is locally nilpotent. Then the Rapoport-Zink space is the functor $\Mscr_G(b,\mu): Nilp_{O_L} \rightarrow Sets, S \mapsto \Mscr_G(b,\mu)(S)$ given by the following data up to isomorphism.
  \begin{itemlist}
   \item A {\BTpEL} $\underline{X}$ over $S$ with associated (P)EL-datum $\Bscr$. Moreover the Kottwitz determinant condition holds.
   \item An isogeny $\underline{\XX}_{\overline{S}} \to \underline{X}_{\overline{S}}$, where $\overline{S}$ denotes the closed subscheme of $S$ defined by $p\Oscr_S$.
  \end{itemlist}
 \end{definition}
 
 \begin{example}
  Let $\Dscr= (\Bsf, ^{\ast},\Vsf,\pair,\Osf_\Bsf,\Lambda,h)$ be a PEL-Shimura datum and $[b]$ be an isogeny class of \BTDs. Then $(\Bsf_{\QQ_p},\,^\ast,\Vsf_{\QQ_p},\pair,\Osf_{\Bsf,p},\Lambda,[\mu_h],[b])$ is an unramified Rapoport-Zink datum and $\Mscr_G(b,\mu)$ is the moduli functor parameterizing {\BTDs} inside the isogeny class $[b]$.
 \end{example}

 \begin{theorem}[\cite{RZ96}, Thm.~3.25]
  The functor $\Mscr_G(b,\mu)$ is representable by a formal scheme locally of finite type. 
 \end{theorem}
  
 \begin{notation}
  From now on we denote by $\Mscr_G(b,\mu)$ the \emph{underlying reduced subscheme} of the moduli space above. 
 \end{notation}

 By \cite{RZ96}~\S~3.23(b) the Kottwitz determinant condition can be reformulated as follows. Let $B = \prod B_i = \prod_i \M_{d_i}(F_i)$ be a decomposition of $B$ into simple factors inducing $X = \prod X_i$ and $V^1 = \prod V_{i}^1$. Furthermore we decompose $\Lie X$ and $V_{i}^1$ according to the $O_{F_i}$- resp.\ $F_i$-action.
 \begin{eqnarray*}
  \Lie X_i &=& \prod_{\tau: F_i \mono L} (\Lie X_i)_\tau \\
  V_{i}^1 &=& \prod_{\tau: F_i \mono L} V_{i,\tau}^1 
 \end{eqnarray*}
 Now $X$ satisfies the determinant condition if and only if $(\Lie X_i)_\tau$ has rank $\dim_{\QQ_p} V_{i,\tau}^1$ for all $i,\tau$.

 We obtain that by Dieudonn\'e theory the $k$-valued points of $\Mscr_G(b,\mu)$ correspond to $O_B$-stable lattices in $V_L$ which are self-dual up to $L^\times$-scalar and satisfy $p\Lambda \subset b\sigma(\Lambda) \subset \Lambda, \dim_k (\Lambda/b\sigma(\Lambda))_\tau = \dim V_{i,1,\tau}$. By \cite{kottwitz92}~Cor.~7.3 and Rem.~7.5 the group $G(L)$ acts transitively on the set of $O_B$-stable lattices in $V_L$ which are self-dual up to a constant, thus
 \begin{eqnarray*}
  \Mscr_G(b,\mu)(k) &=& \{g \Lambda \mid g \in G(L), gb\sigma(g)^{-1} \in G(O_L)\mu(p)G(O_L)\} \\
  &=& \{g \in G(L)/G(O_L) \mid  gb\sigma(g)^{-1} \in G(O_L)\mu(p)G(O_L)\}.
 \end{eqnarray*}
 
 \subsection{Comparison of PEL structure and $\Dscr$-structure} \label{ss PEL vs D}
 The aim of this subsection is to show the following proposition.

 \begin{proposition} \label{prop PEL vs D}
  Any {\BTD} is a {\BTPEL} of type $(AC)$ with $\ZZ_p^\times$-homogeneous polarisation. On the other hand, for every {\BTPEL} of type (AC) with $\ZZ_p^\times$-homogeneous polarisation  over a connected $\FFbar_p$-scheme there exists an unramified PEL-Shimura datum $\Dscr'$ such that the PEL structure defines a $\Dscr'$-structure.
 \end{proposition}

 The construction of the PEL-Shimura datum associated to the {\BTPEL} consists of two steps. First we construct a local datum similar to a Rapoport-Zink datum and then we show that it can be obtained by localizing a PEL-Shimura datum. The first step also works for Barsotti-Tate groups with EL structure and will be needed later. Thus we include this case in our construction.

 Let $\Xund$ be a {\BTEL} or with PEL structure of type (AC). As in section \ref{ss isocrystals} we associate $(V,\Lambda)$ and $G$ to $\Xund$. It remains to construct a $G(\QQbar_p)$-conjugacy class of cocharacters such that $\Xund$ satisfies the Kottwitz determinant condition. For this we give an explicit description of the group $G$.
 
 Assume first that we are in the EL-case. Let $B=\prod B_i$ be a decomposition into simple factors, inducing a decomposition $\Lambda = \prod \Lambda_i$ into $O_{B_i}$-modules. Then
 \[
  G = \prod_i \GL_{O_{B_i}} (\Lambda_i),
 \]
 thus it suffices to give a description of $G$ in the case where $B$ is simple. By Morita equivalence we may furthermore assume that $B=F$ is a finite unramified field extension of $\QQ_p$. By choosing an $O_F$-basis $e_1,\ldots,e_n$ of $\Lambda$, we obtain an isomorphism
 \[
  G \cong \Res_{O_F/\ZZ_p} \GL_n.
 \]
 We denote $\GL_{O_F,n} = \Res_{O_F/\ZZ_p} \GL_n$. 

 If $\Xund$ is a {\BTPEL} of type (AC) a similar construction works. Let $(B,\, ^\ast) = \prod B_i$ be the decomposition into indecomposable factors and $\Lambda = \prod_i \Lambda_i$ accordingly. We denote by $G_i$ the group schemes of similitudes of $\Lambda_i$. Then
 \[
  G = (\prod_i G_i)^1 := \{(g_i)_i \in \prod_i G_i \mid c(g_i) = c(g_j) \textnormal{ for all } i,j\}.
 \]
 Hence it suffices to give a description of $G$ in the case where $(B,^\ast)$ is indecomposable. Using the notation of section~\ref{ss normal forms of BTpELs} we may replace $\pair$ by a $\M_d(F)$-hermitian form on $V$ without changing $G$ (\cite{knus91}~Thm.~7.4.1). Furthermore we assume $d=1$ by hermitian Morita equivalence (\cite{knus91}~Thm.~9.3.5), i.e.\ $B =F \times F$ in case (I) and $B=F$ in cases (II\textsubscript{C}) and (III).

 We first consider case (I). We decompose $\Lambda = \Lambda' \oplus \Lambda''$ with the first resp.\ second factor of $O_F \times O_F$ acting on $\Lambda'$ resp.\ $\Lambda''$. Then $\pair$ induces an isomorphism $\Lambda' \cong (\Lambda'')^\vee$ of $O_F$-modules. Now
 \begin{eqnarray*}
  G &=& \{(g',g'') \in \underline{\Aut} (\Lambda') \times \underline{\Aut} (\Lambda''); g'' = c(g)\cdot g'^{-t}\} \\
  &\cong&  \underline{\Aut} (\Lambda') \times \GG_m \\
  &\cong& \GL_{O_F,n} \times \GG_m.
 \end{eqnarray*}
 Here $\GG_m$ parametrizes the similitude factor. By choosing dual bases of $\Lambda'$ and $\Lambda''$, we may also consider $G$ as a closed subgroup of $\GL_{O_F,n} \times \GL_{O_F,n}$.  

 Now assume that $(B,\,^\ast)$ is of the form (II\textsubscript{C}). By the uniqueness assertion of Lemma \ref{lem isocrystal} there exists an $O_F$-basis $x_1, \ldots, x_n$ of $\Lambda$ such that
 \[
  \langle x_i , x_j \rangle = \delta_{i,n+1-j} \textnormal{ for } i \leq j.
 \]
 This identifies $G$ with the closed subgroup $\GSp_{O_F,n}$ of $\GL_{O_F,n}$, which is given by
 \[
  \GSp_{O_F,n}(R) = \{g \in \GL_n(F\otimes R)\mid gJ_a g^t = c(g)J_a \textnormal{ for some } c(g)\in R^\times\},
 \]
 where $J_a$ denotes the matrix
  \begin{center}
   \begin{tikzpicture}
    \node (lower entry) at (0,0) {-1} ;
    \node (lower central entry) at (1,1) {-1} ;
    \node (upper central entry) at (1.4,1.4) {1} ;
    \node (upper entry) at (2.4,2.4) {1} ;

    \node[left delimiter = (, right delimiter = ),fit=(lower entry) (upper entry)] {};

    \draw  (lower entry.north east) -- (lower central entry.south west) ;
    \draw  (upper central entry.north east) -- (upper entry.south west) ;

    \draw[decorate,decoration=brace] let \p1=(lower entry.west), \p2=(lower central entry) in (\x2,-0.5) -- (\x1,-0.5) node[midway,yshift=-10] {\footnotesize $\frac{n}{2}$} ;
    \draw[decorate,decoration=brace] let \p1=(upper central entry), \p2=(upper entry.east) in (\x2,-0.5) -- (\x1,-0.5) node[midway,yshift=-10] {\footnotesize $\frac{n}{2}$} ;
   \end{tikzpicture}
  \end{center}

 If $(B,^\ast)$ is of type (III), we choose a basis $x_1,\ldots,x_n$ of $\Lambda$ such that
 \[
  \langle x_i,x_j\rangle = u\cdot\delta_{i,n+1-j}
 \]
 where $u \in O_F^\times$ with $\sigma_{F'}(u) = -u$. This defines a closed embedding of $G$ into $\GL_{O_F,n}$ with image
 \[
  \GU_{O_F,n}(R) = \{g \in \GL(F\otimes R)\mid gJ \sigma_{F'}(g)^t = c(g)J \textnormal{ for some } c(g)\in R^\times\}.
 \]
 Here $J$ denotes the matrix with ones on the anti-diagonal and zeros anywhere else.
  
 We are interested in the $G(\QQbar_p)$-conjugacy classes of $\Hom(\GG_m,G_{\QQbar_p})$. As $G$ is reductive, the conjugacy classes are in canonical one-to-one correspondence with the subset of dominant elements $X_*(T)_\dom$ of $X_*(T)$ for some choice $T\subset B \subset G$ of a maximal torus $T$ and a Borel subgroup $B$. For $G = \GL_{O_F,n}, G=\GSp_{O_F,n}, G= \GU_{O_F,n}$ we choose $T$ to be the diagonal torus and $B$ to be the subgroup of upper triangular matrices. In case (I) take $T$ and $B$ to be induced by our choice for $\GL_{O_F,n}$ and the isomorphism $G \cong \GL_{O_F,n} \times \GG_m$. With respect to the canonical embedding $G \mono \GL_{O_F,n} \times \GL_{O_F,n}$ this means that $B$ denotes the Borel subgroup of pairs $(g',g'')$ where the first factor is an upper triangular matrix (or equivalently the second factor is a lower triangular matrix).
 
 Using the canonical identification $X_*(T) \cong \prod \ZZ^n$ in the $\GL_{O_F,n}$ case we obtain in the EL-case
 \[
  X_*(T)_{\dom} = \{ \mu \in \prod_{\tau:F\mono\QQ_p} \ZZ \mid \mu_{\tau,1} \geq \ldots \geq \mu_{\tau.n}\} 
 \]
 and in the PEL case
 \begin{eqnarray*}
  \textnormal{(I)} \quad X_*(T)_\dom &=& \{ \mu= (\mu',\mu'') \in (\prod_\tau \ZZ^n)^2\mid \mu'_{\tau,1} \geq \ldots \geq \mu'_{\tau, n}, \mu_i'' = c(\mu) - \mu_{n+1-i}', c(\mu) \in \ZZ\}. \\
  \textnormal{(II\textsubscript{C})} \quad X_*(T)_\dom &=& \{ \mu \in \prod_\tau \ZZ^n\mid \mu_{\tau,1} \geq \ldots \geq \mu_{\tau, n},\; \mu_{\tau,i}+\mu_{\tau,n+1-i} = c(\mu) \textnormal{ with } c(\mu) \in \ZZ\}. \\
  \textnormal{(III)} \quad X_*(T)_\dom &=& \{ \mu \in \prod_\tau \ZZ^n\mid \mu_{\tau,1} \geq \ldots \geq \mu_{\tau, n},\; \mu_{\tau,i}+\mu_{\tau+\sigma_{F'},n+1-i} = c(\mu) \textnormal{ with } c(\mu) \in \ZZ\}.
 \end{eqnarray*}
 For details, see the appendix.

 \begin{lemma} \label{lem mu}
  Let $\Xund$ be a {\BTEL} or PEL structure of type (AC) and $V,\Lambda,G$ be as above. There exists a (unique) $[\mu]$ with weights $0$ and $1$ on $V_{\QQbar_p}$ such that $\Xund$ satisfies the determinant condition, i.e.
  \[
   \cha(\iota(a)|\Lie X) = \cha(a|V^1)
  \]
 \end{lemma}
 \begin{proof}
  We start with the EL-case. By the virtue of section~\ref{ss normal forms of BTpELs} we may assume without loss of generality that $B=F$ is an unramified field extension of $\QQ_p$. Recall that the determinant condition is equivalent to
  \[
   \dim (\Lie X)_\tau = \dim_{\QQ_p} V_{\tau}^1 = \#\{i\mid \mu_{i,\tau} = 1\}.
  \]
  Thus $\mu = ((\underbrace{1,\ldots,1,}_{\dim (\Lie X)_\tau} 0,\ldots,0))_{\tau:F \mono \QQbar_p}$.

  In the PEL-case assume without loss of generality that $(B,\,^\ast)$ is indecomposable and that $B=F$ resp.\ $B = F \times F$ in case (I). We exclude case (I) for the moment. Denote by $n$ be the relative height of $\Xund$. As in the EL-case, we need $\mu$ to satisfy
  \[
   \dim (\Lie X)_\tau = \#\{i\mid \mu_{i,\tau} = 1\}.
  \]
  Now use the argument given in \cite{RZ96}~\S~3.23(c). The above condition is open and closed on $S$ and local for the \'etale topology. Hence we assume that $S = \Spec k$ is the spectrum of an algebraically closed field of characteristic $p$.  Let $E(X)$ be the universal extension of $X$. By functoriality we obtain an $O_B$-action on $E(X)$ and an isomorphism $E(X) \isom E(X^\vee)$. Now by Lemma~\ref{lem isocrystal} the crystal induced by $X$ over $k$ is a free $W(k) \otimes O_F$-module, thus the value of this crystal at $\Spec k$ is a free $k \otimes O_F$-module. The value at $\Spec k$ is the Lie-algebra of the universal extension $E(X)$ of $X$. Thus
  \[
   \dim (\Lie E(X))_\tau = n.
  \]
  Now we have a short exact sequence of $O_F \otimes k$-modules
  \begin{equation} \label{term universal extension}
   0 \to (\Lie X^\vee)^\vee \to \Lie E(X) \to \Lie X \to 0.
  \end{equation}
  Thus
  \[
   \dim (\Lie X)_\tau = \dim((\Lie X^\vee)^\vee)_\tau = \left\{ \begin{array}{ll}
                                                                 n - \dim (\Lie X)_\tau & \textnormal{ in case (II\textsubscript{C})} \\
								 n - \dim (\Lie X)_{\tau+\sigma_{F'}} & \textnormal{ in case (III)}
                                                                \end{array}\right.,
  \]
  which is equivalent to the constraint $c(\mu) = 1$.

  In case (I) let $V = V' \oplus V''$ and $\Lie X = (\Lie X)' \oplus (\Lie X)$ be induced by the decomposition of $B$ resp.\ $O_B$ into two factors. The relative height of $\Xund$ is of the form $(n,n)$. The determinant condition is equivalent to
  \begin{eqnarray*}
   \dim (\Lie X)'_\tau &=& {V'_{\tau}}^1 \\
   \dim (\Lie X)''_\tau &=& {V''_{\tau}}^1.
  \end{eqnarray*}
  Now the side condition is $\dim {V'_{\tau}}^1 = \dim {V''_{\tau}}^0$, or equivalently $c(\mu) = 1$. Indeed by (\ref{term universal extension}) we have
  \[
   \dim (\Lie X)''_\tau + \dim (\Lie X)'_\tau = n.
  \]
 \end{proof}

 \begin{proof}[Proof of Proposition~\ref{prop PEL vs D}]
  We have associated a datum $(B,\,^\ast,V,\pair,O_B,\Lambda,[\mu])$ to $\Xund$ so far. It remains to show that there exists a PEL-Shimura datum $(\Bsf,\,^\ast,\Vsf,\pair',\Osf_{\Bsf},\Lambda,h)$ such that 
  \begin{itemlist}
   \item $(B,^\ast) \cong (\Bsf_{\QQ_p},^\ast)$, inducing $O_B = \Osf_{\Bsf,p}$.
   \item $(V,\pair) \cong (\Vsf_{\QQ_p},\pair')$ inducing $\Lambda \cong \Lambda'$, $[\mu] = [\mu_h]$.
  \end{itemlist}
  This is proven in   \cite{VW13}~Lemma~10.4. In the other direction, the assertion that the PEL structure induced by $\Dscr$ is of type (AC) is well-known to be equivalent to the condition that $\Gsf$ is connected. 
 \end{proof}

 \section{The Newton stratification} \label{sect newton stratification}

 \subsection{Comparison between the Newton stratifications on $\Ascr_0$ and on $\Sscr_{\Xund}$} \label{ss comparision}
 Let $\Xund$ be a {\BTEL} or PEL structure of type (AC) of fixed relative height over an $\FF_p$-scheme $S$. As in the previous section this yields an unramified reductive group scheme $G$ over $\ZZ_p$  and we choose $T \subset B \subset G$ as above. Now for every geometric point $\sbar$ of $S$ the isogeny class of $\Xund_{\sbar}$ defines a $\sigma$-conjugacy class of $G_{\QQ_p}$ and thus $\Xund$ induces functions
 \begin{eqnarray*}
  b_{\Xund}: \{\textnormal{ geometric points of } S \} &\to& B(G) \\
  \nu_{\Xund}: \{\textnormal{ geometric points of } S \} &\to& X_*(T)_{\QQ,\dom}^\Gamma \\
  \kappa_{\Xund}:  \{\textnormal{ geometric points of } S \} &\to& \pi_1(G)_\Gamma.
 \end{eqnarray*}

 \begin{lemma}[\cite{RR96}, Thm.~3.6]\label{lem RR}
  The map $b_{\Xund}$ is lower semi continuous.
 \end{lemma}

 \begin{definition}
  For $\Xund$ as above and $b \in B(G)$ we consider
  \begin{eqnarray*}
   S^{b} &:=& \{s \in S; b_{\Xund} (\sbar) = b \} \\
   S^{\leq b} &:=& \{s \in S; b_{\Xund} (\sbar) \leq b\}
  \end{eqnarray*}
  as locally closed subschemes with reduced structure. The $S^b$ are called Newton strata and the $S^{\leq b}$ are called closed Newton strata.
 \end{definition}

 Now the isogeny classes of {\BTDs} correspond to $B(G,\mu)$ (see for example \cite{VW13}~ch.~8). Hence $\kappa$ is locally constant on $\Ascr_0$, thus it suffices to consider the Newton points to distinguish the Newton strata. We will also write $\Ascr_0^{\nu(b)}$ instead of $\Ascr_0^b$ whenever it is convenient. 
 
 \begin{proposition} \label{prop relation}
  Assume Theorem 1.1 holds true. Let $\Xund$ be a {\BTEL} or PEL structure of type (AC) over a perfect field $k_0$; denote by $G$ the associated reductive group over $\ZZ_p$ denote by $b_0 \in B(G)$ the isogeny class of $\Xund_{\kbar_0}$. For any $b \in B(G)$ which corresponds to an isogeny class of {\BTpELs} the following assertions hold. 
  \begin{subenv}
   \item $\Sscr_{\Xund}^b$ is non-empty if and only if $b \geq b_0$.
   \item If $\Sscr_{\Xund}^{\leq b}$is non-empty then it is equidimensional of dimension
   \[
    \langle \rho_{G}, \mu+\nu(b) \rangle - \frac{1}{2} \defect (b).
   \]
   \item $\Sscr_{\Xund}^{\leq b}$ is the closure of $\Sscr_{\Xund}^{ b}$ in $\Sscr_{\Xund}$.
  \end{subenv}
 \end{proposition}
 \begin{proof}
  As mentioned in section~\ref{ss deformation}, we may assume in the PEL case that the polarisation is $\ZZ_p^\times$-homogeneous. In the EL case we may replace $\Xund$ by $\Xund \times \Xund^\vee$ with the obvious $\ZZ_p^\times$-homogeneous PEL structure by Corollary~\ref{cor polarised BT decomposition}. We note that the similitude factor is constant modulo $\ZZ_p^\times$, as it is determined by the Kottwitz point, which is constant on $\Sscr_{\Xund}$ by \cite{RR96}~Thm.~3.8. Also, we may replace $\Xund$ by $\Xund_{\kbar_0}$ and thus assume that $k_0$ is algebraically closed. 

  Now by Proposition~\ref{prop PEL vs D} $\Xund$ is a {\BT} with $\Dscr'$-structure for a suitable unramified PEL-Shimura datum $\Dscr'$. Thus by \cite{VW13}~Thm.~10.2 there exists $x \in \Ascr_{\Dscr',0}(k_0)$ such that the associated {\BTD} is isomorphic to $\Xund$. Now the assertions follows from Theorem~\ref{thm dimension shimura} by applying Proposition~\ref{prop serre-tate} to $x$.
 \end{proof}

 \subsection{The Newton polygon stratification} 
 Our medium-term goal is to generalize the following improvement of de Jong-Oort's purity theorem by Yang to \BTpELs. It will be our main tool to compare the two assertions of Theorem~\ref{thm dimension shimura}.
 
 \begin{proposition}[cf.~\cite{yang11}~Thm.~1.1] \label{prop yang}
  Let $S$ be a locally noetherian connected $\FF_p$-scheme and $X$ be a \BT\ over $S$. If there exists a neighbourhood $U$ of a point $s\in S$ such that the Newton polygons $\NP(X)(x)$ of $X$ over all points $x \in U\setminus \overline{\{s\}}$ have a common break point, then either $\codim_U(\overline{\{s\}}) \leq 1$ or $NP(X)(s)$ has the same break point.
 \end{proposition}

 \begin{remark}
  The original formulation of this theorem gives the assertion for arbitrary $F$-crystals. For simplicity, we will work with \BT s instead. However, we remark that our generalisation also works in the setting of $F$-crystals.
 \end{remark}

 We will generalise the above proposition by comparing the Newton stratification on $S$ to a stratification given by a family of Newton polygons. We consider the following stratification.

 \begin{definition}
  Let $\Xund$ be a {\BTpEL} over a connected $\FF_p$-scheme $S$ and $(B,O_B,(^\ast))$ the associated (P)EL-datum. Let $B = \prod B_i$ be the decomposition into simple algebras and $X = \prod \Xund_i$ be as in Lemma \ref{lem BT decomposition}. Denote by $\NP(\Xund) = (\NP(\Xund_i))_i$ the family of (classical) Newton polygons associated to the $\Xund_i$. We call the decomposition of $S$ according to the invariant $\NP(\Xund)$  the Newton polygon stratification.
 \end{definition}
 
 \begin{remark}
  One might also think of considering the stratification given by the Newton polygon of $X$, i.e. the stratification given by of the isogeny class of $X_{\sbar}$ (forgetting the (P)EL-structure) for geometric points $\sbar$ of $S$. We warn the reader that this stratification is in general coarser than the Newton stratification.
 \end{remark}

 Now the aim of this subsection is to show that the Newton polygon stratification and the Newton stratification on $S$ coincide. It follows from the definition that the Newton polygon stratification is at worst coarser than the Newton stratification. We reformulate the invariant $\NP$ in group theoretic terms to show that they are in fact equal.

 Let $\Xund,S$ be as above and assume in the PEL-case that $(B,O_B,\, ^\ast)$ is of type (AC). Let $V,\Lambda$ and $G$ be associated to $\Xund$ as in subsection \ref{ss PEL vs D}. Define $H := \prod \GL_{\ZZ_p} \Lambda_i$ with canonical embedding $i: G \mono H$. Then the Newton polygon stratification is given by the invariant
 \[
  \NP(\sbar) := \nu_H(i(b_{\Xund}(\sbar))).
 \]

 \begin{lemma} \label{lem newton polygon stratification}
  The Newton stratification and the Newton polygon stratification on $S$ coincide.
 \end{lemma}
 \begin{proof}
  We start with the EL-case. Obviously it suffices to show the claim when $B$ is simple, by Morita equivalence we may further assume that $B=F$ is a finite unramified field extension of $\QQ_p$. In this case the assertion was already proven in Example \ref{ex isocrystal}.

  In the PEL-case, $i$ factorizes as  $G \stackrel{i_1}{\mono} \GL_{O_B}(\Lambda) \stackrel{i_2}{\mono} H$. We have already seen that $i_2$ separates Newton points, thus it suffices to prove this for $i_1$. But we have already noted in section~\ref{ss PEL vs D} that $i_1$ induces an injective map on the dominant cocharacters thus in particular separates Newton points.
 \end{proof}

 \subsection{Break points of Newton polygons and Newton points}
 In order to generalize Proposition~\ref{prop yang} to the PEL-case, we need the group theoretic description of break points.

 \begin{definition}
  Let $G$ be a reductive group over $\ZZ_p$, $T\subset B \subset G$ be a maximal torus and an Borel subgroup of $G$ and let $\Delta_{\QQ_p}^+(G)$ denote the set of simple relative roots of $G_{\QQ_p}$. For any $\nu \in \Nscr(G)$ and $\beta\in \Delta_{\QQ_p}^+(G)$ we make the following definitions.
  \begin{subenv}
   \item We say that $\nu$ has a break point at $\beta$ if $\langle \nu, \beta \rangle > 0$. We denote by $J(\nu)$ the set of break points of $\nu$.
   \item Let $\omega_\beta^\vee,$ be the relative fundamental coweight of $G^\ad$ corresponding to $\beta$ and let \\ 
   $pr_\beta:X_*(T)^\Gamma \to X_*(T^\ad)^\Gamma \to \QQ\cdot\omega_\beta$ be the orthogonal projection (cf.~\cite{chai00}~chapter~6). If $j \in J(\nu)$, we call $(\beta, pr_\beta(\nu))$ the break point of $\nu$ at $\beta$.
  \end{subenv}
 \end{definition}

 \begin{example}
  Consider the case $G=GL_h$ with diagonal torus $T$ and Borel of upper triangular matrices $B$. Then the simple roots are
  \[
   \beta_i: \diag(t_1,\ldots,t_n) \mapsto t_{i+1}-t_i.
  \]
  Hence $J(\nu)$ is the set of all $j$ such that $\nu_j >\nu_{j+1}$. In terms of the Newton polygon $P$ associated to $\nu$, this is the set of all places where the slope on the left and the slope on the right do not coincide.

  We use the standard identification $X_*(T^\ad)_\QQ \cong \QQ^n/\QQ$ where $\QQ\mono\QQ^n$ is the diagonal embedding. We write $[\nu_1,\ldots,\nu_n]$ for $(\nu_1, \ldots, \nu_n) \mod \QQ$. Now $\omega_{\beta_j}^\vee = [1,\ldots,1,0,\ldots,0]$ and a short calculation shows
  \begin{eqnarray*}
   pr_\beta(\nu) &=& [\underbrace{\frac{\nu_1+\cdots +\nu_j}{j},\ldots,\frac{\nu_{1}+\cdot+\nu_j}{j}}_{j \textnormal{ times}}, \underbrace{\frac{\nu_{j+1}+\cdots+\nu_n}{n-j},\ldots,\frac{\nu_{j+1}+\cdots+\nu_n}{n-j}}_{n-j \textnormal{ times}}] \\
   &=& (\frac{\nu_1+\ldots+\nu_j}{j} - \frac{\nu_{j+1}+\ldots+\nu_n}{n-j})\cdot \omega_{\beta_j}^\vee \\
   &=& (\frac{n}{j(n-j)} P(j)-\frac{1}{n-j} P(h) ) \cdot \omega_{\beta_j}^\vee.
  \end{eqnarray*}
  In particular if $\nu,\nu' \in \Nscr(G)$ with corresponding Newton polygons $P,P'$ satisfying $P(n) = P'(n)$ we have $pr_{\beta_j}(\nu) = pr_{\beta_j}(\nu')$ if and only if $P(j) = P'(j)$. Thus, under the premise that $P(n) = P'(n)$, the notion of a common break point for $\nu,\nu' \in \Nscr(\GL_h)$ coincides with the classical definition for Newton polygons.
 \end{example}

 \begin{lemma} \label{lem break points}
  In the situation of the previous subsection, let $\beta$ a relative root of $G_{\QQ_p}$. There exists a relative root $\beta'$ of $H_{\QQ_p}$ such that for any subset $\{\nu_i\} \subset \Nscr(G)$ the $\nu_i$ have a common break point at $\beta$ if and only if the $i(\nu_i)$ have a common break point at $\beta'$.
 \end{lemma}
 \begin{proof}
  We proceed similarly as in Lemma~\ref{lem newton polygon stratification}. In the EL-case we reduce to the case $B=F$ which follows from the explicit description given in Example~\ref{ex isocrystal}. In the PEL-case it thus suffices to consider the embedding $i_1:G \mono \GL_{O_F} \Lambda$. We note that the simple roots of $G$ are precisely the restriction of the simple roots of $\GL_{O_F} \Lambda$. Denote by $R'(\beta)$ the set of simple roots of $\GL_{O_F} \Lambda$ above $\beta$. Obviously we have for any $\beta' \in R'(\beta)$
  \[
   \langle \nu, \beta \rangle = \langle i_1(\nu), \beta'\rangle.
  \]
  Thus $\nu$ has a break point at $\beta$ if and only if $i_1(\nu)$ has a breakpoint at some (or equivalently every) element of $R'(\beta)$. Now it remains to show that for $\nu_1,\nu_2 \in \Nscr(\beta)$ we have $pr_\beta (\nu_1) = pr_\beta (\nu_2) \Lra pr_{\beta'} (i_{1}(\nu_1)) = pr_{\beta'} (i_{1}(\nu_2))$. As all functions appearing in the term are linear, it suffices to show
  \[
   pr_\beta (\nu) = 0 \Lra pr_{\beta'}(i_{1}(\nu)) = 0.
  \]
  By the definition of the projections $pr_\beta, pr_{\beta'}$ this is equivalent to
  \[
   \nu \in \sum_{\alpha \in \Delta_{\QQ_p}^+(G) \setminus \beta} \RR\alpha^\vee \Lra i_{1}(\nu) \in \sum_{\alpha' \in \Delta_{\QQ_p}^+(\GL_{O_F}\Lambda) \setminus \beta'} \RR\alpha'^\vee.
  \]
  This follows from the fact that $\alpha^\vee$ is a scalar multiple of the sum of the coroots corresponding to elements in $R'(\alpha)$, which is easily checked using the data given in the appendix.
 \end{proof}

 \subsection{The relationship between the dimension of Newton strata and their position relative to each other}
 Now we can finally formulate and prove the generalisation of Proposition~\ref{prop yang}.

 \begin{proposition}\label{prop yang generalised}
  Let $S$ be a connected locally noetherian $\FF_p$-scheme and $\Xund$ be a \BT\ over $S$ with EL structure or PEL structure of type (AC). If there exists a neighbourhood $U$ of a point $s\in S$ such that the Newton polygons $\nu_{\Xund}(x)$ of $\Xund$ over all points $x \in U\setminus \overline{\{s\}}$ have a common break point, then either $\codim_U(\overline{\{s\}}) \leq 1$ or $\nu_{\Xund}(s)$ has the same break point.
 \end{proposition}
 \begin{proof}
  By Lemma~\ref{lem newton polygon stratification} and \ref{lem break points} it suffices to prove the assertion for a family $(X_i)$ of Barsotti-Tate groups without additional structure. But this assertion obviously follows from Proposition~\ref{prop yang}.
 \end{proof}

 We now proceed to one of the central pieces of the proof of the main theorem.

 \begin{proposition} \label{prop stratification}
  We fix an unramified PEL-Shimura datum $\Dscr$ and assume that we have that for any $b \in B(G,\mu)$
  \begin{equation} \label{term blah} 
   \dim \Ascr_0^{\nu} \leq \langle \rho,\mu+\nu\rangle - \frac{1}{2} \defect_G (b).
  \end{equation}
  Then Theorem~\ref{thm dimension shimura} holds true.
 \end{proposition}
 \begin{proof}
  This proof is the same as the proof of the analogous assertion in the equal characteristic case considered in \cite{viehmann13}.Due to the importance of this assertion, we give the full proof anyway.

  Note that by Corollary~\ref{cor dimension formulas} the inequality on the dimension is equivalent to $\codim \Ascr_0^{\nu} \geq \lengthG{\nu}{\mu}$. We prove the claim by induction on $\lengthG{\nu}{\mu}$. For $\lengthG{\nu}{\mu} = 0$, i.e.\ $\nu=\mu$ the dimension formula of Theorem~\ref{thm dimension shimura} certainly holds true and the density of $\Ascr_0^{\mu}$ is known by Wedhorn \cite{wedhorn99}. Now fix an integer $i$ and assume the statement holds true for all $\nu\leq\mu$ and with $\lengthG{\nu}{\mu} < i$. Let $\nu'$ with $\lengthG{\nu'}{\mu} = i$. We fix an element $\nu \in [\nu',\mu]$ such that $\lengthG{\nu'}{\nu} = 1$. Then $\lengthG{\nu}{\mu} = i-1$ and by \cite{viehmann13} Lemma 5 (or more precisely the remark after this lemma), we have 
  \[
   \Nscr_{\leq\nu'} = \{\xi \in \Nscr_{\leq\nu}\mid pr_\beta(\xi) < \lambda\}
  \]
  for some break point $(\beta,\lambda)$ of $\nu$. In particular, $pr_\beta(\nu') < \lambda$.

  Let $\eta$ be a maximal point of $\Ascr_0^{\leq\nu'}$ (i.e.~the generic point of an irreducible component). By applying Proposition~\ref{prop yang generalised} to $S=\Spec \Oscr_{\Ascr_0^{\leq\nu},\eta}$ and $s = \eta$ we see that every irreducible component of $\Ascr_0^{\leq\nu'}$ has at most codimension 1 in $\Ascr_0^{\leq\nu}$. By (\ref{term blah}) and induction hypothesis we have $\dim \Ascr_0^{\leq\nu'} < \dim \Ascr_0^\nu$, thus $\Ascr_0^{\leq\nu'}$  is of pure codimension $\lengthG{\nu'}{\mu}$ in $\Ascr_0$. By (\ref{term blah}) we know that for any $\xi < \nu'$ that $\codim \Ascr_0^\xi \geq \lengthG{\xi}{\mu} > \lengthG{\nu'}{\mu}$, thus $\Ascr_0^{\nu'}$ is dense in $\Ascr_0^{\leq\nu'}$ and Theorem~\ref{thm dimension shimura} follows.
\end{proof}
 
 \section{Central leaves of $\Ascr_0$} \label{sect central leaves}

 In the previous proposition we reduced Theorem~\ref{thm dimension shimura} to an estimate of the dimension of Newton strata in $\Ascr_0$. In the special case of the Siegel moduli variety, Oort has calculated the dimension of Newton strata by writing irreducible components ``almost'' as a product of so-called central leaves and isogeny leaves and calculating the dimension of these. We will use a similar approach to prove the estimate, applying Mantovan's theorem which implies that a Newton stratum of a PEL-Shimura variety is in finite-to-finite correspondence with the product of a central leaf with a truncated Rapoport-Zink space (cf.\ Proposition~\ref{prop almost decomposition}). First, let us recall the definition and the basic properties of central leaves.

 \subsection{The geometry of central leaves}
  
 \begin{definition}
  Let $\Xund$ be a completely slope divisible {\BTD} over $k$ with isogeny class $b$ (Recall that by \cite{OZ02}~Lemma~1.4 a \BT $X$ over $k$ is slope divisible if and only if it is isomorphic to a direct sum of isoclinic \BT s which are defined over a finite field). The corresponding central leaf is defined as
  \[
   C_{\Xund} := \{x \in \Ascr_0 \mid \Aund^{\univ}_{\overline{x}}[p^\infty] \cong \Xund_{\overline{k(x)}}\}
  \]
 \end{definition}
 The central leaf is a smooth closed subscheme of $\Ascr_0^{b}$ by \cite{mantovan05}~Prop.~1. The fact that $C_{\Xund}$ is closed in the Newton stratum is proven by writing $C_{\Xund}$ as a union of irreducible components of
 \[
  C_{X} := \{x \in \Ascr_0 \mid A^{\univ}_{\overline{x}}[p^\infty] \cong X_{\overline{k(x)}}\},
 \]
 which is closed in $\Ascr_0^{b}$ by a result of Oort. Viehmann and Wedhorn showed that every central leaf is non-empty (\cite{VW13}~Thm.~10.2). Furthermore, the dimension of $C_{\Xund}$ only depends on $b$.

 \begin{proposition} \label{prop dimension central leaves}
  Let $\Dscr' = (B,\ast,V,\langle\, , \rangle, O_B,\Lambda,h)$ be a second unramified Shimura datum which agrees with $\Dscr$ except for $h$ and let ${K^p}' \subset G(\AA^p)$ be a sufficiently small open compact subgroup. Denote by $\Ascr_0'$ the special fibre of the associated moduli space. We consider two \BT s $\Xund = (X,\iota,\lambda)$ and $\Xund' = (X',\iota',\lambda')$ over $k$ equipped with $\Dscr$- resp.\ $\Dscr'$-structure and assume that we have an isogeny $\rho: \Xund \to \Xund'$. Denote by $C_X$ and $C_{X'}$ the associated central leaves in $\Ascr_0$ resp.\ $\Ascr'_0$. Then
  \[
   \dim C_{\Xund} = \dim C_{\Xund'}.
  \]
 \end{proposition}

 \begin{remark}
  \begin{subenv}
   \item Oort proved the analogous assertion for moduli spaces of (not necessarily principally) polarised abelian varieties (\cite{oort04}~Thm.~3.13). Our proof is a generalisation of his proof.
   \item In the case $\Dscr = \Dscr'$ the assertion was already proved by Mantovan. In the proof of \cite{mantovan04}~Prop.~4.7 she showed the proposition only for some special cases of PEL-Shimura data, but her proof can be generalised to arbitrary unramified PEL-Shimura data using \cite{mantovan05}~ch.4 and 5.
  \end{subenv}
 \end{remark}

 \begin{proof}
  First we fix some notation. Let $n' = \dim_\QQ V$. Oort showed in \cite{oort04}~Cor.~1.7 that there exists an integer $N(n')$ such that any two \BT s $X,X'$ of height $n'$ over an algebraically closed field with $X[p^{N(n')}] \cong X'[p^{N(n')}]$ are isomorphic. Denote $H := \ker \rho$, furthermore $i := \deg \rho$ and $n = N(n')+i$. Choose $x =(A,\iota,\lambda,\eta) \in C_{\Xund}$ and let $x' = (A/G,\iota',\lambda',\eta')$. We denote by $C_x$ and $C_{x'}$ the connected components of the central leaves containing $x$ resp.\ $x'$. We denote the corresponding universal abelian varieties with additional structure by
  \begin{eqnarray*}
   \Mund &\rightarrow& C_x \\
   \Mund'&\rightarrow& C_{x'}
  \end{eqnarray*}
  and
  \begin{eqnarray*}
   \Yund &=& \Mund[p^\infty] \\
   \Yund' &=& \Mund'[p^\infty].
  \end{eqnarray*}

  Using a slight generalisation of \cite{oort04}~Lemma~1.4 (the same proof still works) there exists a scheme $T \to C_{x}$ finite and surjective such that $\Mund [p^n]_T$ is constant. We assume that $T$ is irreducible. The abelian variety with additional structure $\Mund_T/H_T$ defines a morphism $f:T \rightarrow \Ascr_0'$. Using \cite{oort04}~Cor.~1.7, we see that $f$ factorises over $C_X$; as $T$ is irreducible it thus factors over $C_{x'}$.

  We now show that $f$ is quasi-finite. We denote by $\varphi: \Yund_T \rightarrow (\Mund_T/H_T)[p^\infty]$ the isogeny constructed above and choose $\psi$ such that $\varphi \circ \psi = \psi\circ\varphi = p^n$. Let $u\in C_x$ and $S \subset T_u$ a reduced and irreducible subscheme. Then
  \[
   \psi_S: (\Yund'_u)_S = (\Mund_T/H_T)[p^\infty]_S \rightarrow X_S.
  \]
  By arguing as in \cite{oort04}~\S~1.11 one checks that the kernel of $\psi_S$ is constant. Thus
  \[
   \Mund_S \cong \Mund'_S/\ker \psi_S
  \]
  is also constant, i.e.\ the image of $S$ in $C_x$ is a single point. As $T$ is finite over $C_x$, this implies that $S$ is a single point.

  So we get $\dim C_x = \dim T \leq \dim C_{x'}$. As the assertion of the proposition is symmetric in $X$ and $X'$, the claim follows.
 \end{proof}

 \begin{proposition} \label{prop almost decomposition}
  Let $\nu \in B(G,\mu)$ and let $\XXund$ be a slope divisible \BT\ over $\FFbar_p$ with $\Dscr$-structure in the isogeny class $b$. Denote by $\Mscr_G(b,\mu)$ the Rapoport-Zink space associated to $\XXund$. Then
  \[
   \dim \Ascr_0^b = \dim C_{\XXund} + \dim \Mscr_G(b,\mu)
  \]
 \end{proposition}
 \begin{proof}
  By \cite{mantovan05}~ch.~5 there exists a finite surjective map $\pi_N: J_{m,b} \times \Mscr_b^{n,d} \to \Ascr_0^\nu$ ($m,n,d\gg 0$). Here $J_{m,b}$ is an Igusa variety over $C_\XXund$ (cf.~\cite{mantovan05}~ch.~4), which in particular means finite \'etale over $C_\XXund$, and $\Mscr_b^{n,d}$ are truncated Rapoport-Zink spaces (cf.~ \cite{mantovan05}~ch.~5), which are quasi-compact and of the same dimension of $\Mscr_G(b,\mu)$ for $n,d \gg 0$. Hence
  \[
   \dim \Ascr_0^b = \dim J_{m,b} + \dim \Mscr_b^{n,d} = \dim C_\XXund + \dim \Mscr(b,\mu).
  \]
 \end{proof}

 \subsection{Results on Ekedahl-Oort-strata}
 Proposition \ref{prop almost decomposition} reduces the dimension formula of Theorem \ref{thm dimension shimura} to computing the dimension of the central leaves and of Rapoport-Zink spaces. Because of Proposition~\ref{prop dimension central leaves} it suffices to calculate the dimension of the central leaf for one representative of each isogeny class.
 In order to find a suitable central leaf we consider the Ekedahl-Oort stratification, which is studied in great detail in the paper \cite{VW13} of Viehmann and Wedhorn. We recall some of their notions and results.

 For a \BTD\ $\Xund$ over $k$ we denote by $\tc(\Xund)$ the isomorphism class of $\Xund[p]$. The map $\tc$ has the following group theoretic interpretation (cf.~\cite{VW13}~ch.~8). Recall that an isomorphism class of \BTDs\ corresponds to an element of $C(G)$, the set of $G(O_L)$-$\sigma$-conjugacy classes in $G(L)$. By Dieudonn\'e theory the isomorphism classes of truncated \BTDs\ correspond to $G(O_L)$-$\sigma$-conjugacy classes in $G_1 \backslash G(L) / G_1$ where $G_1 := \ker(G(O_L) \epi G(k))$. Now the truncation map $\tc$ is given by the canonical projection.
 \[
  C(G) \epi \{ \sigma-\textnormal{conjugacy classes in } G_1 \backslash G(L) / G_1 \}
 \]

 The Ekedahl-Oort stratification is the decomposition of $\Ascr_0(\FFbar_p)$ given by the invariant $\tc$. The Ekedahl-Oort strata are locally closed subsets of $\Ascr_0(\FFbar_p)$ and are indexed by elements $w\in W$ which are the shortest element in the coset $W_{\sigma^{-1}(\mu)}w$. Here $W_{\sigma^{-1}(\mu)}$ denotes the Weyl group of the centralizer of $\sigma^{-1}(\mu)$. We denote the Ekedahl-Oort stratum corresponding to $w$ by $\Ascr_{0,w}$.

 \begin{proposition}[\cite{VW13}~Thm.~11.1 and Prop.~11.3]
  $\Ascr_{0,w}$ is non-empty and of pure dimension $\ell(w)$.
 \end{proposition}

 As we want to calculate the dimension of central leaves, we are interested in Ekedahl-Oort strata corresponding to $p$-torsion subgroups (with additional structure) which determine their {\BTD} uniquely.

 \begin{definition}
  \begin{subenv}
   \item A \BTD\ $\Xund$ over $k$ is called minimal if every \BTD\ $\Yund$ with $\Xund[p] \cong \Yund[p]$ is isomorphic to $\Xund$.
   \item An Ekedahl-Oort stratum is called minimal if the fibre of $\Aund^{\univ}[p^\infty]$ over some (or equivalently every) point of it is minimal.
  \end{subenv}
 \end{definition}

 Certainly an Ekedahl-Oort stratum is a central leaf if and only if it is minimal and the corresponding {\BT} over $\FFbar_p$ is slope-divisible. Viehmann and Wedhorn show in their paper that every isogeny class contains a minimal \BTD, more precisely they show that this is true for the stronger notion of fundamental elements. We recall their definition.

 \begin{definition}
  \begin{subenv}
   \item Let $P$ be a semistandard parabolic subgroup of $G_{O_F}$, denote by $N$ its unipotent radical and let $M$ be the Levi factor containing $T_{O_F}$. Furthermore denote by $\Nbar$ the unipotent radical of the opposite parabolic. We denote
   \[
    \Iscr_M = \Iscr \cap M(O_L),\quad \Iscr_N = \Iscr \cap N(O_L),\quad \Iscr_{\Nbar} = \Iscr\cap \Nbar(O_L).
   \]
   Then an element $x \in \Wtilde$ is called $P$-fundamental if
   \begin{eqnarray*}
    \sigma(x \Iscr_M x^{-1}) &=& \Iscr_M, \\
    \sigma(x \Iscr_N x^{-1}) &\subseteq& \Iscr_N, \\
    \sigma(x \Iscr_{\Nbar} x^{-1}) &\supseteq& \Iscr_{\Nbar}.
   \end{eqnarray*}
   \item We call an element $x \in \Wtilde$ fundamental if it is $P$-fundamental for some semi-standard parabolic subgroup $P \subset G_{O_F}$.
   \item We call $\pot{c} \in C(G)$ fundamental if it contains a fundamental element of $\Wtilde$. A {\BTD} over $\FFbar_p$ is called fundamental if the corresponding element of $C(G)$ is fundamental.
   \item An Ekedahl-Oort stratum is called fundamental if the fibre of $\Aund^{\univ}[p^\infty]$ over some point of it is fundamental.
  \end{subenv}
 \end{definition}

 By \cite{VW13}~Rem.~9.8 every fundamental {\BTD} is also minimal. As mentioned above they also show the following assertion.

 \begin{proposition}[\cite{VW13}~Thm.~9.18] \label{prop VW}
  Let $G$ and $\mu$ be the reductive group scheme and the cocharacter associated to an unramified PEL-Shimura datum. Then for every $b \in B(G,\mu)$ there exists a fundamental element $x \in \Wtilde$ such that $x \in b$ and $x \in W\mu'$ for some minuscule cocharacter $\mu'$. Furthermore the {\BTD} associated with $\pot{x}$ is completely slope divisible.
 \end{proposition} 

 The slope divisibility is not mentioned in \cite{VW13}, but following their construction of $x$ one easily checks that the induced {\BTD} is completely slope divisible.

 \section{The dimension of certain Ekedahl-Oort-strata} \label{sect EO strata}
 
 In order to apply the formula $\dim \Ascr_{0,w} = \ell(w)$ to fundamental Ekedahl-Oort strata, we have to explain how to compute $w$, or rather $\ell(w)$, if one is given a fundamental element $x \in \Wtilde$. This can be done by the algorithm provided by Viehmann in the proof of \cite{viehmann}~Thm.~1.1. Before we apply this algorithm, we recall some combinatorics of the Weyl group which it uses.

 \subsection{Shortest elements in cosets of the extended affine Weyl group}
 We denote by $S$ resp.\ $S_a$ the set of simple reflections in $W$ resp.\ $W_a$. For $J \subset S_a$ let $W_J \subset \Wtilde$ be the subgroup generated by $J$. Then every right (resp.\ left) $W_J$-coset of $\Wtilde$ contains a unique shortest element. We denote by $\Wtilde^J$ (resp. ${^J\Wtilde}$) the set of those elements. Then every $x \in \Wtilde$ can be written uniquely as $x = x^J \cdot w_J = {_Jw} \cdot {^Jx}$ with $x^J \in \Wtilde^J, w_J \in   W_J, {_Jw} \in {_JW}, {^Jx} \in {^J\Wtilde}$. We have $\ell(x^J) + \ell(w_J) = \ell(x) = \ell(_Jw) + \ell(^Jx)$. (cf.\ \cite{DDPW}~Prop.~4.16)

 Moreover there exists a unique shortest element for every double coset $W_J\backslash \Wtilde/W_K$. we denote the set of all shortest elements in their respective double coset by ${^J\Wtilde^K}$. Let $u \in {^J\Wtilde^K}$ and $K' := K \cap u^{-1}Ju$. Then $W_{K'} = W_K \cap u^{-1}W_Ju$ and every element $x \in W_JuW_K$ can be written uniquely as $x = w_Juw_{K'}$ with $w_J \in W_J$ and $w_{K'}\in W_{K'}$. Moreover, we have $\ell(x) = \ell(w_J) + \ell(u) + \ell(w_{K'})$. (cf.\ \cite{DDPW}~Lemma~4.17 and Thm.~4.18)

 Of course the above statements also hold for $J,K \subset S$ and $W$ instead of $\Wtilde$. We denote the respective sets of shortest elements by $W^J, {^JW}$ resp.\ ${^JW^K}$. 

 For $\mu \in X_*(T)_{\dom}$ denote by $\tau_\mu$ the shortest element of $Wp^\mu W$. Then $\tau_\mu$ is of the form $x_\mu \cdot p^\mu$ with $x_\mu \in W$. Let $M_\mu$ be the centralizer of $\mu$ and $W_\mu = \{w\in W\mid w(\mu) = \mu\}$ the Weyl group of $M_\mu$. We have the following useful lemmas.

 \begin{lemma}
  $W_\mu = W_{S \cap \tau_\mu^{-1}S\tau_\mu} = W \cap \tau_\mu^{-1} W \tau_\mu$.
 \end{lemma}
 \begin{proof}
  It suffices to show that for $s\in S$ we have $s(\mu) = \mu$ if and only if $s \in \tau_\mu^{-1}S\tau_\mu$. If $s(\mu)=\mu$ then $\tau_\mu s \tau_\mu^{-1} \in W$. Thus
  \[
   \ell(\tau_\mu s \tau_\mu^{-1}) = \ell(\tau_\mu s \tau_\mu^{-1} \tau_\mu ) - \ell(\tau_\mu) = \ell(\tau_\mu s) - \ell(\tau_\mu) = \ell(\tau_\mu) + \ell(s) - \ell(\tau_\mu) = \ell(s) = 1
  \]
  and hence $\tau_\mu s \tau_\mu^{-1} \in S$. On the other hand we have
  \[
   \tau_\mu s \tau_\mu^{-1} = x_\mu p^\mu s p^{-\mu} x_{\mu}^{-1} = x_\mu s x_{\mu}^{-1} p^{x_\mu(s(\mu)-\mu)}.
  \]
  Thus $\tau_\mu s \tau_{\mu}^{-1} \in S$ implies $s(\mu)-\mu = 0$.
 \end{proof}

 We denote $S_\mu := S \cap \tau_\mu^{-1}S\tau_\mu$.

 \begin{lemma} \label{lem length preservation} 
  Let $J,K \subset S_a, u \in {^J\Wtilde^K}$. Denote by $K' := K \cap u^{-1}Ju$ and let $w \in W_{K'}$. Then
  \[
   \ell(uwu^{-1}) = \ell(w).
  \]
 \end{lemma}
 \begin{proof}
  We have
  \begin{eqnarray*}
   l(w) &=& \ell(uw) - \ell(u) \\
        &=& \ell((uwu^{-1})u) - \ell(u) \\
        &=& \ell(uwu^{-1}) + \ell(u) - \ell(u) \\
        &=& \ell(uwu^{-1})
  \end{eqnarray*}
 \end{proof}

 \begin{corollary} \label{cor length preservation}
  For $x \in \Wtilde_\mu$ we have
  \[
   \ell(x_\mu x x_\mu^{-1}) = \ell(x)
  \]
 \end{corollary}
 \begin{proof}
  This is a consequence of $x_\mu x x_\mu^{-1} = \tau_\mu x \tau_\mu^{-1}$ and the preceding two lemmas.
 \end{proof}

 \subsection{$G(O_L)$-$\sigma$-conjugacy classes in $\widetilde{W}$}
 We have the following result on $G(O_L)$-$\sigma$-conjugacy classes of $G_1\backslash G(L) / G_1$.

 \begin{proposition}[\cite{viehmann}~Thm.~1.1]
  Let $\Tscr = \{(w,\mu) \in W \times X_*(T)_\dom \mid w \in {^{M_{\sigma^{-1}(\mu)}}W}\}$. Then the map assigning to $(w,\mu)$ the $G(O_L)$-$\sigma$-conjugacy class of $G_1w\tau_\mu G_1$ is a bijection between $\Tscr$ and the $G(O_L)$-$\sigma$-conjugacy classes in $G_1\backslash G(L) / G_1$.
 \end{proposition}

 In the case where $G$ is given by an unramified PEL-Shimura datum, the stratum $\Ascr_{0,w}$ corresponds to the $G(O_L)$-$\sigma$-conjugacy class of $G_1 w\tau_{\mu_h} G_1$. The proof of the above theorem provides an algorithm which determines the pair $(w,\mu)$ associated to the $G(O_L)$-$\sigma$-conjugacy class of $G_1bG_1$ for any $b \in G(L)$. For the special case that $b \in \Wtilde$ the algorithm simplifies as follows.
 \begin{enumerate}
  \item Denote by $\mu'$ the image of $b$ under the canonical projection $\Wtilde \epi X_*(T), w\cdot p^\lambda \mapsto p^\lambda$ and let $\mu = \mu'_{\dom}$. Then $b \in W\tau_\mu W$, thus we may write
  \[
   b = w \tau_\mu w'
  \]
  with $w,w' \in W$, $\ell(b) = \ell(w) + \ell(\tau_\mu) + \ell(w')$. Now replace $b$ by its $\sigma$-conjugate
  \[
   \sigma^{-1}(w') b w^{-1} = \sigma^{-1}(w') w \tau_\mu =: b_0 \tau_\mu.
  \]
  \item Define sequences of subsets $J_i,J_i' \subset S$ and sequences of elements $u_i \in W,b_i \in W_{J_i'}$ as follows;
  \begin{enumerate}
   \item $J_0 = J_0' = S$, \newline
   $J_1 = \sigma^{-1}(S_\mu)$ and $J_1' = x_\mu S_\mu x_\mu^{-1}$, \newline
   $J_i = J_{i-1}' \cap u_{i-1}J_1u_{i-1}^{-1}$ and $J_i' = x_\mu \sigma(u_{i-1} J_i u_{i-1}^{-1}) x_\mu^{-1}$.
   \item $u_0 =1$ and $b_0$ as above. Let $\delta_i$ be the shortest length representative of $W_{J_i}\delta_{i-1}W_{J_i'}$ in $W_{J_{i-1}'}$. Then $u_i = u_{i-1}\delta_i$ and $b_i$ is defined as follows. Decompose
   \[
    b_{i-1} = w_i \delta_i w_i'
   \]
   with $w_i \in J_i$, $w_i' \in J_i'$ such that $\ell(b_{i-1}) = \ell(w_i) + \ell(\delta_i) + \ell(w_i')$. Then
   \[
    b_i = w_{i}' \cdot x_\mu \sigma(u_{i-1}w_{i}u_{i-1}^{-1}) x_\mu^{-1}
   \]
  \end{enumerate}
  The above sequences satisfy the following properties. The sequences $(J_i)$ and $(J_i')$ are descending and $u_i \in {^{J_1} W ^{J_i'}}$.
  \item Let $n$ be sufficiently large such that $J_n = J_{n+1}, J_n' = J_{n+1}'$. Then $(w,\mu)$ is given by $w=u_n$ and $\mu$ as in step $1$.
 \end{enumerate}
 
 \begin{remark}
  Following the construction above, we see that we not only the $K_1$-double cosets of $b$ and $w\tau_\mu$ are $G(O_L)$-$\sigma$-conjugate, but also $b$ and $\tau_\mu$ themselves. Indeed, we have a chain of $\sigma$-conjugations
  \begin{eqnarray*}
   u_0b_0\tau_\mu &=&  \sigma^{-1}(w') b w^{-1} \\
   u_ib_i\tau_\mu &=& (u_{i-1}w_iu_{i-1}^{-1})^{-1} u_{i-1}b_{i-1}\tau_\mu\; \sigma(u_{i-1}w_iu_{i-1}^{-1}) \textnormal{ for } i > 0
  \end{eqnarray*}
  and Viehmann shows in her proof of the above theorem that $u_n\tau_\mu$ is $G(O_L$)-$\sigma$-conjugated to $u_nb_n\tau_\mu$.
 \end{remark}
 
 \begin{proposition} \label{prop length tr}
  Let $b,(w,\mu)$ as above. Then
  \[
   \ell(b) \geq \ell(w\tau_\mu).
  \]
 \end{proposition}
 \begin{proof}
  We follow the algorithm above. We have
  \[
   \ell(b_0\tau_\mu) = \ell(\sigma^{-1}(w')w\tau_\mu) \leq \ell(w') + \ell(w) + \ell(\tau_\mu) = \ell(b).
  \]
  Next, we show that $\ell(u_ib_i)\leq \ell(u_{i-1}b_{i-1})$ for all $i$. Note that
  \[
   \ell(u_{i-1}b_{i-1}) = \ell(u_{i-1}) + \ell(b_{i-1}) = \ell(u_{i-1}) + \ell(w_i) + \ell(\delta_i) + \ell(w_i') = \ell(u_i) + \ell(w_i) +\ell(w_i').
  \]
  Thus we have to show that $\ell(b_i) \leq \ell(w_i) + \ell(w_i')$. Now
  \begin{eqnarray*}
   \ell(b_i) &=& \ell(w_i' x_\mu \sigma(u_{i-1} w_i u_{i-1}^{-1}) x_\mu^{-1}) \\
   &\leq& \ell(w_i') + \ell(x_\mu \sigma(u_{i-1} w_i u_{i-1}^{-1}) x_\mu^{-1}) \\
   &\stackrel{\rm Cor.~\ref{cor length preservation}}{=}& \ell(w_i') + \ell(\sigma(u_{i-1}w_iu_{i-1}^{-1})) \\
   &=& \ell(w_i') + \ell(u_{i-1}w_iu_{i-1}^{-1}) \\
   &\stackrel{\rm Lemma~\ref{lem length preservation}}{=}& \ell(w_i') + \ell(w_i).
  \end{eqnarray*}
  Altogether, we have
  \[
   \ell(w\tau_\mu) = \ell(w)+\ell(\tau_\mu) = \ell(u_n)+ \ell(\tau_\mu) = \ell(u_nb_n) - \ell(b_n) + \ell(\tau_\mu) \leq \ell(b_0) + \ell(\tau_\mu) \leq \ell(b).
  \]
 \end{proof}

 \subsection{Fundamental Ekedahl-Oort strata} \label{ss fundamental elements}
 It is well-known that one has for any $x \in \Wtilde$ the inequality $\ell(x) \geq \langle 2\rho, \nu(x) \rangle$. One calls $x$ a $\sigma$-straight element if equality holds, or equivalently if
 \[
  \ell(x\cdot \sigma(x) \cdot \ldots \cdot \sigma^n(x)) = (n+1) \cdot \ell(x)
 \]
 for every non-negative integer $n$ (see for example \cite{he}~Lemma~8.1). 

 \begin{lemma}
  Let $G$ be a reductive group scheme with extended affine Weyl group $\Wtilde$. Then any fundamental element $x \in \Wtilde$ is $\sigma$-straight.
 \end{lemma}
 \begin{proof}
  In the split case this was proven in \cite{GHKR10}~Prop.13.1.3. Our proof uses the same idea but is a bit more technical than the proof in the split case. We show
  \[
  \ell(x\cdot \sigma(x) \cdot \ldots \cdot \sigma^n(x)) = (n+1)\cdot \ell(x)
 \]
 via induction on $n$. For $n=0$ the assertion is tautological. For $n \geq 1$ let $x_{(n)} = \sigma(x)\cdot \ldots \sigma^n(x)$. As $x$ is fundamental we deduce
 \[
  \sigma(x\Iscr_Mx^{-1}) = \Iscr_M \Rightarrow \Iscr_{\sigma^k(M)} = \sigma^k(x^{-1}) \Iscr_{\sigma^{k-1}(M)} \sigma^k(x) \textnormal{ for all } k \Rightarrow x_{(n)}^{-1} \Iscr_M x_{(n)} = \Iscr_{\sigma^n(M)} \subset \Iscr,
 \]
 \[
  \sigma(x\Iscr_{\Nbar}x^{-1}) \supseteq \Iscr_{\Nbar} \Rightarrow \Iscr_{\sigma^k(\Nbar)} \supseteq \sigma^k(x^{-1}) \Iscr_{\sigma^{k-1}(\Nbar)} \sigma^k(x) \textnormal{ for all } k \Rightarrow x_{(n)}^{-1} \Iscr_{\Nbar} x_{(n)} \subseteq \Iscr_{\sigma^n(\Nbar)} \subset \Iscr.
 \]
 Hence by Iwahori factorisation
 \[
  x\Iscr x_{(n)} = x\Iscr_N \Iscr_M \Iscr_{\Nbar} x_{(n)} = (x\Iscr_Nx^{-1})x x_{(n))} (x_{(n)}^{-1}\Iscr_M x_{(n)}) (x_{(n)}^{-1} \Iscr_{\Nbar}x_{(n)}) \subset \Iscr x x_{(n)} \Iscr
 \]
 which implies
 \[
  \ell(x\cdot x_{(n)}) = \ell(x) + \ell(x_{(n)}) = \ell(x) + n\cdot \ell(x) = (n+1) \cdot \ell(x).
 \]
 \end{proof}

 \begin{corollary} \label{cor dimension of central leaves}
 We fix $b \in B(G,\mu)$.
  \begin{subenv}
   \item Let $x\in\Wtilde$ and $\mu' \in X_*(T)$ be as in Proposition~\ref{prop VW}. Denote by $\Ascr_0'$ the special fibre of the moduli space associated to the datum $(\Bsf,\Osf_\Bsf,\Vsf, \langle\,\, ,\, \rangle, \Lambda, \mu'_\dom)$ and by $\Ascr'_{0,w}$ the Ekedahl-Oort stratum corresponding to $\tc (x)$. Then $\dim \Ascr'_{0,w} = \langle 2\rho,\nu_G(b) \rangle$.
   \item Any central leaf in $\Ascr_0^{b}$ has dimension $\langle 2\rho,\nu_G(b) \rangle$.
   \item We have $\dim \Ascr_0^{b} = \dim \Mscr_G(b,\mu) + \langle 2\rho,\nu_G(b) \rangle$.
  \end{subenv}  
 \end{corollary}
 \begin{proof}
  We have $\dim \Ascr_{0,w} = \ell(w) = \ell(w\tau_\mu)$. Here the last equalisty holds as $\ell(\tau_\mu) = 0$ because $\mu$ is minuscule. Now $\nu_G(w \tau_\mu) = \nu_G(b)$, thus $\ell(w\tau_\mu) \geq \langle 2\rho,\nu_G(b) \rangle$. On the other hand we have $\ell(w\tau_\mu) \leq \ell(x) = \langle 2 \rho, \nu_G(b) \rangle$ by Proposition \ref{prop length tr} and the previous lemma, thus proving part (1). The second part is a direct consequence of part (1) and Proposition~\ref{prop dimension central leaves} and the last assertion follows by Prop.~\ref{prop almost decomposition}.
 \end{proof}

 \section{Epilogue} \label{sect epilogue}

 Now we have reduced the task of proving the main theorems to proving the following proposition.

 \begin{proposition} \label{prop main}
  \begin{subenv}
   \item We have
   \begin{equation}
    \dim \Mscr_G(b,\mu) \leq \langle \rho, \mu-\nu_G(b) \rangle - \frac{1}{2} \defect_G(b)
   \end{equation}
   \item Assume that $b$ is superbasic. Then the connected components of $\Mscr_G(b,\mu)$ are projective.
  \end{subenv}  
 \end{proposition}

 We recall the reduction steps we made so far for the readers convenience.  First we reduced Theorem~\ref{thm dimension shimura} to the inequality (\ref{term blah}) in Proposition~\ref{prop stratification}. Assuming that the above proposition holds true, we obtain the inequality (\ref{term blah}) by applying the above estimate for $\dim \Mscr_G(b,\mu)$ to Corollary~\ref{cor dimension of central leaves}~(3). Now Theorem~\ref{thm dimension deformation} follows from Theorem~\ref{thm dimension shimura} by Proposition~\ref{prop relation}.

 In return, applying the dimension formula for $\Ascr_0^b$ of Theorem~\ref{thm dimension shimura} to Corollary~\ref{cor dimension of central leaves}~(3) yields the first assertion of Theorem~\ref{thm dimension RZ-space}. The other assertion is identical to the second assertion of the proposition above.

 \section{The correspondence between the general and superbasic case} \label{sect BT again}

 \subsection{Simple RZ-data and Rapoport-Zink spaces over perfect fields} \label{ss simple RZ}
 We call an un\-ra\-mi\-fied Rapoport-Zink datum $(B,O_B,(*),V,(\pair),\Lambda)$ \emph{simple} if $B=F$ is an unramified field extension. In \cite{fargues04}~\S~2.3.7 Fargues gives an open and closed embedding
 \[
  \Mscr_G(b,\mu) \mono \prod_i \Mscr_{G_i}(b_i,\mu_i)
 \]
 of an arbitrary Rapoport-Zink space into a product of Rapoport-Zink spaces associated to simple data using assertions as in section~\ref{ss normal forms of BTpELs}. In particular,
 \[
  \dim \Mscr_G(b,\mu) = \sum_i \dim \Mscr_{G_i}(b_i,\mu_i)
 \]
 and one obtains an analogous formula for the right hand side of the dimension estimate of Proposition \ref{prop main}. Thus it suffices to prove the proposition for simple Rapoport-Zink data.

 We note that as a consequence of \cite{RR96}~Lemma~1.3 the reduced subscheme of a Rapoport-Zink space can be defined over $\FFbar_p$, thus we can make the following assumption.

 \begin{notation}
  In the following we will assume that $k = \FFbar_p$.
 \end{notation}

 In order to estimate the dimension of the Rapoport-Zink spaces we will use some methods which only work for schemes defined over a finite field. So we consider the more general setup as in section~\ref{ss RZ spaces} where the RZ-spaces are defined over perfect fields. We fix a perfect field $k_0$ and denote $L_0 = W(k_0)_\QQ$.

 \begin{definition} \label{def relative RZ datum}
  A simple Rapoport-Zink datum relative to $L_0$ is a datum $\hat\Dscr = (F,O_F,(^\ast),V,$ 
  $(\pair),\Lambda,[b])$ as in Definition \ref{def RZ datum} but $[b]$ denotes a $\sigma$-conjugacy class in $G(L_0)$ which is assumed to be decent, i.e.\ it contains an element $b$ satisfying
  \begin{equation} \label{term decent}
   (b\sigma)^s = \nu_G (b)(p^s) \sigma
  \end{equation}
  for some natural number $s$.
 \end{definition}

 As in the case of an algebraically closed field, one can associate an isomorphism class of $\Bscr$-isocrystals to $[b]$ (\cite{RZ96}~Lemma~3.37). Now let $(\XX,\iota_\XX,(\lambda_\XX))$ be a {\BTpEL} with $\Bscr$-isocrystal $(L_0 \otimes V, b(\sigma \otimes \id))$.

 \begin{proposition}[\cite{RZ96}, Cor.~3.40] \label{prop relative RZ space}
  Assume that $L_0$ contains the local Shimura field and $\QQ_{p^s}$ (where $s$ is chosen as in (\ref{term decent})). Then the associated functor $\Mscr_G$ defined as in Definition \ref{def RZ space} is representable by a formal scheme formally locally of finite type over $\Spf W(k_0)$.
 \end{proposition}

 We note that every class $[b] \in B(G)$ contains an element $b_0$ which satisfies the decency equation (\ref{term decent}) above. By \cite{RZ96}~Cor.~1.9 we have $b_0 \in   \QQ_{p^s}$ thus every Rapoport-Zink space can be defined over a finite field $k_0$.

 \subsection{Construction of the correspondence}
  We fix a simple Rapoport-Zink datum $\hat\Dscr = (B,O_B,(^\ast),V,(\pair),\Lambda,\mu)$. Let $P \subset G$ be a standard parabolic subgroup defined over $\ZZ_p$ and $M \subset P$ be its Levi subgroup which contains $T$ . We will later assume that $[b]$ induces a superbasic $\sigma$-conjugacy class in $M(L)$. However, our construction works in greater generality.

  By \cite{SGA3-3}~Exp.~XXVI~ch.~1 the pairs $M \subset P$ as above are given by the root datum of $G$ in analogy with reductive groups over fields. Using the explicit description of the root data in the appendix, one sees that there exists a decomposition $\Lambda = \bar\Lambda_1 \oplus \ldots \oplus \bar\Lambda_r$ such that

 \begin{eqnarray*}
  M &=& \{g \in G; g(\bar\Lambda_i) = \bar\Lambda_i\} \\
  P &=& \{g \in G; g(\Lambda_i) = \Lambda_i\},
 \end{eqnarray*}
 where $\Lambda_{i} = \bar\Lambda_1 \oplus \ldots \oplus \bar\Lambda_i$. We get in the EL-case
 \[
  M \cong \prod_{i=1}^{r'} \GL_{O_F,n_i},
 \]
 where $n_i = \dim_F \overline{V_i}$ and $r'=r$.

 In the PEL-case we can assume that for any $i$
 \[
  \bar\Lambda_i^\perp = \bar\Lambda_1 \oplus \ldots \oplus \bar\Lambda_{r-i-1} \oplus \bar\Lambda_{r-i+1} \oplus \ldots \oplus \bar\Lambda_{r}.
 \]
 We denote $r' = \lfloor\frac{r}{2}\rfloor+1$. Then we have
 \[
  M \cong \left\{ \begin{array}{ll}
                  \prod\limits_{i=1}^{r'-1} \GL_{O_F,n_i} \times \GG_m & \textnormal{ if } r \textnormal{ is even} \\
                  \prod\limits_{i=1}^{r'-1} \GL_{O_F,n_i} \times \GU_{O_F,n_{r'}} & \textnormal{ if } r \textnormal{ is odd, } ^\ast = \id \\
                  \prod\limits_{i=1}^{r'-1} \GL_{O_F,n_i} \times \GSp_{O_F, n_{r'}} & \textnormal{ if } r \textnormal{ is odd, } ^\ast = \bar\cdot .
                 \end{array} \right.
 \]
 We also write this decomposition as $M = \prod M_i$ with $M_{r'}$ denoting the right factor in the PEL-case.

 After replacing $b$ by a $G(L)$-$\sigma$-conjugate, we may assume that $b \in M(L)$. Denote by $[b]_M$ the $\sigma$-conjugacy class in $M(L)$. After $\sigma$-conjugating by an element of $M(L)$ we furthermore assume that $b$ satisfies a decency equation (\ref{term decent}). Denote by $b_i$ the images of $b$ in $M_i$. Also note that the $b_i$ satisfy the same decency equation in $M_i$ as $b$.
  
 Let $k_0$ be finite field containing $k_F$ and $\FF_{p^s}$ where the $s$ is given by the decency equation of $b$. We choose {\BTpELs} $\XXbar_i$ over $k_0$ with $\Bscr$-isocrystals isomorphic to $(V_{i,K_0}, b\sigma)$ having the property that in the PEL-case the pairing $\pair$ induces an isomorphism $\XXbar_i \cong \XXbar_{r-i}^\vee$. In particular the isogeny class of $\XX = \XXbar_1 \oplus \ldots \oplus \XXbar_r$ corresponds to the isomorphism class of $(V,b\sigma)$ (with additional structure). We denote $\XX_i = \XXbar_1 \oplus \ldots \oplus \XXbar_i$.

 The following objects will allow us to relate $\Mscr_G(b,\mu)$ to the (underlying reduced subschemes of) Rapoport-Zink spaces corresponding to the $\sigma$-conjugacy classes $[b_i]_{M_i}$.

 \begin{definition}
  Using the notation introduced above we make the following definitions.
  \begin{subenv}
   \item Let $\Mscr_P(b,\mu)$ be the functor associating to each scheme $S$ over $k_0$ the following data up to isomorphism.
   \begin{itemlist}
    \item $(X,\rho) \in \Mscr_G(b,\mu)(S)$ and
    \item a filtration $X_\bullet = (X_1 \subset \ldots \subset X_r = X)$ of $X$ such that the restriction $\rho_{|\XX_i}$ defines a quasi-isogeny onto $X_i$.
   \end{itemlist}
   \item Similarly, let $\Mscr_M(b,\mu)$ be the functor associating the following data (up to isomorphism) to a scheme $S$ over $k_0$.
   \begin{itemize}
    \item $(X,\rho) \in \Mscr_G(b,\mu)(S)$ and
    \item a direct sum decomposition $X = \Xbar_1 \oplus \ldots \oplus \Xbar_r$ such that the restriction $\rho_{|\overline{\XX_i}}$ defines a quasi-isogeny onto $\Xbar_i$.
   \end{itemize}
  \end{subenv}
 \end{definition}

 We note that if $(X,X_\bullet,\rho) \in \Mscr_P(b,\mu)(S)$ then $X_i = \rho(\XX_i)$. Thus the filtration $X_\bullet$ is uniquely determined by $(X,\rho)$ if it exists, i.e.\ if $\rho(\XX_i)$ is a \BT. Following the proof of \cite{mantovan08}, Prop.~5.1, we conclude that $i_P: \Mscr_P(b,\mu) \rightarrow \Mscr_G(b,\mu), (X,X_\bullet,\rho) \mapsto (X,\rho)$ is the decomposistion of $\Mscr_G(b,\mu)$ into locally closed subsets given by the invariant $(\height(\rho_{|\XX_i}))_i$, i.e.\ there is a canonical isomorphism
 \[
  \Mscr_P(b,\mu) \cong \coprod_{(\alpha_i)\in\ZZ^r} \{ (X,\rho) \in \Mscr_G(b,\mu)(k) \mid \height (\rho_{|\XX_i}) = \alpha_i \textnormal{ for all } i \}
 \] 
 with the reduced subscheme structure on the right hand side which identifies $i_P$ with the canonical embedding of the components of the coproduct into $\Mscr_G(b,\mu)$.
  In particular, $\Mscr_P(b,\mu)$ is locally of finite type and $\dim \Mscr_P(b,\mu) = \dim \Mscr_G(b,\mu)$

 Now the $\XX_i$ are $O_F$-stable and in the PEL-case $\lambda_\XX$ induces isomorphisms $\XX_i \stackrel{\sim}{\rightarrow} (\XX/\XX_{r-i})^\vee$. As $\rho$ is compatible with the PEL structure, the analogous compatibility assertion also holds for the filtration $X_\bullet$ with respect to $\lambda$ and $\iota$. Hence we have a canonical $O_F$-action and polarisation on $\bigoplus_{i=1}^r X_i/X_{i-1}$ (where $X_0 := 0$). Thus we get a morphism 
 \[
  p_M:\Mscr_P(b,\mu) \to \Mscr_M(b,\mu), (X,X_\bullet,\rho) \mapsto (\bigoplus_{i=1}^r X_i/X_{i-1},\oplus\bar\rho_i)
 \]
 where the $\bar\rho_i$ denote the induced isogenies on the subquotients. So we have a correspondence

 \begin{equation}
 \begin{tikzcd} \label{diag correspondence}
  & \Mscr_P(b,\mu) \arrow{ld}[swap]{p_M} \arrow{rd}{i_P} \\
  \Mscr_M(b,\mu) & & \Mscr_G(b,\mu)
 \end{tikzcd}
 \end{equation}

 Now we want to describe $\Mscr_M(b,\mu)$ in terms of Rapoport-Zink spaces. Analogously to our consideration above we get that for $(\bigoplus_i \overline{X_i},\rho) \in \Mscr_M(b,\mu)(S)$ the $\Xbar_i$ are $O_F$-stable and that $\lambda$ induces isomorphisms $\overline{X_i} \stackrel{\sim}{\to} \Xbar_{r-i}^\vee$. As in \cite{RV}~ch.~5.2 we get an isomorphism
 \begin{eqnarray*}
  \Mscr_M(b,\mu) &\stackrel{\sim}{\longrightarrow}& \coprod_{\mu_M \in I_{\mu,b,M}} \Mscr_{M,b,\mu_M} \\
  (\bigoplus_i \Xbar_i, \rho) &\mapsto& (\Xbar_i,\rho_{|\XXbar_i})_i
 \end{eqnarray*}
 where $I_{\mu,b,M}$ denotes the $M(\QQbar_p)$-conjugacy classes of cocharacters $\mu_{M}: \GG_{m,\QQbar_p} \to M_{\QQbar_p}$ which are contained in the conjugacy class of $\mu$ and satisfy $b \in B(M,\mu_M)$ and $\Mscr_{M,b,\mu_M}$ is defined by 
 \[
  \Mscr_{M,b,\mu_M} = \prod_{i=1}^{r'} \Mscr_{M_i}(b_i,\mu_{M,i})
 \]

 Now consider again the diagram (\ref{diag correspondence}). As we have $\dim \Mscr_P(b,\mu) = \dim \Mscr_G(b,\mu)$, it remains to calculate the dimension of the fibres of $p_M$  to relate the dimension of $\Mscr_G(b,\mu)$ to the dimension of $\Mscr_M(b,\mu)$. 

 For this we consider the $k$-valued points of the diagram (\ref{diag correspondence}) with the natural action of $\Gamma_0 = \Gal(k/k_0)$ on the set of points. By section~\ref{ss RZ spaces} we have
 \begin{eqnarray*}
  \Mscr_G(b,\mu)(k) &=& \{g \in G(L)/G(O_L); gb\sigma(g)^{-1} \in G(O_L)\mu(p)G(O_L)\} \\
  \Mscr_P(b,\mu)(k) &=& \{g \in P(L)/P(O_L);\, gb\sigma(g)^{-1} \in G(O_L)\mu(p)G(O_L)\} \\
  \Mscr_M(b,\mu_M)(k) &=& \{m \in M(L)/M(O_L);\, m b\sigma(m)^{-1} \in G(O_L) \mu(p) G(O_L)\}
 \end{eqnarray*}
 and
 \begin{eqnarray*}
  \quad\quad\,\,\,\,\Mscr_{M,b,\mu_M}(k) &=& \{ m \in M(L)/M(O_L);\, mb\sigma(m)^{-1} \in M(O_L)\mu_M(p)M(O_L)\}
 \end{eqnarray*}
 with the canonical $\Gamma_0$-action.

 Now the diagram (\ref{diag correspondence}) induces
 \begin{equation}
 \begin{tikzcd} \label{diag correspondence 2}
  & \Mscr_P(b,\mu)(k) \arrow{ld}[swap]{p_M: mnP(O_L) \mapsto mM(O_L)} \arrow{rd}{i_P: gP(O_L) \mapsto gG(O_L)} \\
  \Mscr_M(b,\mu)(k) & & \Mscr_G(b,\mu)(k)
 \end{tikzcd}
 \end{equation}
 Here $p_M$ uses the decomposition $P \cong M\times N$ where $N$ denotes the unipotent radical of $P$.

 \begin{proposition} \label{prop fibre dimension}
  Let $x \in \Mscr_{M,b,\mu_M}$. Then
  \[
   \dim p_M^{-1}(x) = \langle\rho, \mu-\nu_G(b) \rangle -\langle \rho_M, \mu_M\rangle.
  \]
  where $\rho$ resp.\ $\rho_M$ denotes the half-sum of all positive (absolute) roots in $G$ resp.\ $M$.
 \end{proposition}

 Before we can finally prove this proposition at the end of section~\ref{sect fibre dimension}, we need to establish some notions and lemmas.
 
 We denote $\Fscr := p_M^{-1}(x)$. Choose an element $m \in M(L)$ such that $x = mM(O_L)$. If we replace $\Lambda$ by $m\cdot\Lambda$, we may assume that $m = \id$. Note that this replaces $b$ by a $\sigma$-conjugate. However, the formula above shows that the dimension of $\Fscr$ does not change if we replace $b$ by a $\sigma$-conjugate.

 \section{Numerical dimension of subsets of $N(L)$} \label{sect numerical dimension}

 Let $\tilde{\Fscr}$ denote the preimage of $\Fscr(k)$ in $N(L)$ with respect to the canonical projection $N(L) \twoheadrightarrow N(L)/N(O_L)$. In this section we define the notion of ``numerical dimension'' for certain subsets of $N(L)$ and show that with respect to this notion the dimension of $\tilde\Fscr$ coincides with the dimension of $\Fscr$. We will calculate the numerical dimension of $\tilde{\Fscr}$ and thus prove Proposition \ref{prop fibre dimension} in the next section.

 \subsection{The concept of numerical dimension} \label{ss idea}
 Let us first sketch the idea behind the numerical dimension. The starting point is the following result of Lang and Weil. If $X$ is variety over $\FF_q$ of dimension $d$ then the cardinality of $X(\FF_{q^s})$ grows with the same speed as $q^{s\cdot d}$. In particular the dimension of $X$ is uniquely determined by the $\Gal(k/\FF_q)$-set $X(k)$. Now if we tried to calculate the dimension of $\Fscr$ by applying theorem and counting points, we would run into two problems.

 First of all $\Fscr$ is only locally of finite type. We will solve this problem by defining an exhausting filtration of $N(L)$ by so-called bounded subsets which correspond to closed quasi-compact (and thus finite type) subschemes of $\Fscr$.

 Now the second problem is that Galois-sets without a geometric background are in general too badly behaved to use them. Here the general idea is that by applying theorem of Lang and Weil two times, one sees that any finite type scheme $Y$ with $Y(k) \cong X(k)$ as Galois-sets has the same dimension as $X$. This applies to our situation as follows.

 Let $(\Fscr_m)_{m\in\NN}$ be the above mentioned filtration by finite type subschemes and $(\tilde\Fscr_m)_{m\in\NN}$ be the corresponding bounded subsets in $N(L)$. Unfortunately, it is not known (and probably not true) whether there is any good structure of an ind-scheme of ind-finite type on $N(L)/N(O_L)$ which induces a useful scheme structure on $\tilde\Fscr_m/N(O_L)$. Therefore we will use the following work-around. We replace $\Fscr_m$ by its image $c(\Fscr_m)$ under the conjugation by a suitable semisimple element of $M(L)$ such that $c(\Fscr_m)$ is contained in $N(O_L)$. Here we have a canonical structure of affine spaces on the quotients $N(O_L/p^j)$ given by the truncated $p$-adic loop groups. We will prove that $c(\tilde\Fscr_m)$ is the full preimage of a locally closed subvariety $\Ybar$ of $N(O_L/p^j)$ if $j$ is big enough. Thus we can define the numerical dimension of $\tilde\Fscr_m$ via the dimension of $\Ybar$.

 We note that $\tilde\Fscr_m/N(O_L)$ will in general not be isomorphic to $\Ybar(k)$ (yet it will be the analogue of an affine fibration over $\Ybar(k)$). This is because we have shrunken $\tilde\Fscr_m$ by applying $c$ and have only divided out a subgroup of $N(O_L)$ instead of $N(O_L)$ itself. We can (and will) compensate for these deviation by adding a constant to $\dim \Ybar$ depending on $c$ and subtracting the dimension of $N(O_L/p^j)$ as variety.

 This definition of numerical dimension will give us the same value as we would get if we were counting points (cf.~Proposition~\ref{prop nd counting points}), but has the advantage that it also allows us to use some tools from geometry.

 \subsection{Notation and conventions}
 Let $G$ be a reductive group scheme over $\ZZ_p$. We impose the same notions as in section~\ref{ss group theory}. Moreover, we denote $I := \Gal (O_F/\ZZ_p)$. In particular the action of $\Gamma$ on $X^*(T)$ and $X_*(T)$ factorizes through $I$. For any $\alpha \in R$ denote by $\gfr^\alpha$ the corresponding weight space in the Lie algebra of $G_{O_F}$ and by $U_\alpha$ the corresponding root subgroup. By the uniqueness of $U_\alpha$ (\cite{SGA3-3}~Exp.~XXII~Thm.~1.1) the action of an element $\tau \in I$ on $G_{O_F}$ maps $U_\alpha$ isomorphically onto $U_{\tau(\alpha)}$. 

 We fix a standard parabolic subgroup $P = MN$ of $G$. We write $K = G(O_L), K_M = M(O_L)$ and $K_N = N(O_L)$. Now
 \[
  \Lie P_{O_F} = \bigoplus_{\alpha\in R'} \gfr^\alpha 
 \]
 for a closed $I$-stable subset $R'\subset R$. We denote $R_N := \{\alpha \in R^+\mid -\alpha \not\in R'\}$. Then multiplication in $G$ defines an isomorphism of schemes
 \begin{equation} \label{term N isom}
  N_{O_F} \cong \prod_{\alpha \in R_N} U_\alpha
 \end{equation}
 where the product is taken with respect to an arbitrary (but fixed) total order on $R_N$ (cf.~\cite{SGA3-3} 
 Exp.~XXVI~ch.~1).

 Let $\delta_N$ be the sum of all fundamental coweights corresponding to simple roots in $R_N$. For $i \geq 1$ we define the group schemes
 \begin{eqnarray*}
  N[i] &:=& \prod_{\alpha\in R_N \atop \langle \alpha, \delta_N\rangle \geq i} U_\alpha \subseteq N \\
  N\langle i \rangle &:=& \bigslant{N[i]}{N[i+1]} 
 \end{eqnarray*}
 We note that the sets $\{\alpha\in R_N\mid \langle \alpha, \delta_N \rangle \geq i\}$ are $I$-stable, thus the $I$-action permutes the $U_\alpha$ in the above product. Now the commutator $[u_\alpha,u_\beta]$ of two elements $u_\alpha\in U_\alpha, u_\beta\in U_\beta$ is contained in $\prod U_{i\alpha+j\beta}$, which is contained in $N[i]$ if $U_\alpha$ and $U_\beta$ are. We conclude that the $I$-action on $N$ stabilizes $N[i]$. Thus $N[i]$ descends to $\ZZ_p$ and a posteriori also $N\langle i \rangle$. Also note that the canonical isomorphism of schemes
 \[
  N\langle i \rangle \cong \prod_{\alpha \in R_N \atop \langle \alpha, \delta_N \rangle = 1} U_\alpha
 \]
 is in fact an isomorphism of group schemes.

 Let $\lambda$ be a regular dominant coweight of $T$ defined over $\ZZ_p$. For $i\in\ZZ$ we define $N(i) := \lambda(p^i)N(O_L)\lambda(p^{-i})$. This defines an exhausting filtration
 \[
  \ldots N(-2) \subset N(-1) \subset N(0) \subset N(1) \subset N(2) \subset \ldots
 \]
 of $N(L)$. For every subset $Y \subset N(L)$ we denote $Y(i) := \lambda(p^i)Y\lambda(p^{-i})$.

 For any algebraic group $H$ over $\ZZ_p$ and any integer $i$ we denote $H_i = \ker (H(O_L) \epi H(O_L/p^iO_L))$. In particular, we get a second filtration
 \[
  N(0) = N_0 \supset N_1 \supset N_2 \supset \ldots
 \]
 of $N(0)$.

 \begin{remark} \label{rem N(j)}
  The isomorphism (\ref{term N isom}) induces an homeomorphism
  \[
   N(L) \cong \prod_{\alpha\in R_N} L
  \]
  which yields identifications
  \begin{eqnarray*}
   N(j) &=& \prod_{\alpha\in R_N} p^{j\langle \alpha, \delta_N\rangle} O_L \\
   N_j  &=& \prod_{\alpha\in R_N} p^jO_L 
  \end{eqnarray*}
 \end{remark}

 \begin{definition}
  A subset $Y\subset N(L)$ is called bounded if it is contained in $N(-j)$ for an $j \in \ZZ$.
 \end{definition}
 
 \begin{lemma} \label{lem bounded}
  A subset $Y \subset N(L)$ is bounded if and only if it is bounded as a subset of $G(L)$ in the sense of Bruhat and Tits.
 \end{lemma}
 \begin{proof}
  To avoid confusion we temporarily call $Y$ ``BT-bounded'' if it is bounded in the sense of Bruhat and Tits. By definition a subset $Y \subset N(L)$ is BT-bounded if and only if $\val_p (f(Y))$ is bounded from below for every $f\in \Gamma(G,\Oscr_G)$, or equivalently for every $f\in\Gamma(N,\Oscr_N)$. Using the isomorphism $N_L \cong \AA_L^{R_N}$, we may reformulate the condition as follows. The set $Y \subset L^{R_N}$ is BT-bounded if and only if $\val_p (f(Y))$ is bounded from below for every $f\in L[X_\alpha]_{\alpha\in R_N}$. It obviously suffices to check this for the coordinate functions. Thus $Y$ is BT-bounded if and only if $Y\subset(p^{-k}O_L)^{R_N}$ for some integer $k$. But this is equivalent to $Y$ being bounded by the description given in Remark~\ref{rem N(j)}.
 \end{proof}

 \subsection{Admissible subsets of $N(0)$ and the $p$-adic loop group}
 As a first step we now define the numerical dimension for a family of subsets of $N(0)$. We will extend this definition to certain subsets of $N(L)$ (called ``admissible'') in the next subsection. Before that we give a reminder on $p$-adic loop groups.

 \begin{definition} 
  Let $H$ be an affine smooth group scheme over $\ZZ_p$.
  \begin{subenv}
  \item The $p$-adic loop group is the ind-affine ind-scheme over $\FF_p$ representing the functor
   \[
    L_pH (R) = H(W(R)_\QQ).
   \]
  \item The positive $p$-adic loop group is the affine scheme over $\FF_p$ with $R$-valued points
  \[
   L_p^+H(R) = H(W(R)).
  \]
  \item Let $j$ be a positive integer. We define the positive $p$-adic loop group truncated at level $j$ as the affine scheme of finite type over $\FF_p$ representing the functor
  \[
   L_{p}^{+,j}H(R) = H(W_j(R)).
  \]
  \end{subenv}
 \end{definition}

 For the proof that the functors above can be represented as claimed, we refer the reader to \cite{kreidl}~sect.~3. 

 \begin{lemma} \label{lem truncation map is open}
  Let $H$ be an affine smooth group scheme over $\ZZ_p$. The truncation maps $t_j:L_p^+ H \to L_p^{+,j} H$ are open and surjective.
 \end{lemma}
 \begin{proof}
  We have for any $\FF_p$-algebra $R$
  \begin{eqnarray*}
   L_p^+H(R) &=& \Mor( \Spec \varprojlim W_j(R), H) \\
   &=& \varinjlim \Mor(\Spec W_j(R), H) \\
   &=& \varinjlim L_p^{+,j}H(R).
  \end{eqnarray*}
  Thus $L_p^+H$ is the projective limit of the $L_p^{+,j}H$ and the truncation maps are the canonical projections. In particular we have an homeomorphism of the underlying topological spaces $L_p^+ H \cong \varprojlim L_p^{+,j}H$ by \cite{EGA4-3}, Cor.~8.2.10. Thus it remains to show that the transition maps are surjective. But this follows from the infinitesimal lifting property, as $H$ is smooth.
 \end{proof}

 We note that $L_pH(k) = H(L)$, $L_p^+H(k) = H(O_L)$ and that we have a canonical isomorphism of $\Gamma$-groups
 \[
  L_{p}^{+,j}H(k) \cong H_0/H_j.
 \]
 In our future considerations we equip the quotient $N_0/N_j$ with the structure of a variety over $\FF_p$ via this identification. We denote
 \[
  d_j := \dim L_p^{+,j}N
 \]

 Now our considerations in subsection~\ref{ss idea} motivate the following definition.

 \begin{definition} \label{def admissible}
  \begin{subenv}
   \item A subset $Y \subset N_0$ is called admissible if there exists a positive integer $j$ such that $Y$ is the preimage of a locally closed subset $\overline{Y} \subset N_0/N_j$.
   \item We define the numerical dimension of an admissible subset as
   \[
    \nd Y := \dim \Ybar - d_j
   \]
   with $\Ybar$ and $j$ as above.
  \end{subenv}
 \end{definition}

 Note that the definition of the numerical dimension is certainly independent of the choice of $j$.

 \subsection{Admissible and ind-admissible subsets of $N(L)$}
 It is a straightforward idea to define admissibility of bounded subsets of $N(L)$ by checking admissibility for a $\lambda(p^i)$-conjugate which is a subset of $N_0$. For this definition to make sense, we have to check that any $\lambda(p^i)$-conjugate of an admissible subset of $N_0$ is again admissible.

  \begin{lemma} \label{lem admissible is well-defined}
  Let $Y \subset N_0$ and $i$ be a positive integer.
  \begin{subenv}
   \item $Y(i)$ is admissible if and only if $Y$ is admissible.
   \item If $Y$ is admissible then $Y \cap N(i)$ is admissible.
  \end{subenv}
 \end{lemma}
 \begin{proof}  \emph{(1)} Assume first that $Y$ is admissible. Let $j,\overline{Y}$ as in Definition \ref{def admissible} and let $\Yscr := t_j^{-1}(\overline{Y})$ where $\overline{Y}$ is regarded as sub\emph{scheme} of $L_p^{+,j}N$. Then $\Yscr$ is a locally closed subset of $L_p^+N$, which we equip with the reduced subscheme structure. Since $L_p^+N$ is a closed subfunctor of $L_pN$, the functor $\Yscr$ is also locally closed in $L_pN$. Now conjugation with $\lambda (p^i)$ defines an automorphism of $L_pN$, hence $\Yscr' := \lambda(p^i)\Yscr\lambda(p^{-i})$ is again a locally closed subfunctor of $L_pN$ and thus a locally closed subscheme of $L_p^+N$. We note that $Y(i) = \Yscr'(\Spec k)$ and  that $\Yscr'$ is the preimage of a subset $\overline{Y}' \subset L_p^{+,j'}N$ for $j'$ big enough such that $N_{j'} \subset \lambda(p^i)N_j\lambda(p^i)$ (One checks the second assertion on geometric points). Thus $\Yscr' = t_{j'}^{-1}(t_{j'} \Yscr')$,  hence we have $t_ {j'}(\overline{\Yscr'}) = \overline{t_{j'}(\Yscr')}$ as $t_{j'}$ is open and surjective. So the restriction $t_{j'}: \overline{\Yscr'} \rightarrow \overline{t_{j'}(\Yscr')}$ is again open, in particular $\Ybar'=t_{j'}(\Yscr)$ is open in its closure. Hence $Y'$ is admissible. The other direction has the same proof.

 \emph{(2)} We have seen that $N(i)$ is admissible and obviously the intersection of two admissible sets is again admissible, proving the claim.
 \end{proof}

 \begin{definition}
  \begin{subenv}
   \item A subset $Y \subset N(L)$ is called admissible, if $Y \subset N(-k)$ for some non-negative integer $k$ and $\lambda(p^k)Y\lambda(p^{-k})$ is admissible in the sense of Definition~\ref{def admissible}
   \item A subset $Y \subset N(L)$ is called ind-admissible if $Y \cap N(-k)$ is admissible for all non-negative integers $k$.
  \end{subenv}
 \end{definition}

  We note that by Lemma~\ref{lem admissible is well-defined}~(1), the definition of admissibility does not depend on the choice of the integer $k$ and by Lemma~\ref{lem admissible is well-defined}~(2) that every admissible subset is also ind-admissible.

  Before we can introduce our new notion of dimension, we have to show a few auxiliary results to ensure that our definition will be well-defined. First, we recall a result of Lang and Weil which reduces the calculation of the dimension of certain schemes to the counting of points.

  \begin{definition}
  Let $q$ be a $p$-power and let $f,g$ be two functions defined on the set of $q$-powers with values in the non-negative integers. Then we write $f \sim g$ if there exist positive real constants $C_1,C_2$ with $ C_1 f (q^n) \leq g(q^n) \leq C_2 f(q^n)$ for $n \gg 0$.
 \end{definition}

 \begin{proposition} \label{prop lang weil}
  Let $V$ be a scheme of finite type over $\FF_q$. Then
  \[
   \# V(\FF_{q^n}) \sim q^{n \dim V}.
  \]
 \end{proposition}
 \begin{proof}
  This is an easy consequence of \cite{LW54}, Thm.~1.
 \end{proof}

 \begin{lemma} \label{lem nd is well-defined}
  Let $Y \subset N(0)$ be admissible and let $i$ be a positive integer. Then
  \[
   \nd Y(i) = \nd Y - 2 \langle\rho_N, i \lambda\rangle.
  \]
 \end{lemma}
 \begin{proof}
  Choose $j$ such that $Y(i)$ (and thus $Y$) is $N_j$-stable and denote by $\Ybar$ (resp. $\Ybar(i)$) their images in $N_0/N_j$. Let $\cbar_i: N_0/N_j \to N_0/N_j$ be the morphism induced by conjugation with $\lambda(p^i)$. Then $\Ybar$ is the full preimage of $\Ybar(i)$ w.r.t. $\cbar_i$ and the restriction $\cbar_{i | \Ybar}: \Ybar \to \Ybar(i)$ is surjective. Thus
  \[
   \dim \Ybar = \dim \Ybar (i) + \dim\ker \cbar_i
  \]
  Now by Proposition~\ref{prop lang weil} we have
  \[
   q_F^{(\dim \ker \cbar_j)\cdot s} \sim \prod_{\alpha\in R_N} \#\ker(W_j(\FF_{q_F^s}) \stackrel{\cdot p^{\langle\alpha, i\lambda\rangle}}{\lto} W_j(\FF_{q_F^s})) = \prod_{\alpha\in R_N} p^{(\min\{j,\langle\alpha, i \lambda\rangle\})\cdot s}.
  \]
  Thus for $j$ big enough, which we may assume (and actually is automatic),
  \[
   \dim \ker \cbar_i = \sum_{\alpha\in R_N} \langle\delta_N, i\lambda\rangle = 2 \langle \rho_N, i \lambda \rangle.
  \]
 \end{proof}

 We denote
 \[
  d(i) = 2 \langle\rho_N, i\lambda\rangle.
 \]

 \begin{definition}
  \begin{subenv}
   \item Let $Y \subset N(-k)$ be admissible. We define the numerical dimension of $Y$ as
   \[
    \nd Y := \nd Y(k) + d(k).
   \]
   \item The numerical dimension of an ind-admissible subset $Y \subset N(L)$ is defined as
   \[
    \nd Y = \sup_{k>0} \nd (Y \cap N(-k)).
   \]
  \end{subenv}
 \end{definition}

 \begin{corollary} \label{cor calculation of nd}
  Let $Y \subset N(L)$ be admissible and $m\in M(L)$. Then $mYm^{-1}$ is admissible and has numerical dimension $\nd Y -2\langle \rho_N, \nu_M(m)\rangle$.
 \end{corollary}
 \begin{proof}
  By Cartan decomposition it suffices to consider the following two cases.

  \noindent If $m = \lambda'(p)$ for a coweight $\lambda'$ which is dominant w.r.t.\ $T_{O_L} \subset B_{O_L} \subset M_{O_L}$ then the claim follows as in Lemma~\ref{lem nd is well-defined}.

  \noindent If $m \in M(O_L)$ then conjugation with $m$ stabilizes the $N_j$, thus the admissibility and the numerical dimension do not change. 
 \end{proof}

 \subsection{Equality of $\dim \Fscr$ and $\nd \tilde\Fscr$}
 By Proposition~\ref{prop lang weil} the dimension of $\Fscr$ can be determined by knowing the cardinality of $\Fscr(k)^{\sigma_E^s}$ for any positive integer $s$. In order to relate $\dim \Fscr$ to $\nd \tilde\Fscr$ (under the assumption that the latter is well-defined) we need a similar assertion for the numerical dimension. For this we need to do   bit of preparatory work.

 \begin{lemma} 
  Let $F\subset \QQ_{p^s}$. We have short exact sequences

  \begin{tikzcd}
   0 \arrow{r} & N[i+1](L)^{\sigma^s} \arrow{r} & N[i](L)^{\sigma^s} \arrow{r} & N\langle i \rangle(L)^{\sigma^s} \arrow{r} & 0
  \end{tikzcd}
 
  \noindent and

  \begin{tikzcd}
   0 \arrow{r} & N[i+1](O_L)^{\sigma^s} \arrow{r} & N[i](O_L)^{\sigma^s} \arrow{r} & N\langle i \rangle(O_L)^{\sigma^s} \arrow{r} & 0.
  \end{tikzcd}

 \end{lemma}

 \begin{proof}
  This is an easy consequence of the description of $N[i]$ resp.\ $N\langle i \rangle$ as product of root groups.
 \end{proof}
 
 \begin{lemma}
  Let $F \subset \QQ_{p^s}$. For any $i$, the map $N[i](L)^{\sigma^s} \rightarrow (N[i](L)/N[i](O_L))^{\sigma^s}$ is surjective.
 \end{lemma}
 \begin{proof}
  We prove the claim by descending induction on $i$. For the induction beginning choose $i \gg 0$ such that $N[i] = 0$. Then the claim is certainly true. Now assume that our assertion is true for $i+1$. Then we get a commutative diagram

 \begin{tikzcd}[column sep = small]
  0  \arrow{r}
  & N[i+1](O_L)^{\sigma^s} \arrow{r}\arrow{d}
  & N[i](O_L)^{\sigma^s} \arrow{r}\arrow{d}
  & N\langle i \rangle(O_L)^{\sigma^s} \arrow{r}\arrow{d} 
  & 0 \\
  0 \arrow{r} 
  & N[i+1](L)^{\sigma^s} \arrow{r}{\phi}\arrow[two heads]{d} 
  & N[i](L)^{\sigma^s} \arrow{r}{\psi}\arrow{d}
  & N\langle i \rangle(L)^{\sigma^s} \arrow{r}\arrow[two heads]{d}
  & 0 \\
  & \bigslant{N[i+1](L)}{N[i+1](O_L)}^{\sigma^s} \arrow{r}{\overline{\phi}}
  & \bigslant{N[i](L)}{N[i](O_L)}^{\sigma^s} \arrow{r}{\overline{\psi}}
  & \bigslant{N\langle i \rangle(L)}{N\langle i \rangle (O_L)}^{\sigma^s}.
 \end{tikzcd}

 One easily checks that $\overline{\phi}$ is injective, $\overline{\psi}$ surjective and that $\image \overline{\phi} = \ker \overline{\psi}$ in the category of pointed sets.

 We choose an element $\overline{n} \in (N[i](L)/N[i](O_L))^{\sigma^s}$. By diagram chasing we find an $n\in N(L)^{\sigma^s}$ such that $n$ and $\overline{n}$ have the same image $\hat{n}$ in $(N\langle i \rangle (L) / N\langle i \rangle (O_L))^{\sigma^s}$. Thus $n^{-1} \cdot \overline{n} \in \ker \overline{\psi}$ and we find $n' \in N[i+1](L)^{\sigma^s}$ such that $n'$ is mapped to $n^{-1} \cdot \overline n$. Hence $n\phi(n') \in N(L)^{\sigma^s}$ is mapped to $\overline{n}$, finishing the proof.

 \end{proof}

 \begin{corollary} \label{cor nd is well-defined}
  Let $F \subset \QQ_{p^s}$. For any integer $j$, the map $N(L)^{\sigma^s} \rightarrow (N(L)/N(j))^{\sigma^s}$ is surjective.
 \end{corollary}
 \begin{proof}
  By conjugating with $\lambda(p)$, we see that it suffices to prove the claim for $j=0$. Now our assertion coincides with the assertion of the previous lemma for $i=0$.
 \end{proof}

 \begin{lemma} \label{lem N_j is weakly admissible}
  The canonical map $N_0^{\sigma_F} \rightarrow (N_0/N_j)^{\sigma_F}$ is surjective.
 \end{lemma}
 \begin{proof}
  (1) We show that $H^1_\cont (\Gamma_F, N[i]_j) = 0$ by descending induction on $i$. In particular for $i=0$ we get $H^1_\cont (\Gamma_F,N_j) = 0$ and the claim of the lemma follows. For $i \gg 0$ we have $N[i]_j = 0$ and thus $H^1_\cont(\Gamma_F,N[i]) = 0$. Now  the short exact sequence

  \begin{tikzcd}
   0 \arrow{r}
   & N[i+1]_j \arrow{r}
   & N[i]_j \arrow{r}
   & N\langle i \rangle_j \arrow{r}
   & 0
  \end{tikzcd}

  \noindent gives the exact sequence

  \begin{tikzcd}
   H^1_\cont(\Gamma, N[i+1]_j) \arrow{r}
   & H^1_\cont(\Gamma, N[i]_j) \arrow{r}
   & H^1_\cont(\Gamma, N\langle i \rangle_j).
  \end{tikzcd}
  
  \noindent Under the identification $N\langle i \rangle (L) \cong \prod L$ the group $N\langle i \rangle_j$ is identified with $\prod p^j O_L$. So $N\langle i \rangle_j$ is isomorphic  to a finite product of copies of $O_L$ and thus has trivial cohomology.  We may assume that $H^1_\cont(\Gamma,N[i+1]) = 0$ by induction assumption, then the sequence above implies that also $H^1_\cont(\Gamma_F,N[i])=0$.
 \end{proof}

 \begin{proposition}\label{prop nd counting points}
  Let $Y \subset N(L)$ be admissible and $N(i)$-stable, assume that $F$ is big enough such that $Y$ is $\sigma_F$-stable. Then
  \[
   (\bigslant{Y}{N(i)})^{\sigma_F^s} \sim q^{(\nd Y + d(i)) \cdot s}.
  \]
 \end{proposition}
 \begin{proof}
  Choose $k$ such that $Y \subset N(-k)$. Then conjugation with $\lambda(p^k)$ induces an isomorphism of $\Gamma_F$-sets
  \[
   \bigslant{Y}{N(i)} \cong \bigslant{Y(k)}{N(k+i)}
  \]
  We choose $j$ such that $N_j \subset N(k+i)$. Let $\pi^{(s)}: (Y(k)/N_j)^{\sigma_F^s} \to (Y(k)/N(k+i))^{\sigma_F^s}$ be the map induced by the canonical projection. We get a commutative diagram.
  \begin{center}
  \begin{tikzcd}[column sep = large]
   & Y(k)^{\sigma_F^s} \arrow[two heads]{dr} \arrow[two heads]{dl}
   & \\
   (\bigslant{Y(k)}{N_j})^{\sigma_F^s} \arrow{rr}{\pi^{(s)}}
   &
   & (\bigslant{Y(k)}{N(k+i)})^{\sigma_F^s}
  \end{tikzcd}
  \end{center}
  Thus $\pi^{(s)}$ is surjective. Each of its fibres is canonically bijective to
  \[
   \bigslant{N(k+i)^{\sigma_F^s}}{N_j^{\sigma_F^s}} \cong \prod_{\alpha\in R_N} \bigslant{(p^{(i+k)\cdot \langle \delta_N,\alpha\rangle} O_{F_s})}{p^j O_{F_s}}
  \]
  where $F_s$ denotes the (unique) unramified extension of $F$ of degree $s$. In particular every fibre has $q^{(d_j - d(i+k))\cdot s}$ elements. Altogether,
  \begin{eqnarray*}
   \# (Y/N(k))^{\sigma_F^s} &=& \# (Y(i)/N(k+i))^{\sigma_F^s} \\
   &=& \# (Y(k)/N_j)^{\sigma_F^s} \cdot q_F^{(d(k+i)-d_j)\cdot s} \\
   &\sim&  q_F^{(\nd Y(k) + d_j) \cdot s} \cdot q_F^{(d(k+i)-d_j) \cdot s} \\
   &=& q_F^{(\nd Y + d(i)) \cdot s}.
  \end{eqnarray*}
 \end{proof}

 \begin{proposition} \label{prop numerical dimension equals dimension}
  Assume that $\tilde\Fscr$ is ind-admissible. Then
  \[
   \dim \Fscr = \nd \tilde\Fscr.
  \]
 \end{proposition}
 \begin{proof}
   The obstacle that prevents us from applying Proposition \ref{prop lang weil} directly to $\Fscr$ is the fact that $\Fscr$ is not quasi-compact in general. Thus our method of proof is to find a filtration of $\Fscr$ by quasi-compact subschemes and compare it to the filtration of $\tilde\Fscr$ by $\tilde\Fscr \cap N(-k)$.

  Now as $\height\rho_{|\XX_i}$ is constant on $\Fscr$ for every $i$, the restriction of $i_P$ defines an isomorphism of $\Fscr$ onto its image in $\Mscr_G(b,\mu)$. Thus by \cite{RZ96}, Cor.~2.31 such a filtration is given by
  \[
   \Fscr_{-k} := \{ (X,\rho) \in \Fscr \mid p^k \rho \textnormal{ and } p^k \rho^{-1} \textnormal{ are isogenies}\}.
  \]
  Its preimage in $N(L)$ is
  \[
   \tilde\Fscr_{-k} := \{n \in N(L); p^k\Lambda \subset n \Lambda \subset p^{-k} \Lambda\}.
  \]
  Now by Proposition \ref{prop lang weil}  and \ref{prop nd counting points} we have
  \[
   \dim \Fscr = \sup \dim \Fscr_{-k} = \sup \nd \tilde\Fscr_{-k}. 
  \]
  It remains to show that the filtrations $(\tilde\Fscr_{-k})_{k\in\NN}$ and $(\tilde\Fscr \cap N(-k))_{k\in\NN}$ are refinements of each other. Indeed, by Bruhat-Tits theory the sets  $\tilde\Fscr_{-k}$ are bounded and any bounded subset of $\tilde\Fscr$ is contained in one of the sets $\tilde\Fscr_{-k}$.
 \end{proof}

 \section{Calculation of the fibre dimension} \label{sect fibre dimension}

 Let $K := G(O_L)$ and $K_M := M(O_L)$. Analogous to \cite{GHKR06} we define the function $f_{m_1,m_2}$ for $m_1,m_2 \in M(L)$ by
 \[
  f_{m_1,m_2}: N(L) \to N(L), n \mapsto m_1 n^{-1} m_1^{-1} \cdot m_2 \sigma(n)m_2^{-1}.
 \]
 Then we have
  \begin{eqnarray*}
  \tilde\Fscr &=& \{n \in N(L)\mid n^{-1}b\sigma(n) \in K\mu(p)K\} \\
  &=& \{n \in N(L)\mid n^{-1}b\sigma(n)b^{-1} \in K\mu(p)Kb^{-1} \cap N(L)\} \\
  &=& f_{1,b}^{-1} (K\mu(p) Kb^{-1} \cap N(L)).
 \end{eqnarray*}
 Hence we divide the computation of $\nd \tilde\Fscr$ into two steps:
 \begin{itemlist}
  \item We have to show that $K\mu(p)Kb^{-1} \cap N(L)$ is admissible and compute its dimension.
  \item We have to calculate the difference $\nd f_{m_1, m_2}^{-1} Y - \nd Y$ for admissible $Y \subset N(L)$.
 \end{itemlist}

 The maps $f_{m_1,m_2}$ are defined in greater generality than the functions we actually need, but this will turn out to be an advantage in the second step.

 \subsection{Admissibility and dimension of $K\mu(p)Kb^{-1} \cap N(L)$} \label{ssect fibre dimension 1}
 We note that we have two notions of dominant elements in $X_*(T)$, one coming from the Killing pair $T\subset B$ in $G$ and one coming from $T \subset B\cap M$ in $M$. Let $X_*(T)_\dom$ resp.\ $X_*(T)_{M-\dom}$ denote the set of cocharacters that are dominant in $G$ resp.\ $M$. 

 For $\mu_M \in X_*(T), \mu \in X_*(T)_{\dom}$, we denote
 \[
  C(\mu,\mu_M) := (N(L)\mu_M(p)K \cap K\mu(p)K)/K 
 \]
 considered as $\Gamma_F$-set. In order to calculate the numerical dimension of $K\mu(p)Kb^{-1} \cap N(L)$, we first study the sets $C(\mu,\mu_M)$.

 We denote for $\mu,\mu_M$ as above and $\kappa \in \pi_1(M)_\Gamma$
 \begin{eqnarray*}
  S_M(\mu) &:=& \{\mu_M \in X_*(T)_{M-\dom}\mid C(\mu,\mu_M) \not= \emptyset\} \\
  S_M(\mu,\kappa) &:=& \{\mu_M \in S_M(\mu,\kappa)\mid \kappa_M(\mu_M)=\kappa\}.
 \end{eqnarray*}
 We will compare these sets to
 \begin{eqnarray*}
  \Sigma(\mu) &:=& \{\mu' \in X_*(T)\mid \mu'_\dom \leq \mu\} \\
  \Sigma(\mu)_{M-\dom} &:=& \{\mu' \in \Sigma(\mu)\mid \mu' \in X_*(T)_{M-\dom} \} \\
  \Sigma(\mu)_{M-\mmax} &:=& \{\mu' \in \Sigma(\mu)_{M-\dom}\mid \mu' \textnormal{ is maximal w.r.t.\ the Bruhat-order of } M\}.
 \end{eqnarray*}

 \begin{lemma}
  There are inclusions
  \[
   \Sigma(\mu)_{M-\mmax} \subset S_M(\mu) \subset \Sigma(\mu)_{M-\dom}.
  \]
 \end{lemma}
 \begin{proof}
  This is the $p$-adic analogue of Lemma 5.4.1 in \cite{GHKR06}. The proof is literally the same.
 \end{proof}

 \begin{remark}
  Assume that we are in the situation of Prop.~\ref{prop fibre dimension}. Then
  \[
   I_{M,\mu,b} = \{\mu_M\in\Sigma(\mu)_{M-\dom}\mid \kappa_M(\mu_M) = \kappa_M(b)\} = \{\mu \in \Sigma(\mu)_{M-\mmax}\mid \kappa_M(\mu_M) = \kappa (b)\},
  \]
  where the first equality follows from the fact that $b$ is superbasic in $M(L)$ and the second equality holds because $\mu$ is minuscule. Now the above lemma implies that $I_{M,\mu,b} = S_M(\mu,\kappa(b))$.
 \end{remark}

 \begin{proposition} \label{prop GHKR 544}
  \begin{subenv}
   \item Let $\mu$ be a dominant coweight. Then for every $\mu_M \in S(\mu)$ there exists an integer $d(\mu,\mu_M)$ such that
   \[
    C(\mu,\mu_M)^{\sigma_F^s} \sim q_F^{d(\mu,\mu_M) \cdot s}.
   \]
   \item We have an inequality
   \begin{equation} \label{term d(mu,mu_M)}
    d(\mu,\mu_M) \leq \langle \rho, \mu+\mu_M\rangle - 2\langle \rho_M,\mu_M \rangle.
   \end{equation}
   \item If $\mu_M \in \Sigma(\mu)_{M-\mmax}$, then the inequality (\ref{term d(mu,mu_M)}) is an equality.
  \end{subenv}
 \end{proposition}

 \begin{proof}
  The function field analogue of (1) and (2) are proven in \cite{GHKR06}, Prop.\ 5.4.2. Its proof determines the number of points 
  \[
   \# \left(\bigslant{N(\FF_q\rpot{t}) \mu_M(t) G(\FF_q\pot{t}) \cap G(\FF_q\pot{t}) \mu(t) G(\FF_q\pot{t})}{G(\FF_q\pot{t})}\right)
  \]
 for split groups, which still works in the $p$-adic case. Now (1) and (2) follow by applying the proof to the (split) group $G_{O_F}$ and (3) is the analogue of Corollary 5.4.4 in \cite{GHKR06}.
 \end{proof}

 \begin{proposition} \label{prop first step}
  \begin{subenv}
   \item The set $K\mu(p)Kb \cap N(L)$ is admissible.
   \item Let $b \in K_M\mu_M(p)K_M$. Then
   \[
    \nd K\mu(p) Kb \cap N(L) = d(\mu,\mu_M) - 2\langle \rho_N, \nu(b) \rangle 
   \]
  \end{subenv}
 \end{proposition}
 
 \begin{proof}
  The set $K\mu(p) Kb \cap N(L)$ is bounded by Lemma~\ref{lem bounded}. Choose $k$ such that it is contained in $N(-k)$. We denote $Y' := \lambda(p^k)(K\mu(p)Kb \cap N(L))\lambda(p^{-k})$. Let $j$ be big enough such that $N_j \subset \lambda(p^k)b^{-1}N_0b\lambda(p^{-k})$. Then $Y'$ is right-$N_j$-stable.

  For every $\FFbar_p$-algebra $R$ and $g \in L_pG(R)$ the subset
  \[
   \{s\in\Spec R\mid g_{\overline{k(s)}} \in L_p^+(\overline{k(s)}) \mu(p) L_p^+(\overline{k(s)})\}
  \]
  is locally closed in $\Spec R$ (cf.~\cite{CKV}~Lemma~2.1.6). Here $k(s)$ denotes the fraction field at $s$. Thus the set of all $s\in L_p^+N$ whose geometric points are an element of
  \[
   \lambda(p^k) (L_p^+(\overline{k(s)})\mu(p)L_p^+(\overline{k(s)})b \cap L_pN(\overline{k(s)}))\lambda(p^{-k})
  \]
  form a locally closed subset $\Yscr'$ of $L_p^+N$ with $\Yscr'(k) = Y'$. Furthermore, $\Yscr'$ is the preimage of some subset of $L_p^{+,j}N$ w.r.t.\ $t_j$. As we have seen in the proof of Lemma \ref{lem admissible is well-defined}, this implies that $Y'$ is admissible. Thus $K\mu(p) Kb \cap N(L)$ is admissible.

  Now the map $x \mapsto xb^{-1}$ induces an isomorphism of $\Gamma_F$-sets
  \[
   \bigslant{N(L)bK \cap K\mu_M(p)K}{K} \stackrel{\sim}{\longrightarrow} \bigslant{N(L)\cap K\mu(p)Kb^{-1} }{bN(0)b^{-1}}.
  \]
  Now choose $k_M \in K_M$ such that $bK_M= k_M\mu_M(p)K_M$. Then multiplication by $k_M$ defines an isomorphism of $\Gamma_F$-sets
  \[
   \bigslant{N(L)\mu_M(p)K \cap K\mu_M(p)K}{K} \stackrel{\sim}{\longrightarrow} \bigslant{N(L)bK \cap K\mu_M(p)K}{K}.
  \]
  Altogether, we get
  \begin{eqnarray*}
   \nd K\mu Kb \cap N(L) &=& d(\mu,\mu_M) + \nd bN(0)b^{-1} \\
   &=& d(\mu,\mu_M) - \langle 2\rho_N, \nu(b)\rangle.
  \end{eqnarray*}
  The last equality is an easy consequence of Corollary~\ref{cor calculation of nd}.
 \end{proof}

 \subsection{Relative dimension of certain morphisms $f:L^n \rightarrow L^n$}
 Before we can continue with the second step of our proof, we need to explain how the analogue of section three and four of \cite{GHKR06} works in the $p$-adic case.

 Let $V$ be a finite dimensional vector space over $L$ and $\Lambda_2 \subset \Lambda_1$ be two lattices in $V$. We define the structure of a variety on $\Lambda_1/\Lambda_2$ as follows. By the elementary divisor theorem, we find a basis $v_1, \ldots, v_n$ of $\Lambda_1$ such that $\Lambda_2$ has a basis of the form $p^{\alpha_1}v_1, \ldots , p^{\alpha_n}v_n$. This induces an isomorphism
 \[
  \bigslant{\Lambda_1}{\Lambda_2} \stackrel{\sim}{\longrightarrow} \prod W_{\alpha_i} (k), \quad \sum \beta_i  v_i \mod \Lambda_2 \mapsto (\beta_i \mod p^{\alpha_i})_i.
 \]
 As $W_{\alpha_i}$ is represented by the scheme $\AA^{\alpha_i}$, this defines the structure of an affine space on $\Lambda_1/\Lambda_2$. The variety structure does not depend on the choice of $v_1,\ldots,v_n$: Let $w_1,\ldots,w_n$ another basis as above. Define $\phi$ such that the diagram

 \begin{tikzcd}
  \bigslant{\Lambda_1}{\Lambda_2} \arrow{r}{\id} \arrow{d}[swap]{\sum \beta_i v_i \mod \Lambda_2 \mapsto \atop \beta_i \mod p^{\alpha_i}}
  & \bigslant{\Lambda_1}{\Lambda_2} \arrow{d}{\sum \beta_i w_i \mod \Lambda_2 \mapsto \atop \beta_i \mod p^{\alpha_i}} \\
  \prod W_{\alpha_i}(k) \arrow{r}{\phi}
  & \prod W_{\alpha_i}(k)
 \end{tikzcd}

 \noindent commutes. Now $\phi$ is $W(k)$-linear and hence can be expressed as family of polynomials in the coordinates of the truncated Witt vectors. So $\phi$ is a morphism of varieties. The same argument shows that $\phi^{-1}$ is also a morphism of varieties, thus $\phi$ is an isomorphism and the structure of an affine space on $\Lambda_1/\Lambda_2$ given by the bases $v_1,\ldots,v_n$ and $w_1,\ldots,w_n$ are the same.

 Now one can can define admissible resp.\ ind-admissible subsets of $V$ and their dimension literally as in \cite{GHKR06}. Also, the $p$-adic analogue of the statements and proofs of section 4 in \cite{GHKR06} hold. Thereof we will need the following notations and results.

 \begin{definition}
  Let $(V,\Phi)$ be an isocrystal. We define
  \[
   d(V,\Phi) = \sum_{\lambda < 0} \lambda \dim V_\lambda,
  \]
  where $V_\lambda$ is the isoclinic component of $(V,\Phi)$ of slope $\lambda$.
 \end{definition}

 \begin{proposition} \label{prop isocrystal}
  Let $V,V'$ be two finite dimensional $L$-vector spaces of the same dimension. Let $\phi: V \rightarrow V'$ be an $L$-linear isomorphism and $\psi:V \rightarrow V'$ be a $\sigma$-linear bijection. We define $f:V \rightarrow V'$ by $f := \psi - \phi$. Then for any lattice $\Lambda$ in $V$ there exists a lattice $\Lambda'$ in $V'$ and a non-negative integer $j$ such that
  \[
   p^j\Lambda' \subset f\Lambda \subset \Lambda'.
  \]
  For any such triple $j,\Lambda,\Lambda'$ and $l \geq j$, consider the induced morphism
  \[
   \overline{f}: \bigslant{\Lambda}{p^l \Lambda} \rightarrow \bigslant{\Lambda'}{p^l \Lambda'}.
  \]
  Then
  \begin{subenv}
   \item $\image \overline{f} \supset p^j \Lambda'/p^l \Lambda'$.
   \item $\dim\ker \overline{f} = d(V,\phi^{-1}\psi)+ [\Lambda : \phi^{-1}\Lambda']$.
   \item $(\ker \overline{f})^0 \subset p^{l-j}\Lambda/p^l\Lambda$
  \end{subenv}

 \end{proposition}
 \begin{proof}
  The proposition is the $p$-adic analogue of Prop.~4.2.2 in \cite{GHKR06}. Its proof is literally the same.
 \end{proof}

 \subsection{Dimension of the preimage under $f_{m_1,m_2}^{-1}$}

 \begin{proposition} \label{prop GHKR 532}
  The map $f_{m_1,m_2}$ is surjective. Moreover, for any admissible subset $Y \subset N(L)$ the inverse image is ind-admissible and
  \begin{equation} \label{term dim eq}
   \nd f^{-1}_{m_1,m_2} Y - \nd Y = d(\nfr(L), \Ad_\nfr (m_1)^{-1} \Ad_\nfr (m_2\sigma)) + \val\det \Ad_\nfr(m_1).
  \end{equation}
  We denote the right hand side of (\ref{term dim eq}) by $d(m_1,m_2)$ for convenience.
 \end{proposition}
 \begin{proof}
  This is the analogue of Prop.\ 5.3.2 of \cite{GHKR06}. Its proof is almost literally the same. As this is the part where the main part of the calculation of $\dim \Fscr$ is done, we give a brief outline of the proof and explain why the arguments carry over.

  By multiplying $m_1$ and $m_2$ by a suitable power of $\lambda(p)$ we may assume that the $N_j$ are stable under conjugation with $m_1$ and $m_2$, in particular $f$ maps $N_j$ into $N_j$. One easily checks by replacing $m_1$ and $m_2$ by their $\lambda(p^k)$-multiple increases both sides of (\ref{term dim eq}) by $d(k)$, thus the assertion of the proposition does not change.

  Now we consider the maps $f\langle i \rangle: N\langle i \rangle \rightarrow N\langle i \rangle$ induced by $f$. By choosing an isomorphism of the root subgroups with their Lie algebra, we identify $N\langle i \rangle \cong \Lie N\langle i \rangle$. Under this identification $f\langle i \rangle$ is identified with $\Ad_{\Lie N\langle i \rangle} (m_2)\sigma - \Ad_{\Lie N\langle i \rangle}(m_1)$. Thus we are in the situation considered in the previous subsection.

  The following two claims are obtained from Proposition \ref{prop isocrystal} (resp.\ Prop. 4.2.2 in the proof in \cite{GHKR06}) using purely group theoretic arguments. Therefore the proofs given in \cite{GHKR06} also work in the $p$-adic case.

  \begin{claim}
   There exists an integer $k$ such that for any $i \geq 1$
   \[
    \lambda(p^k)N[i]_0\lambda(p^{-k}) \subset f(N[i]_0).
   \]
  \end{claim}

  \begin{claim}
   Choose positive integers $j,k,l$ such that $f\langle i \rangle N\langle i \rangle_0 \supset p^j N\langle i \rangle$ for any $i$, $k$ as above and $N[i]_{l-j} \subset \lambda(p^k)N[i]_0\lambda(p^{-k})$.  We denote by $H = N_0/N_l$ and by $\overline{f}:H \rightarrow H$ the morphism induced by $f$. Then
   \[
    \dim \overline{f}^{-1} (1) = d(m_1,m_2).
   \]
  \end{claim}

  We will give the rest of the proof in greater detail, as this is the part where we have to work with the notion of admissibility, which is slightly different from the one in \cite{GHKR06}. However the concept is still the same as in their proof.

  We denote by $f_0$ the restriction of $f$ to $N_0$.
  \begin{claim}
   Assume that $Y$ is an admissible subset of $N(k)$ with $k$ as in claim 2 (ensuring that $Y$ is contained in the image of $f_0$). Then $f_0^{-1}Y$ is admissible and
   \[
    \nd f_0^{-1}Y - \nd Y = d(m_1,m_2).
   \]
  \end{claim}

  To prove claim 3, we choose $l \gg 0$ such that Claim 2 holds and such $Y$ is the preimage of a locally closed subset $\overline{Y}$ in $H = N_0/N_l$. Then all non-empty (reduced) fibres of $\overline{f}$ are isomorphic to each other. Indeed, if $n \in \image (f)$ and $n_0$ is a preimage of $n$ then
  \[
   \overline{f}^{-1}(n) = \overline{f}^{-1}(1) n_0.
  \]
  Now $\overline{Y}$ is contained in the image of $\overline{f}$ by Claim 1 and hence every fibre of $\overline{f}$ has dimension $d(m_1,m_2)$ by claim 2, so
  \begin{equation}\label{term 1}
   \dim f_0^{-1} \overline{Y} - \dim \overline{Y} = d(m_1,m_2).
  \end{equation}
  As $f_0^{-1}Y = t_l^{-1} (\overline{f}^{-1}\overline{Y})(k)$, we see that it is admissible that the equation (\ref{term 1}) implies
  \[
   \nd f_0^{-1}Y - \nd Y = d(m_1,m_2)
  \]
  This finishes the proof of the claim.

  Now let $Y \subset N(L)$ be admissible and $j$ big enough such that $Y\subset N(-j)$. Now $f_{m_1,m_2}$ and conjugation with (a power of) $\lambda(p)$ commute. Hence
  \[
   f_{m_1,m_2}^{-1}(Y) \cap N(-j-k) = (f_0^{-1}( Y (j+k))) (-j-k),
  \]
  where $k$ is chosen as above. Hence $f_{m_1,m_2}^{-1}(Y) \cap N(-j-k)$ is admissible by Claim 3 and has numerical dimension $\nd Y + d(m_1,m_2)$, proving the ind-admissibility of $f_{m_1,m_2}^{-1}(Y)$ and the dimension formula (\ref{term dim eq}).

  For the surjectivity of $f$, note that by Claim 1 there exists an integer $k$ such that $N(k)$ is contained in the image of $f$. As $f$ commutes with conjugation with $\lambda(p)$ this implies that $N(j)$ is contained in the image of $f_{m_1,m_2}$ for every $j$. As the $N(j)$ exhaust $N(L)$, the assertion follows.
 \end{proof}

 \begin{proof}[Proof of Proposition~\ref{prop fibre dimension}]
  Altogether, we get
   \begin{eqnarray*}
    \dim \Fscr &\stackrel{\rm Prop.~\ref{prop numerical dimension equals dimension}}{=}& \nd \tilde\Fscr \\
    &\stackrel{\rm Prop.~\ref{prop GHKR 532}}{=}& \nd (K\mu(p)Kb^{-1} \cap N(L)) + \langle \rho,\nu_M(b)-\nu_M(b)_\dom \rangle \\
    &\stackrel{\rm Prop.~\ref{prop first step}}{=}& d(\mu,\mu_M) - 2\langle \rho_N,\nu_M(b)\rangle + \langle \rho,\nu_M(b)-\nu_M(b)_\dom \rangle \\
    &\stackrel{\rm Prop.~\ref{prop GHKR 544}}{=}& \langle\rho,\mu+\mu_M\rangle - 2\langle \rho_M,\mu_M\rangle - 2\langle \rho_N,\nu_M(b)\rangle + \langle \rho,\nu_M(b)-\nu_M(b)_\dom \rangle\\
    &=& \langle\rho,\mu-\nu_M(b)_\dom\rangle - \langle\rho_M,\mu_M\rangle \\ & & + \underbrace{\langle\rho_N,\mu_M\rangle - \langle\rho_N,\nu_M(b)\rangle}_{=0} + \underbrace{\langle\rho,\nu_M(b)\rangle-\langle \rho_N,\nu_M(b)\rangle}_{=0} \\
    &=& \langle\rho, \mu-\nu_G(b) \rangle -\langle \rho_M, \mu_M\rangle.
   \end{eqnarray*}
 \end{proof}

 \section{Reduction to the superbasic EL-case} \label{sect reduction to superbasic}

 We first reduce to the case of a superbasic Rapoport-Zink datum which still might be of PEL type. 

 \begin{proposition}
  Assume Proposition~\ref{prop main} is true for every simple superbasic Rapoport-Zink datum. Then it is also true in general.
 \end{proposition}
 \begin{proof}
  We use the notation of Prop.~\ref{prop fibre dimension} and choose $b$ and $M$ such that $b \in M(L)$ is superbasic. If Proposition~\ref{prop main} is true in the superbasic case, we have
  \[
   \dim \Mscr_{M,b,\mu_M} \leq \langle \rho_M,\mu_M\rangle - \frac{1}{2}\defect_M(b).
  \]
  By Proposition~\ref{prop fibre dimension} we get for $\mu_M \in I_{\mu,b,M}$
  \begin{eqnarray*}
   \dim p_{M}^{-1} \Mscr_{M,b,\mu_M} &\leq&  \langle \rho, \mu - \nu_G(b) \rangle - \frac{1}{2}\defect_M (b).
  \end{eqnarray*}
  Since $M$ is a Levi subgroup of $M_b$, the group $J_{M,b}$ is a Levi subgroup of $J_{G,b}$. As the $\QQ_p$-rank of a linear algebraic group is the same as the $\QQ_p$-rank of its Levi subgroups, this implies $\defect_G(b) = \defect_M(b)$. Altogether,
  \[
   \dim \Mscr_G(b,\mu) = \dim \Mscr_P(b,\mu) = \max \dim p_M^{-1} \Mscr_{M,b,\mu_M} \leq \langle \rho, \mu - \nu_G(b) \rangle.
  \]
 \end{proof}

 In the superbasic case we can (and will) prove Theorem~\ref{thm dimension RZ-space} directly, i.e.\ without the detour via Proposition~\ref{prop main}.

 \begin{lemma}
  Assume that Theorem \ref{thm dimension RZ-space} holds for every superbasic simple Rapoport-Zink datum of EL type. Then it is also true for every superbasic Rapoport-Zink datum of PEL type.
 \end{lemma}
 \begin{proof}
  Let $\hat\Dscr$ be a superbasic unramified RZ-datum of PEL type. Then by \cite{CKV}~Lemma~3.1.1 the adjoint group $G^\ad$ is isomorphic to a product of $\Res_{F_i/\QQ_p} \PGL_{h_i}$. As mentioned in section~\ref{ss simple RZ}, we can (and will) assume that $\hat\Dscr$ is simple. Thus either $G \cong \GSp_{F,n}$ or $G \cong \GU_{F,n}$. Comparing the Dynkin diagrams (with Galois action) of the groups $G^\ad$ and $\Res_{F_i/\QQ_p} \PGL_{h_i}$, we see that there are only two cases where these groups are isomorphic:

   1. $G \cong \GSp_{F,2}$.

   2. $G \cong \GU_{F,2}$.

  Let $\hat\Dscr'$ be the Rapoport-Zink datum of EL type one gets by forgetting the polarization and $G'$ the associated linear algebraic group. We get a canonical closed embedding
  \[
   \Mscr_{G}(b,\mu) \hookrightarrow \Mscr_{G'}(b,\mu).
  \]
  In the first case this is an isomorphism since $\Res_{F/\QQ_p} \GSp_2 = \Res_{F/\QQ_p} \GL_2$, so Theorem \ref{thm dimension RZ-space} is true for $\Dscr$.

  Now assume that $G \cong \GU_{F,2}$. We first show that $\hat\Dscr'$ is also superbasic. Using the explicit description of $X_*(T)$ given after Proposition~\ref{prop PEL vs D}, we see that the Newton point of the $\sigma$-conjugacy class associated to $\hat\Dscr'$ is of the form $(\alpha, 1-\alpha)$. This is central in $\GU_{F,2}$ if and only $\alpha = \frac{1}{2}$, thus $\hat\Dscr'$ is also superbasic. Now each connected component of $\Mscr_G(b,\mu)$ is isomorphic to a closed subset of a connected component of $\Mscr_{G'}(b,\mu)$ and thus projective.

  By \cite{CKV}~Thm.~1.3, all connected components of a Rapoport-Zink space are isomorphic, so we can determine their dimension by counting points of some connected component. For any reductive group $H$ over $\ZZ_p$, $b' \in H(E)$, $\mu' \in X_*(H_{O_L})$ let
  \[
   X_H(b',\mu') := \{h\cdot H(O_L) \in \bigslant{H(L)}{H(O_L)}\mid hb'\sigma(h)^{-1} \in H(O_L)\mu'(p)H(O_L), \}
  \]
  denote the affine Deligne-Lusztig set. Let $\eta_H: H(L) \to \pi_1(H)$ denote the unique $H(O_L)$-bi-invariant map which maps $\mu_H(p)$ to the image of $\mu_H$ in $\pi_1(H)$ for every dominant cocharacter $\mu_H$. We denote for $\omega' \in \pi_1(H)$
  \[
   X_H(b',\mu')_{\omega'} := X_H(b',\mu') \cap \eta_H^{-1}(\{\omega'\}).
  \]
  By \cite{CKV}~Thm.~1.1 the connected components of $\Mscr_G(b,\mu)(k)$ are precisely the subspaces $X_G(b',\mu)_\omega$ which are non-empty.

  As $\GL_{F',2}$ and $\GU_{F,2}$ are not isomorphic, we compare their affine Deligne-Lusztig sets via their (isomorphic) adjoint groups. For $\omega\in \pi_1(G)$ with $X_G(b,\mu)_\omega \not= \emptyset$ we have 
  \[
   X_G(b,\mu)_\omega \cong X_{G^\ad}(b^\ad,\mu^\ad)_{\omega^\ad}
  \]
  as $\Gamma_E$-sets, where $b^\ad, \mu^\ad, \omega^\ad$ denote the images of $b,\mu,\omega$ in $G^\ad(L), X_*(G^\ad), \pi_1(G^\ad)$ respectively. Using the same argument for $G'' = \GL_{F',2}$ and suitable $b'',\mu'',\omega''$ we get
  \[
   X_G(b,\mu)_\omega \cong X_{G''}(b'',\mu'')_{\omega''}
  \]
  as $\Gamma_E$-sets and thus $\dim \Mscr_G(b,\mu) = \dim \Mscr_{G''}(b'',\mu'')$ by Proposition~\ref{prop lang weil}. As the value of the right hand side of (\ref{term dimension RZ-space}) only depends on the images in $G^\ad(L)$ resp.\ $X_*(T^\ad)^\Gamma_\QQ$ this proves Theorem \ref{thm dimension RZ-space} for $\hat\Dscr$.
 \end{proof}

 So Theorem \ref{thm dimension RZ-space} is reduced to the claim that it holds in the case of a simple superbasic Rapoport-Zink datum of EL type. This is proved in the next section, see Corollary \ref{cor superbasic RZ-space is qc} and Proposition \ref{prop main for superbasic RZ-data of EL type}.

 \section{The superbasic EL-case} \label{sect superbasic}

 \subsection{Notation and conventions}
 From now on we restrict to the EL-case with $[b]$ superbasic. Fixing a basis $e_i$ of $V$ (as $F \otimes L$-module), we get an identification $G = \Res_{O_F/\ZZ_p} \GL_h$. Let $d$ be the degree of the unramified field extension $F/\QQ_p$, then $I := \Gal(F/\QQ_p) \cong \Gal(k_F/\FF_p) \cong \ZZ / d\cdot \ZZ$. We choose the isomorphism such that the Frobenius $\sigma$ is mapped to $1$. Let $T \subset B \subset G$ where $T$ is the diagonal (maximal) torus and $B$ is the Borel subgroup of lower triangular matrices in $G$.

 We fix a superbasic element $b\in G(L)$ with Newton point $\nu \in X_*(T)_\QQ^\Gamma$ and a dominant cocharacter $\mu \in X_*(T)$ such that $b\in B(G,\mu)$. We have to show that
 \[
  \dim \Mscr_G(b,\mu) = \langle \rho, \mu - \nu \rangle - \frac{1}{2} \defect_G(b) \label{term superbasic}.
 \]
 By Remark~\ref{rem dimension formula} this is equivalent to
 \begin{equation} \label{term main EL}
  \dim \Mscr_G(b,\mu) = \sum_{i=1}^{h-1} \lfloor \langle \omega_i, \mu-\nu \rangle \rfloor,
 \end{equation}
 where $\omega_i$ are lifts of the fundamental weights of $G^{\rm der}$ (see also \cite{hamacher}~Prop.~4.4).

 As $T$ splits over $O_F$, the action of the absolute Galois group on $X_*(T)$ factorizes over $I$. We identify $X_*(T) = \prod_{\tau\in I} \ZZ^h$ with $I$ acting by cyclically permuting the factors. This yields an identification of $X_*(T)^I$ with $\ZZ^h$ such that
 \[
  X_*(T)^I \hookrightarrow X_*(T), \nu' \mapsto (\nu')_{\tau\in I}.
 \]

 We denote by $(N,F) \cong ((L \otimes_{\QQ_p} F)^h, b\sigma)$ the $F$-isocrystal associated to our Rapoport-Zink datum. We decompose $N = \prod_{\tau\in I} N_\tau$ according to the $F$-action as in Example~\ref{ex isocrystal}. Denote by $e_{\tau,i}$ the image of $e_i$ in $N_\tau$, then $\{e_{\tau,i}\}_{i=1}^h$ is a basis of the $N_\tau$ and  $\varsigma(e_{\tau,i})= e_{\varsigma\tau,i}$ for all $\varsigma \in I$. For $\tau \in I, l \in \ZZ, i=1,\ldots ,h$ denote $e_{\tau,i+l\cdot h} := p^l\cdot e_{\tau,i}$. Then each $v\in N_\tau$ can be written uniquely as infinite sum 
 \[
  v = \sum_{n \gg -\infty} [a_n]\cdot e_{\tau,n}
 \]
 with $a_n \in k$.

 Using Dieudonn\'e theory, we get an identification
 \[
  \Mscr_G(b,\mu)(k) = \{(M_\tau \subset N_\tau \textnormal{ lattice})_{\tau\in I}\mid \inv (M_\tau, b\sigma (M_{\tau-1})) = \mu_\tau\}.
 \]
 Here $\inv$ means the invariant and is defined as follows. Suppose we are given two lattices $M,M' \subset L^h$. By the elementary divisor theorem we find a basis $v_1, \ldots, v_n$ of $M$ and a unique tuple of integers $a_1 \leq \ldots \leq a_n$ such that $p^{a_1}v_1,\ldots , p^{a_n}v_n$ form a basis of $M'$. We define the cocharacter $\inv (M,M'):\GG_m \rightarrow \GL_h, x \mapsto \diag (x^{a_1}, \ldots , x^{a_n})$. If we write $M' = gM$ with $g \in GL_h (L)$ we may equivalently define $\inv (M,M')$ to be the unique cocharacter of the diagonal torus which is dominant w.r.t.\ the Borel subgroup of lower triangular matrices and satisfies $g \in \GL_h(O_L) \inv(M,M')(p) \GL_h(O_L)$.

 \begin{definition}
  \begin{enumerate}
   \item We call a tuple of lattices $(M_\tau \subset N_\tau)_{\tau\in I}$ a $G$-lattice.
   \item We define the volume of a $G$-lattice $M = gM^0$ to be the tuple
   \[
    \vol (M) = (\val \det g_\tau)_{\tau \in I}.
   \]
    Similarly, we define the volume of $M_\tau$ to be $\val\det g_\tau$. We call $M$ special if $\vol (M) = (0)_{\tau \in I}$.
  \end{enumerate}
 \end{definition}

 As $[b]$ is superbasic, $\nu$ is of the form $(\frac{m}{d\cdot h}, \frac{m}{d \cdot h}, \ldots , \frac{m}{d \cdot h})$  with $(m,h) = 1$. Now the condition $[b] \in B(G,\mu)$ translates to $\sum_{\tau \in I, i=1,\ldots h} \mu_{\tau,i} = m$. Replacing $b$ by a $\sigma$-conjugate if necessary, we can assume that $b$ is the form $b(e_{\tau,i}) = e_{\tau,i+m_\tau}$ where $m_\tau = \sum_{i=1}^h \mu_{\tau,i}$ (see also \cite{CKV}~Lemma~3.2.1). We could have chosen any tuple of integers $(m_\tau)$ such that $\sum_{\tau\in I} m_\tau = m$ but this particular choice has the advantage that the components of any $G$-lattice in $X_\mu(b)$ have the same volume. In general,
 \begin{eqnarray*}
  \vol M_\tau - \vol M_{\tau-1} &=& (\vol M_\tau - \vol b\sigma(M_{\tau-1})) + (\vol b\sigma(M_{\tau-1}) - \vol M_{\tau-1}) \\
  &=& (\sum_{i=1}^h \mu_{\tau,i}) - m_\tau.
 \end{eqnarray*}

 In the cases $\nu = (0)$ and $\nu = (1)$ the moduli space $\Mscr_G(b,\mu)$ is isomorphic to $\bigslant{\End_{\QQ}(\XXund)}{\End(\XXund)} \cong \ZZ$, considered as discrete union of points (\cite{CKV}~Thm.~1.1). In this case we have $\mu = \nu$ thus the right hand side of (\ref{term main EL}) is also zero and Theorem \ref{thm dimension RZ-space} holds. We assume $\nu \not= (0), (1)$ from now on.

 Then the connected components of $\Mscr_G(b,\mu)$ are
 \[
  \Mscr_G(b,\mu)^i = \{ M \in \Mscr_G(b,\mu)(k); \vol M = (i)_{\tau\in I}\}.
 \]
 where $i$ ranges over the integers (use \cite{CKV}~Thm.~1.1). In particular, we have
 \[
  \dim \Mscr_G(b,\mu) = \dim \Mscr_G(b,\mu)^0.
 \]

 \subsection{A decomposition of $\Mscr_G(b,\mu)$}
 In order to calculate the dimension of $\Mscr_G(b,\mu)$, we decompose $\Mscr_G(b,\mu)^0$ into locally closed sets, whose dimension is given by a purely combinatorial formula. Let
 \begin{eqnarray*}
  \Iscr_\tau: N_\tau \setminus \{ 0\} \quad &\rightarrow& \ZZ \\
  \sum_{n \gg -\infty} [a_n]\cdot e_{\tau,n} &\mapsto& \min \{n\in\ZZ; a_n \not= 0\}.
 \end{eqnarray*} 
 Note that $\Iscr_\tau$ satisfies the strong triangle inequality for every $\tau$. We denote $N_{hom} := \coprod_{\tau \in I} (N_\tau\setminus \{0\})$, analogously $M_{hom}$. For $M \in \Mscr_G(b,\mu)^0 (k)$, we define
 \[
  A(M) := \Iscr (M_{hom})
 \]
 where $\Iscr = \sqcup \Iscr_\tau: N_\tau \to \coprod_{\tau\in I} \ZZ$. For a subset $A$ of $\coprod_{\tau\in I} \ZZ$ we denote $\Sscr_A$ the subset of all $G$-lattices in $\Mscr_G(b,\mu)^0$ whose image under $\Iscr$ equals $A$.

 We first study the possible values of $A(M)$ to determine which $\Sscr_A$ are non-empty.

 \begin{definition}
  Let $\ZZ^{(d)} := \coprod_{\tau \in I} \ZZ_{(\tau)}$ be the disjoint union of $d$ isomorphic copies of $\ZZ$. For $a\in\ZZ$ we denote by $a_{(\tau)}$ the corresponding element of $\ZZ_{(\tau)}$ and write $|a_{(\tau)}| := a$. We equip $\ZZ^{(d)}$ with a partial order ``$\leq$'' defined by
  \[
   a_{(\tau)} \leq c_{(\varsigma)} :\Lra a\leq c \textnormal{ and } \tau = \varsigma
  \]
  and a $\ZZ$-action given by
  \[
   a_{(\tau)} + n = (a+n)_{(\tau)}.
  \]
  Furthermore we define the function 
  \begin{eqnarray*}
   f:\ZZ^{(d)} \rightarrow \ZZ^{(d)} &,& a_{(\tau)} \mapsto (a+m_{\tau+1})_{(\tau+1)}
  \end{eqnarray*}
 \end{definition}

  We impose the notation that for any subset $A \subset \ZZ^{(d)}$ we write $A_{(\tau)} := A\cap  \ZZ_{(\tau)}$.
 
 \begin{definition}
  \begin{subenv}
   \item An EL-chart is a non-empty subset $A \subset \ZZ^{(d)}$ which is bounded from below, stable under $f$ and addition of $h$. We call an EL-chart $A$ small if it satisfies $A+h \subset f(A)$.
   \item Let $A$ be an EL-chart and $B = A\setminus (A+h)$. We say that $A$ is normalized if $\sum_{b \in B_{(0)}} b = \frac{h\cdot (h-1)}{2}$.
  \end{subenv}
 \end{definition}
 
 Let $A$ be a small EL-chart and $B = A\setminus (A+h)$. It is easy to see that $\#B_{(\tau)} = h$ for all $\tau \in I$ and $B = B^- \sqcup B^+$ where
 \begin{eqnarray*}
  B^+ &=& \{ b\in B\mid f(b) \in B \} \\
  B^- &=& \{ b\in B\mid f(b)-h \in B \}.
 \end{eqnarray*}
 We define a sequence $b_0, \ldots b_{d\cdot h -1}$ of distinct elements of $B$ as follows. Denote by $b_0$ the minimal element of $B_{(0)}$ and let
 \[
  b_{i+1} = \left\{ \begin{array}{ll}
                     f(b) & \textnormal{ if } b \in B^+ \\
                     f(b)-h & \textnormal{ if } b \in B^-.
                    \end{array} \right.
 \]
 These elements are indeed distinct: If $b_i = b_j$ then obviously $i \equiv j \mod d$ and then $b_{i + k\cdot d} \equiv b_i + k\cdot m \mod h$ implies that $i=j$ as $m$ and $h$ are coprime.

 Define the cocharacter $\mu' \in X_*(T)$ by
 \[
  \mu'_\tau = (\underbrace{0,\ldots,0}_{\# B_{(\tau-1)}^+},\underbrace{1,\ldots,1}_{\# B_{(\tau-1)}^-}).
 \]
 We call $\mu'$ the Hodge-point of $A$.

 \begin{remark}
  One easily checks that a small EL-chart is the same as an EL-Chart in the sense of Def.~5.2 of \cite{hamacher} whose type only has coordinates $0$ and $1$ and that the definition of the Hodge-point in both cases coincides. In particular the Hodge-point $\mu'$ is minuscule so that by Cor.~5.10 of \cite{hamacher} the notion of an EL-chart for $\mu'$ and an extended EL-chart for $\mu'$ are also equivalent, allowing us to use the combinatorics of \cite{hamacher} in our case. 
 \end{remark}

 \begin{proposition} \label{prop A(M) is EL-chart}
  Let $M \in \Mscr_G(b,\mu)^0(k)$. Then $A=A(M)$ is a normalized small EL-chart with Hodge-point $\mu$.
 \end{proposition}
 \begin{proof}
  $A(M)$ is stable under $f$ and addition with $h$ since
  \begin{eqnarray*}
   \Iscr(F v) &=& f(\Iscr(v)) \\
   \Iscr(p\cdot v) &=& \Iscr(v)+h.
  \end{eqnarray*}
  Furthermore, we have
  \[
   A(M)+h = \Iscr(pM) \subset \Iscr(FM) = f(A(M)).
  \]
  The fact that $A(M)$ is bounded from below is obvious. Let $M = gM^0$. We have
  \[
   0 = \val\det g_0 = |\NN^d \setminus A(M)_{(0)}| - |A(M)_{(0)} \setminus \NN^d|,
  \]
  hence
  \[
   \sum_{b \in B(M)_{(0)}} b = \sum_{i=0}^{h-1} i = \frac{h(h-1)}{2} 
  \]
  and thus $A(M)$ is indeed a normalized small EL-chart. We have for every $\tau \in I$
  \[
   \#\{i\mid \mu_{\tau,i} = 1\} = \dim_{k_0} \bigslant{M_\tau}{b\sigma(M)_\tau} = \# B^-_{(\tau-1)},
  \]
  thus the Hodge point of $A(M)$ is $\mu$.
 \end{proof}

 \begin{corollary} \label{cor decomposition}
  The $\Sscr_A$ define a decomposition of $\Mscr_G(b,\mu)^0$ into finitely many locally closed subsets. In particular, $\dim \Mscr_G(b,\mu)^0 = \max_{A} \dim \Sscr_A$.
 \end{corollary}
 \begin{proof}
  By Proposition \ref{prop A(M) is EL-chart}, $\Mscr_G(b,\mu)$ is the (disjoint) union of the $\Sscr_{A}$ with $A$ being a small EL-chart with Hodge-point $\mu$. By Corollary 5.11 of \cite{hamacher} this union is finite. It remains to show that $\Sscr_{A}$ is locally closed. One shows that the condition $A(M)_{(\tau)} = A_{(\tau)}$ is locally closed analogously to the proof of Prop.~5.1 in \cite{viehmann08}. Then $\Sscr_{A}$ is locally closed as it is the intersection of finitely many locally closed subsets. 
 \end{proof}
 
 \begin{definition}
  Let $A$ be a small EL-chart with Hodge-point $\mu$. We define
 \[
  \Vscr_A = \{ (j,i) \in (\ZZ/dh\ZZ)^2 \mid b_j \in B^-, b_i \in B^+, b_j < b_i\}
 \]
 \end{definition}

 \begin{remark}
  Our notion $\Vscr_A$ coincides with $\Vscr(A,\varphi)$ in \cite{hamacher} with the slight difference that the latter considers pairs $(b_j,b_i)$ instead of $(j,i)$.
 \end{remark}

 \begin{proposition}\label{prop decomposition}
  Let $A$ be an EL-chart with Hodge-point $\mu$. Then $\Sscr_A \cong \AA^{\Vscr_A}$.
 \end{proposition}
 \begin{proof}
  This proposition is proven for the case $G=\GL_h$ in \cite{viehmann08}, \S 5. The construction of a morphism $f: \AA^\Vscr_A \to \Sscr_A$ is very similar to that in \cite{viehmann08} and the proof that it is well-defined and an isomorphism is the same. Therefore we only explain the construction of $f$. 

  We denote $R  = k[t_{j,i}\mid (j,i) \in \Vscr_A]$. The morphism $f:\AA^{\Vscr_A} \to \Sscr_A$ corresponds to a quasi-isogeny $X \mapsto \XX_{\AA^{\Vscr_A}}$, which we will describe by the construction a subdisplay of the isodisplay $N_{W(R)_\QQ}$ of $\XX_{\AA^{\Vscr_A}}$. There exists a unique family $\{v_i; 0 \leq i < dh\} \subset N_{W(R)_\QQ}$ which satisfies the following relations:
  \begin{eqnarray*}
   v_0 &=& e_{b_0} \\
   v_{i+1} &=& \left\{ \begin{array}{ll}
                        Fv_i & \textnormal{ if } b_i,b_{i+1} \in B^+ \\
                        Fv_i + \sum_{(j,i) \in \Vscr_A} [t_{j,i}]v_i & \textnormal{ if } b_i \in B^+, b_{i+1} \in B^- \\
                        \frac{F(v_i)}{p} & \textnormal{ if }b_i \in B^-, b_{i+1} \in B^+ \\
                        \frac{F(v_i)}{p} + \sum_{(j,i) \in \Vscr_A} [t_{j,i}] v_i &  \textnormal{ if } b_i,b_{i+1} \in B^-
                       \end{array} \right.
  \end{eqnarray*}
  The proof that $(v_i)$ exists and is unique is literally the same as in \cite{viehmann08}. Let
  \begin{eqnarray*}
   L &=& \spa_{W(R)}(v_i\mid b_i \in B^-) \\
   T &=& \spa_{W(R)}(v_i\mid b_i \in B^+) \\
   P &=& L \oplus T \\
   Q &=& L \oplus I_R T.
  \end{eqnarray*}
  Then $(P,Q,F,\frac{F}{p})$ is a subdisplay of $N_{W(R)_\QQ}$, which yields a quasi-isogeny $X \mapsto \XX_{\AA^{\Vscr_A}}$ corresponding to a point $f\in \Sscr_A(\AA^{\Vscr_A})$.
 \end{proof}

 \begin{corollary} \label{cor superbasic RZ-space is qc}
  $\Mscr_G(b,\mu)^0$ is projective
 \end{corollary}
 \begin{proof}
  By the above proposition and Corollary \ref{cor decomposition} the underlying topological space of $\Mscr_G(b,\mu)^0$ has a decomposition into finitely many quasi-compact subspaces. Thus it is quasi-compact. Now \cite{RZ96}~Cor.~2.31 and Prop.~2.32 imply that $\Mscr_G(b,\mu)^0$ is quasi-projective with projective irreducible components. Thus it is projective.
 \end{proof}

 \begin{proposition} \label{prop main for superbasic RZ-data of EL type}
  The dimension formula (\ref{term dimension RZ-space}) holds for superbasic Rapoport-Zink data of EL type.
 \end{proposition}
 \begin{proof}
  As we remarked at the beginning of this section, we have to show that
  \[
   \dim \Mscr_G(b,\mu) = \sum_{i=1}^{h-1} \lfloor \langle \omega_i, \nu-\mu \rangle \rfloor.
  \]
  Now by Proposition \ref{prop decomposition} we get
  \[
   \dim \Mscr_G(b,\mu) = \max \{\#\Vscr(A)\mid A \textnormal{ is a small EL-chart with Hodge point } \mu\}.
  \]
  By Prop.~7.1 and Thm.~7.2 of \cite{hamacher} the right hand side equals $\sum_{i=1}^{h-1} \lfloor \langle \omega_i, \nu-\mu \rangle \rfloor$, finishing the proof.
 \end{proof}

 \begin{remark}
  As a consequence of the decomposition above we get an analogous description of the (top-dimensional) irreducible components of $\Mscr_G(b,\mu)^0$ resp.\ the $J_b(\QQ_p)$-orbits of irreducible components of $\Mscr_G(b,\mu)$ as in the case of affine Deligne-Lusztig varieties with $\mu$ minuscule.
 \end{remark}

 \appendix

 \section{Root data of some reductive group schemes}
 Here we use the notation of section \ref{ss group theory}. Furthermore, we denote all relative root data of $G_{\QQ_p}$ by a subscript $\QQ_p$. Let $I$ denote the Galois group of $O_F$ over $\ZZ_p$.

 \subsection{$\GL_{O_F,n}$} \quad\newline
 We have $\GL_{O_F,n} \otimes O_F \cong \prod_{\tau\in I} \GL_n$ with the Galois action cyclically permuting the $\GL_n$-factors. We choose $T_1 \subset B_1 \subset \GL_{O_F,n}$ to be the diagonal torus torus resp. the Borel subgroup of upper triangular matrices. Furthermore, let
 \begin{eqnarray*}
  e_{\varsigma,i}: T_1 \to \GG_m &,& (\diag(t_{\tau,1},\ldots,t_{\tau,n}))_{\tau \in I} \mapsto t_{\varsigma,i} \\
  e_{\varsigma,i}^\vee: \GG_m \to T_1 &,& x \mapsto (\diag(1,\ldots,1),\ldots, \diag(1,\ldots,1,x,1,\ldots,1),\ldots,\diag(1,\ldots,1))
 \end{eqnarray*}
 where the entry $x$ is the $(\varsigma,i)$\textsuperscript{th} entry. The $e_{\varsigma,i}$ resp. $e_{\varsigma,i}^\vee$ form a basis of $X^*(T)$ resp. $X_*(T)$, thus
 \begin{eqnarray*}
  X^*(T_1) &\cong& \prod_{\tau\in I} \ZZ^n \\
  X_*(T_1) &\cong& \prod_{\tau\in I} \ZZ^n \\
  R &=& \{e_{\tau,i} -e_{\tau,j} \mid \tau \in I, i \not=j \in \{1,\ldots,n\}\} \\
  R^\vee &=& \{e_{\tau,i}^\vee -e_{\tau,j}^\vee \mid \tau \in I, i \not=j \in \{1,\ldots,n\}\} \\
  R^+ &=& \{e_{\tau,i} -e_{\tau,j} \mid \tau \in I, i <j \in \{1,\ldots,n\}\} \\
  R^{\vee,+} &=& \{e_{\tau,i} -e_{\tau,j} \mid \tau \in I, i <j \in \{1,\ldots,n\}\} \\
  \Delta^+ &=& \{e_{\tau,i} -e_{\tau,i+1} \mid \tau \in I, i \in \{1,\ldots,n-1\}\} \\
  \Delta^{\vee,+} &=& \{e_{\tau,i}^\vee -e_{\tau,i+1}^\vee \mid \tau \in I, i \in \{1,\ldots,n-1\}\}.
 \end{eqnarray*}
 Hence
 \[
  X_*(T_1)_{\dom} = \{ \mu \in \prod_{\tau\in I} \ZZ^n \mid \mu_{\tau,1} \geq \ldots \mu_{\tau, n} \textnormal{ for all } \tau\}.
 \]
 Now the maximal split torus $S \subset T_{\QQ_p}$ is given by
 \[
  S_1 = \{ (\diag(t_1,\ldots,t_n))_{\tau \in I} \in T_{\QQ_p}\}.
 \]
 We define
 \begin{eqnarray*}
  e_{\QQ_p,i} &:=& e_{\tau,i|S} \textnormal{ for some} \tau\in I\\
  e_{\QQ_p,i}^\vee &:=& \sum_{\tau\in I} e_{\tau,i}^\vee.
 \end{eqnarray*}
 The $e_{\QQ_p,i}$ resp. $e_{\QQ_p,i}^\vee$ form a basis of $X^*(S_1)$ resp. $X_*(S_1)$, thus
 \begin{eqnarray*}
  X^*(S_1) &\cong&  \ZZ^n \\
  X_*(S_1) &\cong&  \ZZ^n \\
  R_{\QQ_p} &=& \{e_{\QQ_p,i} -e_{\QQ_p,j} \mid i \not=j \in \{1,\ldots,n\}\} \\
  R_{\QQ_p}^\vee &=& \{e_{\QQ_p,i}^\vee -e_{\QQ_p,j}^\vee \mid  i \not=j \in \{1,\ldots,n\}\} \\
  R_{\QQ_p}^+ &=& \{e_{\QQ_p,i} -e_{\QQ_p,j} \mid i <j \in \{1,\ldots,n\}\} \\
  R_{\QQ_p}^{\vee,+} &=& \{e_{\QQ_p,i} -e_{\QQ_p,j} \mid i <j \in \{1,\ldots,n\}\} \\
  \Delta_{\QQ_p}^+ &=& \{e_{\QQ_p,i} -e_{\QQ_p,i+1} \mid i \in \{1,\ldots,n-1\}\} \\
  \Delta_{\QQ_p}^{\vee,+} &=& \{e_{\QQ_p,i}^\vee -e_{\QQ_p,i+1}^\vee \mid \tau \in I, i \in \{1,\ldots,n-1\}\}.
 \end{eqnarray*}
 As a consequence
 \[
  X_*(S_1)_{\QQ,\dom} = \{\nu \in \ZZ^n\mid \nu_1 \geq \ldots \geq \nu_n\}
 \]
 \subsection{$\GSp_{O_F,n}$} \quad\newline
 We have 
 \[
  \GSp_{O_F,n} \otimes O_F \cong (\prod_{\tau\in I} \GSp_n)^1 = \{(g_\tau) \in \prod_{\tau\in I} \GSp_n \mid c(g_\tau) = c \textnormal{ does not depend on } \tau\}
 \]
 with the Galois action cyclically permuting the factors. Let $T_2 \subset B_2 \subset \GSp_{O_F,n}$ denote the maximal torus resp.\ the Borel subgroup of upper triangular matrices. We denote by
 \[
  c:T_2 \to \GG_m
 \]
 The similitude factor. Then
 \begin{eqnarray*}
  X^*(T_2) &\cong& \bigslant{X_*(T_1)}{\langle (e_{\tau,i} + e_{\tau,n+1-i}) - ( e_{\varsigma,j} + e_{\varsigma,n+1-j}) \rangle_{\tau,\varsigma ,i,j}} \\
  X_*(T_2) &=& \{ (\mu \in X_*(T) \mid \mu_{\tau,i} + \mu_{\tau,n+1-i} = c(\mu) \textnormal{ for some integer } c(\mu) \} \\
  R &=& \{e_{\tau,i|T_2} -e_{\tau,j|T_2} \mid \tau\in I, i \not=j \in \{1,\ldots,n/2\}\}\\
    & & \cup \{\pm(e_{\tau,i|T_2} +e_{\tau,j|T_2} -c) \mid \tau\in I i\not=j \in \{1,\ldots,n/2\}\} \\
    & & \cup \{\pm(2e_{\tau,i|T_2} -c) \mid \tau\in I, i\in\{1,\ldots,n/2\}\} \\
    &=& \{e_{\tau,i|T_2} - e_{\tau,j|T_2} \mid \tau\in I, i \not=j \in \{1,\ldots,n\}\} \\
  R^\vee &=& \{e_{\tau,i}^\vee -e_{\tau,j}^\vee + e_{\tau,n+1-j}^\vee - e_{\tau,n+1-i}^\vee \mid \tau\in I, i \not=j \in \{1,\ldots,n/2\}\} \\
         & & \cup \{ \pm(e_{\tau,i}^\vee+e_{\tau,j}^\vee - e_{\tau,n+1-i}^\vee-e_{\tau,n+1-j}^\vee) \mid \tau\in I i \not=j \in \{1,\ldots,n/2\}\} \\
         & & \cup \{ \pm e_{\tau,i}^\vee\mid \tau\in I, i \in \{1,\ldots,n/2\} \\
  R^+ &=& \{e_{\tau,i|T_2} -e_{\tau,j|T_2} \mid \tau\in I, i < j \in \{1,\ldots,n/2\}\}\\
      & & \cup \{e_{\tau,i|T_2} +e_{\tau,j|T_2} -c \mid  i\not=j \in \{1,\ldots,n/2\}\} \\
      & & \cup \{2e_{\tau,i|T_2} -c \mid \tau\in I, i\in\{1,\ldots,n/2\}\} \\
      &=& \{e_{\tau,i|T_2} - e_{\tau,j|T_2} \mid \tau\in I, i \not=j \in \{1,\ldots,n\}\} \\
  R^{\vee,+} &=& \{e_{\tau,i}^\vee -e_{\tau,j}^\vee + e_{\tau,n+1-j}^\vee - e_{\tau,n+1-i}^\vee \mid \tau\in I, i <j \in \{1,\ldots,n/2\}\} \\
         & & \cup \{ \pm(e_{\tau,i}^\vee+e_{\tau,j}^\vee - e_{\tau,n+1-i}^\vee-e_{\tau,n+1-j}^\vee\mid \tau\in I, i \not=j \in \{1,\ldots,n/2\}\} \\
         & & \cup \{ \pm(e_{\tau,i}^\vee\mid \tau\in I, i \in \{1,\ldots,n/2\} \\
  \Delta^+ &=& \{e_{\tau,i|T_2} -e_{\tau,i+1|T_2} \mid \tau \in I, i \in \{1,\ldots,n/2-1\}\} \cup \{2e_{\tau,n/2|T_2} -c \mid \tau \in I\} \\
           &=& \{e_{\tau,i|T_2} -e_{\tau,i+1|T_2} \mid \tau \in I, i \in \{1,\ldots,n-1\}\} \\
  \Delta^{\vee,+} &=& \{e_{\tau,i}^\vee -e_{\tau,i+1}^\vee + e_{\tau,n+1-i}^\vee - e_{\tau,n-i}^\vee \mid \tau \in I, i \in \{1,\ldots,n/2-1\}\} \cup \{e_{\tau,n/2}^\vee\mid \tau \in I\} 
 \end{eqnarray*}
 Hence
 \[
  X_*(T_2)_{\dom} = \{ \mu \in \prod_{\tau\in I} \ZZ^n \mid \mu_{\tau,1} \geq \ldots \mu_{\tau, n} \textnormal{ for all } \tau, \mu_{\tau,i} + \mu_{\tau,n+1-i} = c(\mu) \textnormal{ for some integer } c\}.
 \]
 Denote by $S_2$ the maximal split torus of $T_2$. Now
 \begin{eqnarray*}
  X^*(S_2) &\cong& \bigslant{X_*(S_1)}{\langle (e_{\QQ_p,i} + e_{\QQ_p,n+1-i}) - (e_{\QQ_p,j} + e_{\QQ_p,n+1-j}) \rangle_{i,j}} \\
  X_*(S_2) &=& \{ (\nu \in X_*(S) \mid \nu_{i} + \nu_{n+1-i} = c(\nu) \textnormal{ for some integer } c(\nu) \} \\
  R_{\QQ_p} &=& \{e_{\QQ_p,i|S_2} -e_{\QQ_p,j|S_2} \mid i \not=j \in \{1,\ldots,n/2\}\}\\
            & & \cup \{\pm(e_{\QQ_p,i|S_2} +e_{\QQ_p,j|S_2} -c) \mid  i\not=j \in \{1,\ldots,n/2\}\} \\
            & & \cup \{\pm(2e_{\QQ_p,i|S_2} -c) \mid i\in\{1,\ldots,n/2\}\} \\
            &=& \{e_{\QQ_p,i|S_2} - e_{\QQ_p,j|S_2} \mid i \not=j \in \{1,\ldots,n\}\} \\
  R_{\QQ_p}^\vee &=& \{e_{\QQ_p,i}^\vee -e_{\QQ_p,j}^\vee + e_{\QQ_p,n+1-j}^\vee - e_{\QQ_p,n+1-i}^\vee \mid i \not=j \in \{1,\ldots,n/2\}\} \\
         & & \cup \{ \pm(e_{\QQ_p,i}^\vee+e_{\QQ_p,j}^\vee - e_{\QQ_p,n+1-i}^\vee-e_{\QQ_p,n+1-j}^\vee) \mid i \not=j \in \{1,\ldots,n/2\}\} \\
         & & \cup \{ \pm e_{\QQ_p,i}^\vee\mid i \in \{1,\ldots,n/2\}
 \end{eqnarray*}
 \begin{eqnarray*}
  R_{\QQ_p}^+ &=& \{e_{\QQ_p,i|S_2} -e_{\QQ_p,j|S_2} \mid i < j \in \{1,\ldots,n/2\}\}\\
              & & \cup \{e_{\QQ_p,i|S_2} +e_{\QQ_p,j|S_2} -c \mid  i\not=j \in \{1,\ldots,n/2\}\} \\
              & & \cup \{2e_{\QQ_p,i|S_2} -c \mid i\in\{1,\ldots,n/2\}\} \\
              &=& \{e_{\QQ_p,i|S_2} - e_{\QQ_p,j|S_2} \mid i \not=j \in \{1,\ldots,n\}\} \\
  R_{\QQ_p}^{\vee,+} &=& \{e_{\QQ_p,i}^\vee -e_{\QQ_p,j}^\vee + e_{\QQ_p,n+1-j}^\vee - e_{\QQ_p,n+1-i}^\vee \mid i <j \in \{1,\ldots,n/2\}\} \\
                     & & \cup \{ \pm(e_{\QQ_p,i}^\vee+e_{\QQ_p,j}^\vee - e_{\QQ_p,n+1-i}^\vee-e_{\QQ_p,n+1-j}^\vee\mid i \not=j \in \{1,\ldots,n/2\}\} \\
                     & & \cup \{ \pm(e_{\QQ_p,i}^\vee\mid i \in \{1,\ldots,n/2\} \\
  \Delta^+ &=& \{e_{\QQ_p,i|S_2} -e_{\QQ_p,i+1|S_2} \mid i \in \{1,\ldots,n/2-1\}\} \cup \{2e_{\QQ_p,n/2|S_2} -c\} \\
           &=& \{e_{\QQ_p,i|S_2} -e_{\QQ_p,i+1|S_2} \mid i \in \{1,\ldots,n-1\}\} \\
  \Delta^{\vee,+} &=& \{e_{\QQ_p,i}^\vee -e_{\QQ_p,i+1}^\vee + e_{\QQ_p,n+1-i}^\vee - e_{\QQ_p,n-i}^\vee \mid i \in \{1,\ldots,n/2-1\}\} \cup \{e_{\QQ_p,n/2}^\vee\}.
 \end{eqnarray*}
 Thus
 \[
  X_*(S_2)_{\QQ,\dom} = \{\nu \in \QQ^n\mid \nu_1 \geq \ldots \geq \nu_n, \nu_{i} + \nu_{n+1-i} = c(\nu) \textnormal{ for some integer } c(\nu)\}
 \]

 \subsection{$\GU_{O_F,n}$} \quad \newline
 We have
 \[
  \GU_{O_F,n} \otimes O_F = \{ (g_\tau) \in \prod_{\tau\in I} GL_n \mid g_\tau J g_{\sigma_{F'}\circ\tau} = c(g) J \}
 \]
 With $I$ cyclically permuting the factors. et $T_3 \subset B_3 \subset \GU_{O_F,n}$ denote the maximal torus resp.\ the Borel subgroup of upper triangular matrices. We denote by
 \[
  c:T_3 \to \GG_m
 \]
 the similitude factor. We fix a system of representatives $I' \subset I$ of $I/\sigma_{F'}$ Now
 \begin{eqnarray*}
  X^*(T_3) &\cong& \bigslant{X_*(T_1)}{\langle (e_{\tau,i} + e_{\sigma_{F'} + \tau,n+1-i}) - ( e_{\varsigma,j} + e_{\sigma_{F'}+ \varsigma,n+1-j}) \rangle_{\tau,\varsigma ,i,j}} \\
  X_*(T_3) &=& \{ (\mu \in X_*(T) \mid \mu_{\tau,i} + \mu_{\sigma_{F'} + \tau,n+1-i} = c(\mu) \textnormal{ for some integer } c(\mu) \} \\
  R &=& \{e_{\tau,i|T_3} -e_{\tau,j|T_3} \mid \tau \in I', i \not=j \in \{1,\ldots,n\}\} \\
    &=& \{e_{\tau,i|T_3} -e_{\tau,j|T_3} \mid \tau \in I, i \not=j \in \{1,\ldots,n\}\} \\
  R^\vee &=& \{e_{\tau,i}^\vee -e_{\tau,j}^\vee + e_{\sigma_{F'}+\tau,n+1-j} - e_{\sigma_{F'}+\tau,n+1-i}  \mid \tau \in I', i \not=j \in \{1,\ldots,n\}\} \\
  R^+ &=& \{e_{\tau,i|T_3} -e_{\tau,j|T_3} \mid \tau \in I', i <j \in \{1,\ldots,n\}\} \\
      &=& \{e_{\tau,i|T_3} -e_{\tau,j|T_3} \mid \tau \in I, i <j \in \{1,\ldots,n\}\} \\
  R^{\vee,+} &=& \{e_{\tau,i}^\vee -e_{\tau,j}^\vee + e_{\sigma_{F'}+\tau,n+1-j} - e_{\sigma_{F'}+\tau,n+1-i} \mid \tau \in I, i <j \in \{1,\ldots,n\}\} \\
  \Delta^+ &=& \{e_{\tau,i|T_3} -e_{\tau,i+1|T_3} \mid \tau \in I', i \in \{1,\ldots,n-1\}\} \\
           &=& \{e_{\tau,i|T_3} -e_{\tau,i+1|T_3} \mid \tau \in I, i \in \{1,\ldots,n-1\}\} \\
  \Delta^{\vee,+} &=& \{e_{\tau,i}^\vee -e_{\tau,i+1}^\vee + e_{\sigma_{F'}+\tau,n-i} - e_{\sigma_{F'}+\tau,n+1-i} \mid \tau \in I', i \in \{1,\ldots,n-1\}\}.
 \end{eqnarray*}
 In particular,
 \[
  X_*(T_3)_{\dom} = \{ (\mu \in X_*(T) \mid \mu_{\tau,1} \geq \ldots \geq \mu_{\tau,n}, \mu_{\tau,i} + \mu_{\sigma_{F'} + \tau,n+1-i} = c(\mu) \textnormal{ for some integer } c(\mu) \}.
 \]
 We denote by $S_3$ the maximal split torus of $T_3$. If $n$ is even, then
 \begin{eqnarray*}
  X^*(S_3) &\cong& \bigslant{X_*(S_1)}{\langle (e_{\QQ_p,i} + e_{\QQ_p,n+1-i}) - (e_{\QQ_p,j} + e_{\QQ_p,n+1-j}) \rangle_{i,j}} \\
  X_*(S_3) &=& \{ (\nu \in X_*(S) \mid \nu_{i} + \nu_{n+1-i} = c(\nu) \textnormal{ for some integer } c(\nu) \} \\
  R_{\QQ_p} &=& \{e_{\QQ_p,i|S_2} -e_{\QQ_p,j|S_2} \mid i \not=j \in \{1,\ldots,n/2\}\}\\
            & & \cup \{\pm(e_{\QQ_p,i|S_2} +e_{\QQ_p,j|S_2} -c) \mid  i\not=j \in \{1,\ldots,n/2\}\} \\
            & & \cup \{\pm(2e_{\QQ_p,i|S_2} -c) \mid i\in\{1,\ldots,n/2\}\} \\
            &=& \{e_{\QQ_p,i|S_2} - e_{\QQ_p,j|S_2} \mid i \not=j \in \{1,\ldots,n\}\} \\
  R_{\QQ_p}^\vee &=& \{e_{\QQ_p,i}^\vee -e_{\QQ_p,j}^\vee + e_{\QQ_p,n+1-j}^\vee - e_{\QQ_p,n+1-i}^\vee \mid i \not=j \in \{1,\ldots,n/2\}\} \\
         & & \cup \{ \pm(e_{\QQ_p,i}^\vee+e_{\QQ_p,j}^\vee - e_{\QQ_p,n+1-i}^\vee-e_{\QQ_p,n+1-j}^\vee) \mid i \not=j \in \{1,\ldots,n/2\}\} \\
         & & \cup \{ \pm e_{\QQ_p,i}^\vee\mid i \in \{1,\ldots,n/2\} \\
  R_{\QQ_p}^+ &=& \{e_{\QQ_p,i|S_2} -e_{\QQ_p,j|S_2} \mid i < j \in \{1,\ldots,n/2\}\}\\
              & & \cup \{e_{\QQ_p,i|S_2} +e_{\QQ_p,j|S_2} -c \mid  i\not=j \in \{1,\ldots,n/2\}\} \\
              & & \cup \{2e_{\QQ_p,i|S_2} -c \mid i\in\{1,\ldots,n/2\}\} \\
              &=& \{e_{\QQ_p,i|S_2} - e_{\QQ_p,j|S_2} \mid i \not=j \in \{1,\ldots,n\}\} \\
  R_{\QQ_p}^{\vee,+} &=& \{e_{\QQ_p,i}^\vee -e_{\QQ_p,j}^\vee + e_{\QQ_p,n+1-j}^\vee - e_{\QQ_p,n+1-i}^\vee \mid i <j \in \{1,\ldots,n/2\}\} \\
                     & & \cup \{ \pm(e_{\QQ_p,i}^\vee+e_{\QQ_p,j}^\vee - e_{\QQ_p,n+1-i}^\vee-e_{\QQ_p,n+1-j}^\vee\mid i \not=j \in \{1,\ldots,n/2\}\} \\
                     & & \cup \{ \pm(e_{\QQ_p,i}^\vee\mid i \in \{1,\ldots,n/2\} \\
  \Delta^+ &=& \{e_{\QQ_p,i|S_2} -e_{\QQ_p,i+1|S_2} \mid i \in \{1,\ldots,n/2-1\}\} \cup \{2e_{\QQ_p,n/2|S_2} -c\} \\
           &=& \{e_{\QQ_p,i|S_2} -e_{\QQ_p,i+1|S_2} \mid i \in \{1,\ldots,n-1\}\} \\
  \Delta^{\vee,+} &=& \{e_{\QQ_p,i}^\vee -e_{\QQ_p,i+1}^\vee + e_{\QQ_p,n+1-i}^\vee - e_{\QQ_p,n-i}^\vee \mid i \in \{1,\ldots,n/2-1\}\} \cup \{e_{\QQ_p,n/2}^\vee\}.
 \end{eqnarray*}
 If $n$ is odd, then
 \begin{eqnarray*}
  X^*(S_3) &\cong& \bigslant{X_*(S_1)}{\langle (e_{\QQ_p,i} + e_{\QQ_p,n+1-i}) - (e_{\QQ_p,j} + e_{\QQ_p,n+1-j}) \rangle_{i,j}} \\
  X_*(S_3) &=& \{ (\nu \in X_*(S_1) \mid \nu_{i} + \nu_{n+1-i} = c(\nu) \textnormal{ for some integer } c(\nu) \} \\
  R_{\QQ_p} &=& \{e_{\QQ_p,i|S_3} -e_{\QQ_p,j|S_3} \mid i \not=j \in \{1,\ldots,(n-1)/2\}\}\\
            & & \cup \{\pm(e_{\QQ_p,i|S_3} +e_{\QQ_p,j|S_3} -c) \mid  i\not=j \in \{1,\ldots(n-1)/2\}\} \\
            & & \cup \{\pm(2e_{\QQ_p,i|S_3} -c) \mid i\in\{1,\ldots,(n-1)/2\}\} \\
            & & \cup \{\pm(e_{\QQ_p,i|S_3} -c) \mid i\in\{1,\ldots,(n-1)/2\}\} \\
            &=& \{e_{\QQ_p,i|S_2} - e_{\QQ_p,j|S_2} \mid i \not=j \in \{1,\ldots,n\}\} \\
  R_{\QQ_p}^\vee &=& \{e_{\QQ_p,i}^\vee -e_{\QQ_p,j}^\vee + e_{\QQ_p,n+1-j}^\vee - e_{\QQ_p,n+1-i}^\vee \mid i \not=j \in \{1,\ldots,(n-1)/2\}\} \\
         & & \cup \{ \pm(e_{\QQ_p,i}^\vee+e_{\QQ_p,j}^\vee - e_{\QQ_p,n+1-i}^\vee-e_{\QQ_p,n+1-j}^\vee) \mid i \not=j \in \{1,\ldots,(n-1)/2\}\} \\
         & & \cup \{ \pm e_{\QQ_p,i}^\vee\mid i \in \{1,\ldots,(n-1)/2\} \\
         & & \cup \{ \pm 2e_{\QQ_p,i}^\vee\mid i \in \{1,\ldots,(n-1)/2\} \\
  R_{\QQ_p}^+ &=& \{e_{\QQ_p,i|S_3} -e_{\QQ_p,j|S_3} \mid i < j \in \{1,\ldots,(n-1)/2\}\}\\
              & & \cup \{e_{\QQ_p,i|S_3} +e_{\QQ_p,j|S_3} -c \mid  i\not=j \in \{1,\ldots,(n-1)/2\}\} \\
              & & \cup \{2e_{\QQ_p,i|S_3} -c \mid i\in\{1,\ldots,(n-1)/2\}\} \\
              & & \cup \{e_{\QQ_p,i|S_3} \mid i\in\{1,\ldots,(n-1)/2\}\} \\
              &=& \{e_{\QQ_p,i|S_2} - e_{\QQ_p,j|S_2} \mid i \not=j \in \{1,\ldots,n\}\} \\
 \end{eqnarray*}
 \begin{eqnarray*}
  R_{\QQ_p}^{\vee,+} &=& \{e_{\QQ_p,i}^\vee -e_{\QQ_p,j}^\vee + e_{\QQ_p,n+1-j}^\vee - e_{\QQ_p,n+1-i}^\vee \mid i <j \in \{1,\ldots,(n-1)/2\}\} \\
                     & & \cup \{ e_{\QQ_p,i}^\vee+e_{\QQ_p,j}^\vee - e_{\QQ_p,n+1-i}^\vee-e_{\QQ_p,n+1-j}^\vee\mid i \not=j \in \{1,\ldots,(n-1)/2\}\} \\
                     & & \cup \{ e_{\QQ_p,i}^\vee\mid i \in \{1,\ldots,(n-1)/2\}
                     \cup \{ 2e_{\QQ_p,i}^\vee\mid i \in \{1,\ldots,(n-1)/2\} \\
  \Delta^+ &=& \{e_{\QQ_p,i|S_3} -e_{\QQ_p,i+1|S_3} \mid i \in \{1,\ldots,(n-1)/2-1\}\} \cup \{e_{\QQ_p,(n-1)/2|S_2} -c\} \\
           &=& \{e_{\QQ_p,i|S_3} -e_{\QQ_p,i+1|S_3} \mid i \in \{1,\ldots,n-1\}\} \\
  \Delta^{\vee,+} &=& \{e_{\QQ_p,i}^\vee -e_{\QQ_p,i+1}^\vee + e_{\QQ_p,n+1-i}^\vee - e_{\QQ_p,n-i}^\vee \mid i \in \{1,\ldots,(n-1)/2-1\}\} \cup \{e_{\QQ_p,n/2}^\vee\}.
 \end{eqnarray*}
 In any case we get
 \[
  X_*(S_3)_{\QQ,\dom} = \{\nu \in \QQ^n\mid \nu_1 \geq \ldots \geq \nu_n, \nu_{i} + \nu_{n+1-i} = c(\nu) \textnormal{ for some integer } c(\nu)\}
 \]  
 
 \providecommand{\bysame}{\leavevmode\hbox to3em{\hrulefill}\thinspace}
\providecommand{\MR}{\relax\ifhmode\unskip\space\fi MR }
\providecommand{\MRhref}[2]{%
  \href{http://www.ams.org/mathscinet-getitem?mr=#1}{#2}
}
\providecommand{\href}[2]{#2}

\end{document}